\documentclass[11pt,a4paper]{article}
\usepackage{amsmath}
\usepackage{amssymb}
\usepackage{amsthm}
\usepackage{amsfonts}
\usepackage{latexsym}
\usepackage{verbatim} 
\usepackage{xcolor}
\usepackage[a4paper,width=16cm,top=2cm,bottom=2cm,footskip=1cm
]{geometry}

\usepackage[flushmargin]{footmisc}
\usepackage{color}                    
\usepackage{esvect}
\usepackage{mathrsfs}

\theoremstyle{plain}
\newtheorem{thm}{Theorem}[section]

\newtheorem{lem}{Lemma}[section]
\newtheorem{prop}{Proposition}[section]

\newtheorem{ass}{Assumption}[section]

\theoremstyle{remark}

\theoremstyle{definition}
\newtheorem{defn}{Definition}[section]

\newtheorem{rem}{Remark}[section]

\newcommand{\Complex}{\mathbb C}
\newcommand{\Real}{\mathbb R}
\newcommand{\N}{\mathbb N}
\newcommand{\ddbar}{\overline\partial}
\newcommand{\pr}{\partial}
\newcommand{\ol}{\overline}
\newcommand{\Td}{\widetilde}
\newcommand{\norm}[1]{\left\Vert#1\right\Vert}
\newcommand{\abs}[1]{\left\vert#1\right\vert}
\newcommand{\set}[1]{\left\{#1\right\}}
\newcommand{\To}{\rightarrow}

\title{Quantization and reduction for torsion free CR manifolds}
\author{
	Andrea Galasso\footnote{\noindent{\bf Address:} Dipartimento di Matematica e Applicazioni, Universit\`a degli Studi di Milano-Bicocca, Via R.	Cozzi 55, 20125 Milano, Italy; \\ {\bf ORCID iD:} 0000-0002-5792-1674; {\bf e-mail}: andrea.galasso@unimib.it andrea.galasso.91@gmail.com \\  {\bf Scholarship: } Titolare di una borsa per l'estero dell'Istituto Nazionale di Alta Matematica} \text{ and }Chin-Yu Hsiao\footnote{\noindent{\bf Address:} Department of Mathematics, National Taiwan University, {\bf ORCHID iD:} 0000-0002-1781-0013; {\bf email}: chinyuhsiao@ntu.edu.tw; chinyu.hsiao@gmail.com}
}

\date{}

\begin{document}
\maketitle

\begin{abstract}  Consider a compact torsion free CR manifold $X$ and assume that $X$ admits a compact CR Lie group action $G$. Let $L$ be a $G$-equivariant rigid CR line bundle over $X$. It seems natural to consider the space of $G$-invariant CR sections in the high tensor powers as quantization space, on which a certain weighted $G$-invariant Fourier–Szeg\H{o} operator projects. Under certain natural assumptions, we show that the group invariant Fourier–Szeg\H{o} projector admits a full asymptotic expansion. As an application, if the tensor power of the line bundle is large enough, we prove that quantization commutes with reduction.
\end{abstract}
\tableofcontents


\section{Introduction}\label{s-gue241018yydI}

The study of quantization commutes with reduction on various different geometric situation plays an important role in analysis, geometry and Mathematical physics.
The famous geometric quantization conjecture of 
Guillemin and Sternberg \cite{GS:82} states
that for a compact pre-quantizable symplectic manifold admitting 
a Hamiltonian action of a compact connected Lie group, the principle of 
``quantization commutes with reduction" holds.
This conjecture was first proved independently by Meinrenken \cite{M96} 
and Vergne \cite{Ver:96} for
the case where the Lie group is abelian, and  by Meinrenken \cite{M98}
 in the general case,
 then Tian-Zhang~\cite{TZ98} gave a purely analytic proof in general case
 with various generalizations.
In the case of a non-compact symplectic manifold $M$ 
with a compact connected Lie group action $G$, 
this question was solved by the fundamental paper of Ma-Zhang~\cite{MZ09,MZI}
as a solution to a conjecture of Vergne in her ICM 2006 plenary 
lecture \cite{Ve07}, see \cite{Ma10} for a survey.  

The study of quantization on CR, contact and Sasakian manifolds is closely related to many important geometric and analytic problems in CR and contact geometry (see,~\cite{hsiaohuang},~\cite{hmm},~\cite{f}). In~\cite{hmm}, we established geometric quantization on CR manifolds. 
Let $(X,T^{1,0}X)$ be a compact CR manifold with Reeb one form $\omega_0$. Assume that $X$ admits a CR compact Lie group action $G$ and the Lie group action $G$ preserves the Reeb one form $\omega_0$. The one form $\omega_0$ induces a CR moment map $\mu$. Under the assumptions that zero is a regular value of $\mu$ and $X$ is strongly pseudoconvex near $\mu^{-1}(0)$, it was shown in~\cite{hmm} that up to some finite dimensional subspaces of $L^2$ $G$-invariant CR functions and $L^2$ CR functions on the reduced space, quantization commutes with reduction. In~\cite{hmm}, they assumed that zero is a regular value of $\mu$. It is a natural question that if zero is not a regular value of $\mu$, can we still have geometric quantization?  A 
 special but important case is that if the action $G$ is horizontal in $X$, then zero is not a regular value of $\mu$. Let's see a simple example and explain our motivation.  Consider $\hat X:=M\times X$, where $M$ is a complex manifold. 
The Reeb one form on $X$ can be lifted to a Reeb one form on $\hat X$. 
If $G$ acts only in $M$, then $G$ is horizontal on $\hat X$. We observe that if $M$ admits a positive line bundle, we can study geometric quantization by using the curvature of the line bundle and the Reeb one form on $X$. Therefore, even zero is not a regular value of the moment map induced by the one form on $X$, we can still study geometric quantization for CR manifolds by using the curvature of the CR line bundle and the Reeb one form on the base manifold $X$. Thus, it is very natural to study geometric quantization for CR manifolds by using the curvature of the CR line bundle and the Reeb one form on the base manifold $X$. This is the stating point of this work. In this work, we consider a compact torsion free CR manifold with a CR compact Lie group action $G$. Let $L\To X$ be a $G$-equivariant rigid CR line bundle. We consider the space of $G$-invariant CR sections in the high tensor powers as quantization space, on which a certain weighted $G$-invariant Fourier–Szeg\H{o} operator projects.  Under certain natural assumptions of the curvature of the CR line bundle, the Reeb one form on $X$ and the Lie group action $G$,  we show that the group invariant  Fourier–Szeg\H{o} projector admits a full asymptotic expansion and if the tensor power of the line bundle is large enough, we prove that quantization commutes with reduction. 

We now formulate our results. We refer the reader to Section~\ref{s-gue241015yyd} for the terminology and notations used here. 
Let $(X,T^{1,0}X)$ be a compact orientable CR manifold of dimension $2n+1,\, n\geq 1$ with a transversal and CR Reeb vector field $T\in\mathcal{C}^\infty(X,TX)$ (see \eqref{e-gue240714yyd}). Let 
$\eta: \mathbb R\times X\To X$, $(\eta,x)\To\eta\cdot x$, be the $\mathbb R$-action induced by the flow of $T$ (see \eqref{e-gue240714yydI}) and let $\omega_0\in\mathcal{C}^\infty(X,T^*X)$ be the Reeb one form given by \eqref{e-gue240718yyd} below. Let $HX:={\rm Re\,}T^{1,0}X$ and let $J: HX\To HX$ be the complex structure map given by $J(u+\ol u)=iu-i\ol u$, $u\in T^{1,0}X$. Assume that $X$ admits an action of a compact Lie group $G$ of dimension $d$. We assume that 

\begin{ass}[Group action assumption]\label{ass:2}
	$G$ commutes with the $\mathbb R$-action $\eta$, preserves the CR structure, $g^\ast\omega_0=\omega_0$ on $X$ and $g_\ast J=Jg_\ast$ on $HX$, for every $g\in G$, where $g^*$ and $g_*$ denote the pull-back map and push-forward map of $G$, respectively. 
\end{ass}

Let $(L,h^L)\To X$ be a $G$-equivariant rigid CR line bundle (see Definition~\ref{Def:RigidCVB}, Definition~\ref{Def:RigidCVBz}), where $h^L$ is a $G\times\mathbb R$-invariant Hermitian metric on $L$ (see Definition~\ref{d-gue240715yyd}). 
Let $R^L$ be the curvature of $L$ induced by $h^L$ (see Definition~\ref{d-gue150808g}). Let $\omega_0\in\mathcal{C}^\infty(X,T^*X)$ be the Reeb one form given by \eqref{e-gue240718yyd} and let $\mathcal{L}_x$ be the Levi form of $X$ at $x\in X$ given by \eqref{e-gue240718yydI}. In this work, we assume that 

\begin{ass}[Curvature assumption] \label{ass:1}
There exists a bounded open interval $I\subset \mathbb{R}$ such that $R^L_x-2s\mathcal{L}_x$ is positive definite on $T^{1,0}_xX$ at every $x\in X$, for every $s\in I$.
\end{ass}

From now on, we assume that the Hermitian metrices $\langle\,\cdot\,|\,\cdot\,\rangle$ on $\mathbb CTX$ and $h^L$ are $G\times\mathbb R$-invariant. Note that $\langle\,\cdot\,|\,\cdot\,\rangle$ 
satisfies the following: $T^{1,0}X$ is orthogonal to $T^{0,1}X$, $\langle\, u \,|\, v \,\rangle$ is real if $u, v$ are real tangent vectors, $\langle\,T\,|\,T\,\rangle=1$ and $T$ is orthogonal to $T^{1,0}X\oplus T^{0,1}X$.

For every $\xi \in \mathfrak{g}$, we write $\xi_X$ to denote the infinitesimal vector field on $X$ induced by $\xi$. Put 
\[\underline{\mathfrak{g}}:=\set{\xi_X\in\mathcal{C}^\infty(X,TX);\, \xi\in\mathfrak{g}}.\] Let 
\[\gamma: X\To\mathfrak{g}^*\]
be the moment map induced by $h^L$ (see Definition~\ref{d-gue240719yyd} and Lemma~\ref{l-gue240719ycd}). Let 
\[\mu: X\To\mathfrak{g}^*\]
be the moment map induced by $\omega_0$ (see Definition~\ref{d-gue240921yyd}). 
For every $t\in  I$, let 
 \begin{equation}\label{e-gue240921yyd}
\hat\mu_t:=\gamma-2t\mu: X\To\mathfrak{g}^*.
 \end{equation}
 Recall that $I$ is the open bounded interval as in Assumption~\ref{ass:1}.
 In this work, we assume that

\begin{ass}\label{a-gue240720yyd}
$\hat\mu^{-1}_t(0)=\gamma^{-1}(0)\cap\mu^{-1}(0)$, zero is a regular value of $\hat\mu_t$, for all $t\in I$ and the action $G$ is free near $\hat\mu^{-1}_t(0)$. 
\end{ass} 

\begin{rem}\label{r-gue241015yyd}

(i) If the action $G$ is horizontal, that is, $\omega_0(\xi_X)=0$, for every $\xi_X\in\underline{\mathfrak{g}}$, then $\hat\mu=\gamma$ and $\hat\mu^{-1}_t(0)$ is independent of $t\in I$. 

(ii) If $\gamma=\omega_0$, then $\hat\mu^{-1}_t(0)$ is independent of $t\in I$. 

We refer the reader to Section~\ref{s-gue241018yyd} for more examples. 
 \end{rem}

 We now introduce our result about $G$-invariant weighted Fourier-Szeg\H{o} projection. Let $(\,\cdot\,|\,\cdot\,)_k$ be the $L^2$ inner product on $\mathcal{C}^\infty(X,L^k)$ induced by $\langle\,\cdot\,|\,\cdot\,\rangle$ and $h^{L^k}$. Let 
 \[\mathcal{H}^0_b(X,L^k)^G:=\set{u\in L^2(X,L^k);\, \ddbar_bu=0, g^*u=u,\  \ \forall g\in G},\]
 where $\ddbar_b$ is the tangential Cauchy-Riemann operator with values in $L^k$. Let $\Pi^G_k: L^2(X,L^k)\To\mathcal{H}^0_b(X,L^k)^G$ be the orthogonal projection ($G$-invariant Szeg\H{o} projection). 
 We extend $-iT$ to $L^2$ space by
\[\begin{split}
&-iT: {\rm Dom\,}(-iT)\subset L^2(X,L^k)\To L^2(X,L^k),\\
&{\rm Dom\,}(-iT)=\set{u\in L^2(X,L^k);\, -iTu\in L^2(X,L^k)}.
\end{split}\]
From~\cite[Theorems 4.1, 4.5]{hhl},  $-iT$ is self-adjoint with respect to $(\,\cdot\,|\,\cdot\,)_k$, ${\rm Spec\,}(-iT)$ is countable
and every element in ${\rm Spec\,}(-iT)$ is an eigenvalue of $-iT$, where ${\rm Spec\,}(-iT)$ denotes the spectrum of $-iT$. Let $\tau\in\mathcal{C}^\infty_c(I,\mathbb R_+)$, $\tau_k(t):=\tau(\frac{t}{k})$. Let $\tau_k(-iT)$ be the functional calculus of $-iT$ with respect to $\tau_k$. Since $T$ preserves CR structure, commutes with the action $G$ and $L$ is rigid, $\tau_k(-iT)$ commutes with the $G$-invariant Szeg\H{o} projection $\Pi^G_k$. 
Let 
\begin{equation}\label{e-gue241015ycdy}
P^G_{k,\tau^2}:=\Pi^G_k\circ\tau^2_k(-iT): L^2(X,L^k)\To\mathcal{H}^0_b(X,L^k)^G. 
\end{equation}

Let $s$ be a local $G\times\mathbb R$-invariant CR trivializing section defined on an open set $D\subset X$, $\abs{s}^2_{h^L}=e^{-2\Phi}$. The localized operator of $P^G_{k,\tau^2}$ is given by 
\begin{equation}\label{e-gue241015ycdt}
P^G_{k,\tau^2,s}:=s^{-k}e^{-k\Phi}P^G_{k,\tau^2}s^ke^{k\Phi}: \mathcal{C}^\infty_c(D)\To\mathcal{C}^\infty(D).
\end{equation}
Let $P^G_{k,\tau^2,s}(x,y)\in\mathcal{C}^\infty(D\times D)$ be the distribution kernel of $P^G_{k,\tau^2,s}$.
Let $Y:=\hat\mu^{-1}_t(0)$. The first main result of this work is the following 

 \begin{thm}[Semi-classical $G$-invariant Fourier Szeg\H{o} kernel] \label{thm:Gszego} 
 With the notations and assumptions above, let $\chi\in\mathcal{C}^\infty(X)$ with ${\rm supp\,}\chi\cap Y=\emptyset$. Then, 
 \begin{equation}\label{e-gue240903yydII}
 \chi P^{G}_{k,\tau^2}=O(k^{-\infty})\ \ \mbox{on $X$}.
 \end{equation}

 Let $p\in Y$ and let $s$ be a local $G\times\mathbb R$-invariant CR trivializing section defined on an open set $D\subset X$, $p\in D$, $\abs{s}^2_{h^L}=e^{-2\Phi}$. Then
\begin{equation}\label{e-gue240903yydI}
P_{k,\tau^2,s}^G(x,y)=\int_{\mathbb{R}} e^{ikA(x,y,t)}g(x,y,t,k)dt+O(k^{-\infty})
\end{equation}
on $D\times D$, where 
\begin{equation}\label{e-gue241017yydI}
\begin{split}
&g\in S^{n+1-d/2}_{{\rm loc\,}}(1;
D\times D\times I),\\
&{\rm supp\,}_tg(x,y,t,k)\subset I,
\end{split}
\end{equation}
is a symbol with expansion
\begin{equation}\label{e-gue241017yydII}
\begin{split}
&g(x,y,t,k)\sim \sum_{j=0}^{+\infty} g_j(x,y,t)k^{n+1-d/2-j} \quad\text{ in }S^{n+1-d/2}_{\mathrm{loc}}(1;D\times D\times I),\\
&g_j(x,y,t)\in\mathcal{C}^\infty(D\times D\times I),\ \ j=0,1,\ldots,\\
&{\rm supp\,}_tg_j(x,y,t)\subset I,\ \ j=0,1,\ldots.
\end{split}\end{equation}

Furthermore $A\in\mathcal{C}^{\infty}(D\times D\times I)$ is a complex phase function with ${\rm Im\,}A\geq0$, and  
	\begin{equation}\label{e-gue241017yyda}
		\mathrm{d}_x A(x,\,y,\,t)_{\vert x=y}=-\mathrm{d}_y A(x,\,y,\,t)_{\vert x=y}= -2\,\mathrm{Im}\,\overline{\partial}_b\Phi(x)+t\,\omega_0
	\end{equation}
for every $x\in Y$, $A(x,\,y,\,t)=0$ if and only if $x=y\in Y$ and 
\begin{equation}\label{e-gue241017yydb}
{\rm Im\,}A(x,y,t)\geq C\Bigr( d^2(x,Y)+d^2(x,Y)\Bigr),
\end{equation}
$x, y\in D$, where $C>0$ is a constant.  
For a local description of the phase $A$ in terms of local coordinates defined in Proposition \ref{prop:coordinates}, we refer to equation \eqref{eq:phaseA} and we refer the reader to Theorem~\ref{t-gue170105I} for more properties of ${\rm Im\,}A$.
\end{thm}

We refer the reader to Section~\ref{s-gue170111w} and the discussion after \eqref{sk} for the semi-classcial notations used in Theorem~\ref{thm:Gszego}.

We now give a formula for the leading term of $g(x,y,t,k)$ in \eqref{e-gue241017yydII}. We need to recall one more piece of notation. Fix $Y$, for every $t\in I$, consider the linear map
\begin{equation}\label{e-gue240913ycd}
\begin{split}
R_x(t)\,:\, &\underline{\mathfrak{g}}_x \rightarrow \underline{\mathfrak{g}}_x\,\\
	&u \mapsto R_x(t)u,
\end{split}
\end{equation}
where 
\[\langle R_x(t)u \vert v\rangle = \langle -i(R^L(x)-2t\mathcal{L}_x),\,Ju\wedge v\rangle\,. \]
Let $\det R_x(t)=\mu_1(x,t)\cdots \mu_d(x,t)$, where $\mu_j(x,t)$, $j=1,\,\dots, d$ are the eigenvalues of $R_x(t)$. Furthermore, put $Y_x=\{g\cdot x;\,g\in G\}$. Then $Y_x$ is a $d$-dimensional submanifold of $X$. The $G$-invariant Hermitian metric induces a volume form $\mathrm{dV}_{Y_x}$ on $Y_x$. Put
\[V_{\mathrm{eff}}(x):=\int_{Y_x} \mathrm{dV}_{Y_x}. \]

\begin{thm} \label{thm:2}
	In the same setting of Theorem \ref{thm:Gszego}, for any $x \in Y\cap D$ and $t\in I$, we have In the same setting of Theorem \ref{thm:Gszego}, for any $x \in Y\cap D$ and $t\in I$, we have
	\begin{equation}
 \label{e-gue241007yyd}
	    g_0(x,x,t)=2^{-n-1+d}\,\frac{1}{V_{\mathrm{eff}}(x)}\,\lvert \det R_x(t)\rvert^{-1/2} \pi^{-n-1+d/2}\,\lvert \det(R^L_x-2t\mathcal{L}_x)\rvert\cdot\tau^2(t),\end{equation}
     where $\det(R^L_x-2t\mathcal{L}_x)=\lambda_1(x,t)\cdots\lambda_n(x,t)$, $\lambda_j(x,t)$, $j=1,\ldots,n$, are the eigenvalues of $R^L_x-2t\mathcal{L}_x$ with respect to $\langle\,\cdot\,|\,\cdot\,\rangle$.
\end{thm}

Let $X_G:=Y/G$. By Assumptions \ref{a-gue240720yyd}, we have (see Proposition~\ref{p-gue240808yyd})

\begin{thm}\label{t-gue241015ycda}
$X_G$ is a torsion free CR manifold of dimension $2n-2d+1$ and $L_G:=L/G$ is a rigid CR line bundle over $X_G$. 
\end{thm}

Theorem~\ref{t-gue241015ycda} says that  $X_G:=Y/G$ is endowed by a CR structure in a natural way. We say that $X_G$ is the CR reduction with respect to curvature data. From Theorem~\ref{t-gue241015ycda}, we can study quantization commutes with reduction. 

For every $\lambda\in{\rm Spec\,}(-iT)$, put 
\begin{equation}\label{e-gue241017ycdf}
    \mathcal{H}^0_{b,\lambda}(X,L^k)^G:=\set{u\in\mathcal{H}^0_{b}(X,L^k)^G;\,-iTu=\lambda u}.\end{equation}
    It is not difficult to see that ${\rm dim\,}\mathcal{H}^0_{b,\lambda}(X,L^k)^G<+\infty$. 
   Let  $T_{X_G}$ be the vector field on $X_G$ induced by the $\mathbb R$-action on $X_G$. For every $\lambda\in{\rm Spec\,}(-iT_{X_G})$, put 
\begin{equation}\label{e-gue241017ycdg}
    \mathcal{H}^0_{b,\lambda}(X_G,L^k_G):=\set{u\in\mathcal{H}^0_{b}(X_G,L^k_G);\,-iT_{X_G}u=\lambda u}.\end{equation}
    The following is the  quantization commutes with reduction result obtained in this work 

    \begin{thm}[Quantization commutes with reduction] \label{thm:quantandred}
	With the same notations and assumptions used above,  suppose that $I=(a,b)$, $a<b<+\infty$. There is a $k_0\in N$ such that for all $\lambda\in(ka,kb)$, $k\geq k_0$, $\lambda\in{\rm Spec\,}(-iT)\cap{\rm Spec\,}(-iT_{X_G})$,we have 

\begin{equation}\label{e-gue241017ycdu}
{\rm dim\,}\mathcal{H}^0_{b,\lambda}(X,L^k)^G={\rm dim\,}\mathcal{H}^0_{b,\lambda}(X_G,L^k_G).
\end{equation}

Moreover, if $\lambda\in{\rm Spec\,}(-iT)$ and $\lambda\notin{\rm Spec\,}(-iT_{X_G})$, $\lambda\in(ka,kb)$, $k\geq k_0$, we have that ${\rm dim\,}\mathcal{H}^0_{b,\lambda}(X,L^k)^G=0$. Similarly, if $\lambda\in{\rm Spec\,}(-iT_{X_G})$ and $\lambda\notin{\rm Spec\,}(-iT)$,$\lambda\in(ka,kb)$, $k\geq k_0$, then ${\rm dim\,}\mathcal{H}^0_{b,\lambda}(X_G,L^k_G)=0$.
\end{thm}

Theorem~\ref{thm:quantandred} can be seen as an application of  Theorems \ref{thm:Gszego} and \ref{thm:2}. 
The idea of the proof of Theorem~\ref{thm:quantandred} comes from~\cite{hsiaohuang},~\cite{hmm}. 
 Furthermore, we recall that the way to establish the isometry from kernel expansion for $k$ large comes from \cite{mz}.


\subsection{Examples}\label{s-gue241018yyd} 

In this subsection, we will give some very simple examples that Assumption~\ref{a-gue240720yyd} holds. 

Let $M$ be a compact orbifold with cyclic singularities and let $(L,h^L)\To M$ be an orbifold line bundle, where $h^L$ is a Hermitian metric of $L$. Assume that the curvature of $L$ induced by $h^L$ is positive definite. Suppose that $M$ admits a compact holomorphic Lie group action $G$ and the action $G$ can be lifted to $L$. Assume further that $L$ and $h^L$ are $G$-invariant. Consider the circle bundle 
\[X:=\set{v\in L^*;\, |v|_{h^{L^*}}=1}.\]
Since the singularities of $M$ are cyclic, $X$ is a smooth CR manifold. Actually, $X$ is a quasi-regular Sasakian manifold. The line bundle $L$ can be considered as a CR line bundle over $X$ (we still denote by $L$).
$X$ is a torsion free CR manifold. We will use the same notations as in Section~\ref{s-gue241018yydI}. The $\mathbb R$-action on $X$ is the $S^1$-action on $X$ acting on the fiber of $X$. Take any $G$-invariant volume form $dV_M$ on $M$. For every $k, m\in\mathbb Z$, let 
$\set{f_1,\ldots,f_{d_{k,m}}}$ be an orthonormal basis for $\mathcal{H}^0(M,L^k\otimes L^m)^G$ with respect to the $L^2$ inner product induced by $h^L$ and $dV_M$, 
where $\mathcal{H}^0(M,L^k\otimes L^m)^G$ denotes the space of all $G$-invariant holomorphic sections of $M$ with values in $L^k\otimes L^m$. 
Let 
\[B^G_{k,m}(x):=\sum^{d_{k,m}}|f_j(x)|^2_{h^{L^k}}\in\mathcal{C}^\infty(M).\]
It is straightforward to check that
\begin{equation}\label{e-gue241018yyda}
P^G_{k,\tau^2,s}(x,x)=\sum_{m\in\mathbb Z}\tau^2(\frac{m}{k})B^G_{k,m}(\pi(x)),
\end{equation}
where $\pi: X\To M$ is the natural projection and $P^G_{k,\tau^2,s}(x,y)$ is as in \eqref{e-gue241015ycdt}. 

In this circle bundle case, we can check that $\gamma=c_0\mu$, for some constant $c_0\neq0$, where $\gamma$ and $\mu$ are as in the discussion before \eqref{e-gue240921yyd}.  From this observation, we see that $\hat\mu^{-1}_t(0)$ is independent of $t\in I$. It should be noticed that since $R^L$ is positive, we can take $I$ to be any small open interval of $0\in\mathbb R$ and hence Assumption~\ref{ass:1} holds. 
From Theorem~\ref{thm:Gszego}, \eqref{e-gue241018yyda}, we deduce that if zero is a regular value of $\hat\mu_t$, for some $t\in I$ and the action $G$ is free near $\hat\mu^{-1}_t(0)$, then 
\[\sum_{m\in\mathbb Z}\tau^2(\frac{m}{k})B^G_{k,m}(\pi(x))\] 
admits a full asymptotic expansion.

We now give another simple example. Let 
\[X:=\set{z\in\mathbb C^{n+1};\, \abs{z_1}^{2\alpha_1}+(\abs{z_2}^{2\alpha_2}\cdots+\abs{z_{n+1}}^{2\alpha_{n+1}})^{m}=1},\]
where $\alpha_1,\ldots,\alpha_{n+1},m\in\mathbb N$. Then $X$ is a compact CR manifold. $X$ admits a transversal and CR $\mathbb R$-action:
\[
\eta\cdot(z_1,\ldots,z_{n+1})=(e^{i\beta_1\eta}z_1,\ldots,e^{i\beta_{n+1}\eta}z_{n+1}),\]
where $(\beta_1,\ldots,\beta_{n+1})\in\mathbb R^{n+1}_+$.
Then, $X$ is a torsion free CR manifold. Let $L$ be the trivial line bundle with non-trivial Hermitian metric $\abs{1}^2_{h^L}=e^{-2\abs{z}^2}$. The CR manifold $X$ admits a $S^1$-action: 
\[G=S^1: e^{i\theta}\cdot z=(e^{-i\theta}z_1,e^{i\theta}z_2,\ldots,e^{i\theta}z_{n+1}).\]
We can calculate
\begin{equation}\label{e-gue241018ycdk}
\begin{split}
&\hat\mu^{-1}_t
(\frac{\pr}{\pr\theta})\\
&=\Bigr(\abs{z_1}^2-\sum^{n+1}_{j=2}\abs{z_j}^2\Bigr)-2t\Bigr(\alpha_1
\abs{z_1}^{2\alpha_1}-m(\sum^{n+1}_{j=2}
\abs{z_j}^{2\alpha_j})^{m-1}(\sum^{n+1}_{j=2}\alpha_j\abs{z_j}^{2\alpha_j})\Bigr).
\end{split}
\end{equation}
From \eqref{e-gue241018ycdk}, we see that if $\alpha_1=m$, $\alpha_2=\cdots=\alpha_{n+1}=1$, then $\hat\mu^{-1}_t(0)=\mu^{-1}(0)=\gamma^{-1}(0)$. Moreover, we can check that zero is a regular value of $\hat\mu_t$ at $t=0$ and we can take $I$ to be any small open interval of $0\in\mathbb R$ such that  Assumption~\ref{ass:1} holds. 

\section{Preliminaries}\label{s-gue241015yyd}

\subsection{Standard notations} \label{s-ssna}

We use the following notations: $\mathbb N=\{1,2,\ldots\}$ is the set of natural numbers excluding $0$ and $\mathbb N_0=\mathbb N\cup\{0\}$, $\mathbb R$ is the set of real numbers, ${\mathbb R}_+=\{x\in\mathbb R;\, x>0\}$, $\overline{\mathbb R}_+=\{x\in\mathbb R;\, x\geq0\}$. Furthermore we adopt the standard multi-index notation: we write $\alpha=(\alpha_1,\ldots,\alpha_n)\in\mathbb N^n_0$ 
if $\alpha_j\in\mathbb N_0$, $j=1,\ldots,n$. 

Let $M$ be a smooth paracompact manifold. We let $TM$ and $T^*M$ denote respectively the tangent bundle of $M$ and the cotangent bundle of $M$. The complexified tangent bundle $TM \otimes \mathbb{C}$ of $M$ will be denoted by $\Complex TM$, similarly we write $\Complex T^*M$ for the complexified cotangent bundle of $M$. Consider $\langle\,\cdot\,,\cdot\,\rangle$ to denote the pointwise
duality between $TM$ and $T^*M$; we extend $\langle\,\cdot\,,\cdot\,\rangle$ bi-linearly to $\Complex TM\times\Complex T^*M$. Let $B$ be a smooth vector bundle over $M$. The fiber of $B$ at $x\in M$ will be denoted by $B_x$. Let $E$ be a vector bundle over a smooth paracompact manifold $N$. We write
$B\boxtimes E^*$ to denote the vector bundle over $M\times N$ with fiber over $(x, y)\in M\times N$ consisting of the linear maps from $E_y$ to $B_x$.  

Let $Y\subset M$ be an open set. From now on, the spaces of distribution sections of $B$ over $Y$ and smooth sections of $B$ over $Y$ will be denoted by $\mathcal D'(Y, B)$ and $\mathcal{C}^\infty(Y, B)$, respectively.
Let $\mathcal E'(Y, B)$ be the subspace of $\mathcal D'(Y, B)$ whose elements have compact support in $Y$. Let $\mathcal{C}^\infty_c(Y,B):=\mathcal{C}^\infty(Y,B)\cap\mathcal{E}'(Y,B)$. 
For $m\in\Real$, let $H^m(Y, B)$ denote the Sobolev space
of order $m$ of sections of $B$ over $Y$. Let us denote
\begin{align*}
H^m_{\rm loc\,}(Y, B)=\big\{u\in\mathcal{D}'(Y, B);\, \varphi u\in H^m(Y, B),
    \, \forall\varphi\in \mathcal{C}^\infty_c(Y)\big\}\,,
\end{align*}
and
\begin{align*}
       H^m_{\rm comp\,}(Y, B)=H^m_{\rm loc}(Y, B)\cap\mathcal{E}'(Y, B)\,.
\end{align*}

Let $B$ and $E$ be smooth vector
bundles over paracompact orientable manifolds $M$ and $M_1$, respectively, equipped with smooth densities of integration. If
$A: \mathcal{C}^\infty_c(N,E)\To\mathcal D'(M,B)$
is continuous, we write $A(x, y)$ to denote the distribution kernel of $A$.
The following two statements are equivalent
\begin{enumerate}
	\item $A$ is continuous: $\mathcal E'(N,E)\To\mathcal{C}^\infty(M,B)$,
	\item $A(x,y)\in\mathcal{C}^\infty(M\times N,B\boxtimes E^*)$.
\end{enumerate}
If $A$ satisfies (1) or (2), we say that $A$ is smoothing on $M \times N$. 
We say that $A$ is properly supported if the restrictions of the two projections 
$(x,y)\mapsto x$, $(x,y)\mapsto y$ to ${\rm supp\,}(A(x,y))$
are proper.

Let $H(x,y)\in\mathcal D'(M\times N,B\boxtimes E^*)$. We write $H$ to denote the unique continuous operator $\mathcal C^\infty_c(N,E)\To\mathcal D'(M,B)$ with distribution kernel $H(x,y)$. In this work, we identify $H$ with $H(x,y)$. 

\subsection{Some standard notations in semi-classical analysis}\label{s-gue170111w}

Let $W_1$ be an open set in $\Real^{N_1}$ and let $W_2$ be an open set in $\Real^{N_2}$. Let $E$ and $F$ be vector bundles over $W_1$ and $W_2$, respectively. 
A $k$-dependent continuous operator
$A_k: \mathcal{C}^\infty_c(W_2,F)\To\mathcal{D}'(W_1,E)$ is called $k$-negligible on $W_1\times W_2$
if, for $k$ large enough, $A_k$ is smoothing and, for any $K\Subset W_1\times W_2$, any
multi-indices $\alpha$, $\beta$ and any $N\in\mathbb N$, there exists $C_{K,\alpha,\beta,N}>0$
such that
\[
\abs{\pr^\alpha_x\pr^\beta_yA_k(x, y)}\leq C_{K,\alpha,\beta,N}k^{-N}\:\: \text{on $K$},\ \ \forall k\gg1.
\]
In that case we write
\[A_k(x,y)=O(k^{-\infty})\:\:\text{on $W_1\times W_2$,} \quad
\text{or} \quad
A_k=O(k^{-\infty})\:\:\text{on $W_1\times W_2$.}\]
If $A_k, B_k: \mathcal{C}^\infty_c(W_2, F)\To\mathcal{D}'(W_1, E)$ are $k$-dependent continuous operators,
we write $A_k= B_k+O(k^{-\infty})$ on $W_1\times W_2$ or $A_k(x,y)=B_k(x,y)+O(k^{-\infty})$ on $W_1\times W_2$ if $A_k-B_k=O(k^{-\infty})$ on $W_1\times W_2$. 
When $W=W_1=W_2$, we sometime write ``on $W$".

Let $X$ and $M$ be smooth manifolds and let $E$ and $F$ be vector bundles over $X$ and $M$, respectively. Let $A_k, B_k: \mathcal{C}^\infty(M,F)\To\mathcal{C}^\infty(X,E)$ be $k$-dependent smoothing operators. We write $A_k=B_k+O(k^{-\infty})$ on $X\times M$ if on every local coordinate patch $D$ of $X$ and local coordinate patch $D_1$ of $M$, $A_k=B_k+O(k^{-\infty})$ on $D\times D_1$.
When $X=M$, we sometime write on $X$.

We recall the definition of the semi-classical symbol spaces

\begin{defn} \label{d-gue140826}
Let $W$ be an open set in $\Real^N$. Let
\[
\renewcommand{\arraystretch}{1.2}
\begin{array}{c}
S(1)=S(1;W):=\Big\{a\in\mathcal{C}^\infty(W);\, \forall\alpha\in\mathbb N^N_0:
\sup_{x\in W}\abs{\pr^\alpha a(x)}<\infty\Big\},\\
S^0_{{\rm loc\,}}(1;W):=\Big\{(a(\cdot,m))_{m\in\Real};\,\forall\alpha\in\mathbb N^N_0,
\forall \chi\in \mathcal{C}^\infty_c(W)\,:\:\sup_{m\in\Real, m\geq1}\sup_{x\in W}\abs{\pr^\alpha(\chi a(x,m))}<\infty\Big\}\,.
\end{array}
\]
Hence $a(\cdot,k)\in S^\ell_{{\rm loc}}(1;W)$ if for every $\alpha\in\mathbb N^N_0$ and $\chi\in\mathcal{C}^\infty_c(W)$, there
exists $C_\alpha>0$ independent of $k$, such that $\abs{\pr^\alpha (\chi a(\cdot,k))}\leq C_\alpha k^{\ell}$ holds on $W$.

Consider a sequence $a_j\in S^{\ell_j}_{{\rm loc\,}}(1)$, $j\in\N_0$, where $\ell_j\searrow-\infty$,
and let $a\in S^{\ell_0}_{{\rm loc\,}}(1)$. We say
\[
a(\cdot,k)\sim
\sum\limits^\infty_{j=0}a_j(\cdot,k)\:\:\text{in $S^{\ell_0}_{{\rm loc\,}}(1)$},
\]
if, for every
$N\in\N_0$, we have $a-\sum^{N}_{j=0}a_j\in S^{\ell_{N+1}}_{{\rm loc\,}}(1)$ .
For a given sequence $a_j$ as above, we can always find such an asymptotic sum
$a$, which is unique up to an element in
$S^{-\infty}_{{\rm loc\,}}(1)=S^{-\infty}_{{\rm loc\,}}(1;W):=\cap _\ell S^\ell_{{\rm loc\,}}(1)$.

Let $\ell\in\Real$ and let
\[
S^\ell_{{\rm loc},{\rm cl\,}}(1):=S^\ell_{{\rm loc},{\rm cl\,}}(1;W)
\] be the set of all $a\in S^\ell_{{\rm loc}}(1;W)$ such that we can find $a_j\in\mathcal{C}^\infty(W)$ independent of $k$, $j=0,1,\ldots$,  such that 
\[
a(\cdot,k)\sim
\sum\limits^\infty_{j=0}k^{\ell-j}a_j(\cdot)\:\:\text{in $S^{\ell_0}_{{\rm loc\,}}(1)$}.
\]

Similarly, we can define $S^\ell_{{\rm loc\,}}(1;Y,E)$, $S^\ell_{{\rm loc\,},{\rm cl\,}}(1;Y,E)$  in the standard way, where $Y$ is a smooth manifold and $E$ is a vector bundle over $Y$. 
\end{defn}

\subsection{CR geometry and CR line bundles}
\label{sec:crgeom}

We recall some notations concerning CR geometry. Let $(X, T^{1,0}X)$ be a compact and orientable CR manifold of dimension $2n+1$, $n\geq 1$, where $T^{1,0}X$ is a CR structure of $X$. There is a unique sub-bundle $HX$ of $TX$ such that $\mathbb{C}\,HX=T^{1,0}X \oplus T^{0,1}X$, $T^{0,1}X=\overline{T^{1,0}X}$. Let $J:HX\To HX$ be the complex structure map given by $J(u+\ol u)=i u-i\ol u$, for every $u\in T^{1,0}X$. 
By complex linear extension of $J$ to $\mathbb{C}\,TX$, the $i$-eigenspace of $J$ is $T^{1,0}X$. We shall also write $(X, HX, J)$ to denote a CR manifold. In this work, we assume that

\begin{ass}\label{a-gue240714yyd}
There is a global vector field $T\in\mathcal{C}^\infty(X,TX)$ such that 
\begin{equation}\label{e-gue240714yyd}
\begin{split}
&T^{1,0}X\oplus T^{0,1}X\oplus\mathbb CT(x)=\mathbb CT_xX,\ \ \mbox{for all $x\in X$},\\
&[T,\mathcal{C}^\infty(X,T^{1,0}X)]\subset\mathcal{C}^\infty(X,T^{1,0}X).
\end{split}
\end{equation}
We say that $T$ is a transversal CR vector field. 
\end{ass}

From now on, we fix $T\in\mathcal{C}^\infty(X,TX)$ such that \eqref{e-gue240714yyd} hold. Let 
\begin{equation}\label{e-gue240714yydI}
\begin{split}
\eta: \mathbb R\times X&\To X,\\
(\eta,x)&\To\eta\cdot x,
\end{split}
\end{equation}
be the $\mathbb R$-action induced by the flow of $T$, that is, 
\[(Tu)(x)=\frac{\pr}{\pr\eta}(u(\eta\cdot x))|_{\eta=0},\ \ \mbox{for all $u\in\mathcal{C}^\infty(X)$}.\]
Let $\omega_0\in\mathcal{C}^\infty(X,T^*X)$ be the global one form given by 
\begin{equation}\label{e-gue240718yyd}
\begin{split}
\mbox{$\langle\,\omega_0(x)\,,\,u\,\rangle=0$, for every $u\in H_xX$}, \\
\omega_0(T)\equiv -1,\quad 		\mathrm{d}\omega_0(T,\cdot)\equiv0\ \ \mbox{on $TX$}.
\end{split}
\end{equation}
For each $x \in X$, we define a Hermitian quadratic form $\mathcal{L}_x$ on $T^{1.0}_xX$ as follows: for $U, V \in T^{1,0}_xX$,
\begin{equation}\label{e-gue240718yydI}
\mathcal{L}_x(U,\overline{V}) = \frac{1}{2}\,\mathrm{d}\omega_0(JU, \overline{V}) = -\frac{1}{2i}\,\mathrm{d}\omega_0(U,\overline{V}).
\end{equation}
The Hermitian quadratic form $\mathcal{L}_x$ on $T^{1,0}_xX$ is called Levi form at $x$.

 Fix a smooth Hermitian metric $\langle\, \cdot \,|\, \cdot \,\rangle$ on $\mathbb{C}TX$ so that $T^{1,0}X$ is orthogonal to $T^{0,1}X$, $\langle\, u \,|\, v \,\rangle$ is real if $u, v$ are real tangent vectors, $\langle\,T\,|\,T\,\rangle=1$ and $T$ is orthogonal to $T^{1,0}X\oplus T^{0,1}X$. For $u \in \mathbb{C}TX$, we write $|u|^2 := \langle\, u\, |\, u\, \rangle$. Denote by $T^{*1,0}X$ and $T^{*0,1}X$ the dual bundles of $T^{1,0}X$ and $T^{0,1}X$, respectively. They can be identified with sub-bundles of the complexified cotangent bundle $\mathbb{C}T^*X$. For $q=0,1,\ldots,n$, let $T^{*0,q}X:=\Lambda^q(T^{*0,1}X)$. Let $\Omega^{0,q}(X):=\mathcal{C}^\infty(X,T^{*0,q}X)$ and for an open set $D\subset X$, let $\Omega^{0,q}_c(D):=\mathcal{C}^\infty_c(D,T^{*0,q}X)$. 

Let $\ddbar_b\,:\,\Omega^{0,q}(X)\rightarrow \Omega^{0,q+1}(X)$ be the tangential Cauchy-Riemann operator. Since the $\mathbb{R}$-action is CR, it is straightforward to see that
\[ T\overline{\partial}_b= \overline{\partial}_bT \]
on $\Omega^{0,q}(X)$.

\begin{defn}
	Let $D$ be a sufficiently small open set. We say that a function $u\in\mathcal{C}^{\infty}(D)$ is rigid if $T u = 0$. We say that a function $u\in \mathcal{C}^{\infty}(D)$ is CR if $\overline{\partial}_bu=0$. We say that $u\in\mathcal{C}^\infty(D)$ is rigid CR if  $\ddbar_bu=0$ and $Tu=0$.
\end{defn} 

The following definitions for CR vector bundles can be found in \cite{hn}.
\begin{defn}\label{Def:CRVB}
	A complex  line bundle $\pi\,:\,L\rightarrow X$ is called a CR line bundle if 
	\begin{itemize}
		\item [(i)] \(L\) is a CR manifold of codimension \(2\),
		\item [(ii)] \(\pi\colon L\to X\) is a CR submersion,
		\item [(iii)] \(L\oplus L\ni(\xi_1,\xi_2)\to \xi_1+\xi_2\in L\) and \(\mathbb C\times E\ni(\lambda,\xi)\to \lambda \xi\in L\) are CR maps.
	\end{itemize}
	A smooth section \(s\in\mathcal{C}^\infty(U,L)\) defined on an open set \(U\subset X\) is called CR section if the map \(s\colon U\to L\) is CR.
\end{defn}

Let $L$ be a CR line bundle over $X$. The tangential Cauchy-Riemann operator can be defined on sections of $L$:
\[\ddbar_b: \Omega^{0,q}(X,L)\To\Omega^{0,q+1}(X,L^k),\]
where $\Omega^{0,q}(X,L):=\mathcal{C}^\infty(X,L\otimes T^{*0,q}X)$. 

\begin{defn}\label{Def:LocTriv}
 A CR line bundle $L\To X$ is called locally CR trivializable if for any point \(p\in X\)  there exists an open neighborhood \(U\subset X\) such that \(L|_U\) is CR line bundle isomorphic to the trivial CR vector bundle \(U\times\mathbb C\).
\end{defn}

\begin{defn}\label{Def:CRBL}
Let $L$ be a CR line bundle over $X$. A CR bundle lift of \(T\) to $L$ is a smooth linear partial differential operator $$T=T^L\colon\mathcal{C}^\infty(X,L)\to   \mathcal{C}^\infty(X,L)$$ such that
	\begin{itemize}
		\item [(i)] \(T^{L}(f\cdot s)=T(f)\cdot s+fT^L(s)\) for all \(f\in \mathcal{C}^\infty(X)\) and \(s\in\mathcal{C}^\infty(X,L)\),
		\item [(ii)] \([T^{L},\overline{\partial}_b]=0\).
	\end{itemize}
\end{defn} 

In order to define \([T^{L},\overline{\partial}_b]\) we need to define \(T^{L}\) on \((0,1)\) forms with values in \(L\) first. But this definition follows immediately from the fact that any $w\in\Omega^{0,1}(X,L)$
locally can be written $w=s\otimes g$ where $g$ is a $(0, 1)$-form and $s$ is a local frame of $L$  and that \(T\) is defined also for \((0 ,q)\)-forms using the Lie derivative.

\begin{defn}\label{Def:RigidCVB}
	A CR line bundle $L\To X$ with a CR bundle lift \(T^L\) of \(T\)  is called rigid CR (with respect to \(T^L\)) if for every point \(p\in X\) there exists an open neighborhood \(U\) around \(p\) and a CR frame $s$ of \(L|_U\) with \(T^L(s)=0\).
\end{defn}

A section $s\in\mathcal{C}^\infty(X,L)$ is called a rigid CR section if $T^Ls=0$ and $\overline\partial_b s=0$. The frame $s$ in Definition~\ref{Def:RigidCVB}  is called a rigid CR frame of $L|_U$.
Note that it follows from~\cite[Lemma 2.6]{hhl} that any rigid CR line bundle is locally CR trivializable. The following is well-known~\cite[Lemma 2.10]{hhl}

\begin{lem}\label{l-gue240715yyd}
	Let $L\To X$ be a CR line bundle over $X$ and recall that $T\in  \mathcal{C}^\infty(X,TX)$ is a CR vector field. The following are equivalent:
	\begin{itemize}
		\item [(i)] \(T\) has a CR bundle lift \(T^L\) such that $L\To X$ is rigid CR with respect to \(T^L\).
		\item [(ii)] There exist an open cover \(\{U_j\}_{j\in\N}\) of \(X\) and CR frames \(\{s_j\}\) for \(E|_{U_j}\), \(j\in\N\), such that the corresponding transition functions are rigid CR. 
	\end{itemize}
\end{lem}

From now on, we assume that $L\To X$ is a rigid CR line bundle (with respect to the lifting $T^L$). To simplify the notations, we also write $T$ to denote the lifting $T^L$. 

\begin{defn}\label{d-gue240715yyd}
Let $L\To X$ be a rigid CR line bundle (with respect to the lifting $T^L$). 
Let $h^L$ be a Hermitian metric on $L$. We say that $h^L$ is a rigid Hermitian metric or $\mathbb R$-invariant Hermitian metric if for every local rigid frame $s$ of $L$, we have $T|s|^2_{h^L}=0$.
\end{defn}

By Lemma~\ref{l-gue240715yyd}, there exists
an open covering $(U_j)^N_{j=1}$ and a family of rigid CR trivializing frames $\{s_j\}_{j=1}^N$ with each $s_j$ defined on $U_j$ and the transition functions between different rigid CR frames are rigid CR functions.
Let $L^k$ be the $k$-th tensor power of $L$.
Then $\{s_j^{k}\}^N_{j=1}$ is a family of  rigid CR trivializing frames on each $U_j$. Let
$\overline\partial_b:\Omega^{0,q}(X, L^k)\rightarrow\Omega^{0,q+1}(X, L^k)$ be the tangential Cauchy-Riemann operator.
Since $L^k$ is rigid  CR, we have $\overline\partial_b f=\overline\partial_b f_j\otimes s_j^k$, $Tf=(Tf_j)\otimes s_j^k$ for any $f=f_j\otimes s_j^k\in\Omega^{0, q}(X, L^k)$ and
\begin{equation}\label{e-gue150508d}
T\ddbar_b=\ddbar_bT\ \ \mbox{on $\Omega^{0,q}(X,L^k)$}.
\end{equation} 

Let $h^L$ be a rigid Hermitian fiber metric on $L$. The local weight of $h^L$
with respect to a local rigid CR trivializing section $s$ of $L$ over an open subset $D\subset X$
is the function $\Phi\in\mathcal{C}^\infty(D, \mathbb R)$ for which
\begin{equation}\label{e-gue150808g}
|s(x)|^2_{h^L}=e^{-2\Phi(x)}, x\in D.
\end{equation}
We denote by $\Phi_j$ the weight of $h^L$ with respect to $s_j$.

\begin{defn}\label{d-gue150808g}
Let $h^L$ be a rigid Hermitian metric on $L$.
The curvature of $(L,h^L)$ is the the Hermitian quadratic form $R^L=R^{(L,h^L)}$ on $T^{1,0}X$
defined by
\begin{equation}\label{e-gue150808w}
R_p^L(U, V)=\,\big\langle d(\overline\partial_b\Phi_j-\partial_b\Phi_j)(p),
U\wedge\overline V\,\big\rangle,\:\: U, V\in T_p^{1,0}X,\:\: p\in U_j.
\end{equation}
\end{defn}

Due to~\cite[Proposition 4.2]{HM12}, $R^L$ is a well-defined global Hermitian form,
since the transition functions between different frames $s_j$ are annihilated by $T$. 

\begin{defn}\label{d-gue150808gI}
We say that $(L,h^L)$ is positive if there is an interval $I\subset\mathbb R$ such that the associated curvature $R^L_x-2s\mathcal{L}_x$ is positive definite
at every $x\in X$, for every $s\in I$.
\end{defn}

In this work, we assume that 

\begin{ass}\label{a-gue240715yyd}
There is a positive rigid Hermitian metric $h^L$ on $L$. 
\end{ass}

From now on, we fix a positive rigid Hermitian metric $h^L$ on $L$ and we have 

\begin{equation}\label{e-gue240715ycd}
\mbox{$R^L_x-2s\mathcal{L}_x$ is positive definite, for every $x\in X$, $s\in I$},
\end{equation}
where $I\subset\mathbb R$ is a bounded open interval. From~\cite[Corollary 3.8]{hhl}, we see that the $\mathbb R$-action on $X$ comes form a torus action on $X$. The following is well-known~\cite[Theorem 3.12]{hhl}

\begin{thm}\label{t-gue240716yyd}
With the assumptions and notations above, we can find local CR rigid trivializations of $L$ defined on open sets $D_j$, $j=1,\ldots,N$, such that $X=\bigcup^N_{j=1}D_j$, and
\begin{equation}\label{e-gue240716yyd}
D_j=\bigcup_{\eta\in\mathbb R}\{\eta\cdot x;\, x\in D_j\}
\end{equation}
for every $j=1,\ldots,N$.
\end{thm}

Theorem~\ref{t-gue240716yyd} tells us that the $\mathbb{R}$-action on $X$ can be lifted to $L$. From now on, for any local CR rigid trivialization $s$ defined on an open set $D$ of $X$, we will always assume that $D=\bigcup_{\eta\in\mathbb R}\{\eta\cdot x;\, x\in D\}$ and $s$ is $\mathbb R$-invariant. 

\subsection{The weighted Fourier-Szeg\H{o} operator}\label{e-gue240717yyd}

We will use the same notations and assumptions as before. 
Let $L^k$ be the $k$-th power of $L$. The Hermitian metric on $L^k$ induced by $h^L$ is denoted by $h^{L^k}$. 
We assume that the Hermitian metric $\langle\,\cdot\,|\,\cdot\,\rangle$ is $\mathbb R$-invariant. 
We denote by $dV_X$ the volume form on $X$ induced by $\langle\,\cdot\,|\,\cdot\,\rangle$.
Let $(\,\cdot\,|\,\cdot\,)_k$ be the $L^2$ inner product on $\mathcal{C}^\infty(X,L^k)$ induced by $h^{L^k}$ and
$dV_X$. Let $L^2(X,L^k)$ be the completion of $\mathcal{C}^\infty(X,L^k)$ with respect
to $(\,\cdot\,|\,\cdot\,)_k$. We extend $(\,\cdot\,|\,\cdot\,)_k$ to $L^2(X,L^k)$.
Consider the operator
\[-iT: \mathcal{C}^\infty(X,L^k)\To\mathcal{C}^\infty(X,L^k)\]
and we extend $-iT$ to $L^2$ space by
\[\begin{split}
&-iT: {\rm Dom\,}(-iT)\subset L^2(X,L^k)\To L^2(X,L^k),\\
&{\rm Dom\,}(-iT)=\set{u\in L^2(X,L^k);\, -iTu\in L^2(X,L^k)}.
\end{split}\]
From~\cite[Theorems 4.1, 4.5]{hhl},  $-iT$ is self-adjoint with respect to $(\,\cdot\,|\,\cdot\,)_k$, ${\rm Spec\,}(-iT)$ is countable
and every element in ${\rm Spec\,}(-iT)$ is an eigenvalue of $-iT$, where ${\rm Spec\,}(-iT)$ denotes the spectrum of $-iT$. 

Let $\ddbar_b:\Omega^{0,q}(X,L^k)\To\Omega^{0,q+1}(X,L^k)$
be the tangential Cauchy-Riemann operator with values in $L^k$.
For every $\alpha\in {\rm Spec\,}(-iT)$, put
\begin{equation}\label{e-gue150806I}
\mathcal{C}^\infty_{\alpha}(X,L^k):=\set{u\in\mathcal{C}^\infty(X,L^k);\, -iTu=\alpha u},
\end{equation}
and
\begin{equation}\label{e-gue150806II}
\mathcal{H}^0_{b,\alpha}(X,L^k):=\set{u\in\mathcal{C}^\infty_{\alpha}(X,L^k);\, \ddbar_bu=0}.
\end{equation}
It is easy to see that
for every $\alpha\in {\rm Spec\,}(-iT)$, we have
\begin{equation}\label{e-gue150806III}
{\rm dim\,}\mathcal{H}^{0}_{b,\alpha}(X,L^k)<\infty.
\end{equation}
For any interval $J\subset\mathbb R$, put
\begin{equation}\label{e-gue150806IV}
\mathcal{H}^0_{b,J}(X,L^k):=\bigoplus_{\alpha\in{\rm Spec\,}(-iT), \alpha\in J}\mathcal{H}^0_{b,\alpha}(X,L^k).
\end{equation}

For every $\alpha\in {\rm Spec\,}(-iT)$, let $L^2_{\alpha}(X,L^k)\subset L^2(X,L^k)$ be the eigenspace of $-iT$ with eigenvalue $\alpha$. It is easy to see that
$L^2_{\alpha}(X,L^k)$ is the completion of $C^\infty_{\alpha}(X,L^k)$ with respect to $(\,\cdot\,|\,\cdot\,)_k$.
Let
\begin{equation}\label{e-gue150807}
Q_{\alpha,k}:L^2(X,L^k)\To L^2_\alpha(X,L^k)
\end{equation}
be the orthogonal projection with respect to $(\,\cdot\,|\,\cdot\,)_k$.
We have the Fourier decomposition
\[L^2(X,L^k)=\bigoplus_{\alpha\in{\rm Spec\,}(-iT)} L^2_\alpha(X,L^k).\]
We first construct a bounded operator on $L^2(X,L^k)$
by putting a weight on the components of the Fourier
decomposition with the help of a cut-off function. Let 
\begin{equation}\label{e-gue160105}
\tau\in\mathcal{C}^\infty_c(I),\ \ \tau\geq0,
\end{equation}
where $I$ is the interval as in \eqref{e-gue240715ycd}.
Let $F_{k,\tau}:L^2(X,L^k)\To L^2(X,L^k)$
be the bounded operator given by
\begin{equation}\label{e-gue150807I}
\begin{split}
F_{k,\tau}:L^2(X,L^k)&\To L^2(X,L^k),\\
u&\mapsto\sum_{\alpha\in{\rm Spec\,}(-iT)}\tau\left(\frac{\alpha}{k}\right)Q_{\alpha,k}(u).
\end{split}
\end{equation}
We consider the partial Szeg\H{o} projector
\begin{equation}\label{e-gue150806V}
\Pi_{k,I}:L^2(X,L^k)\To \mathcal{H}^0_{b,I}(X,L^k)
\end{equation}
which is the orthogonal projection on the space of $\Real$-equivariant CR functions
in $\mathcal{H}^0_{b,I}(X,L^k)$. Finally, we consider the weighted Fourier-Szeg\H{o}
operator
\begin{equation}\label{e-gue150807II}
P_{k,\tau^2}:=F_{k,\tau}\circ\Pi_{k,I}\circ F_{k,\tau}:L^2(X,L^k)
\To \mathcal{H}^0_{b,I}(X,L^k).
\end{equation}
The Schwartz kernel of $P_{k,\tau^2}$ with respect to $dV_X$ is the smooth section
$(x,y)\mapsto P_{k,\tau^2}(x,y)\in L^k_x\otimes(L^k_y)^*$ satisfying
\begin{equation}\label{sk}
(P_{k,\tau^2}u)(x)=\int_X P_{k,\tau^2}(x,y)u(y)\,dV_X(y)\,,\:\:u\in L^2(X,L^k).
\end{equation} 

We pause and introduce some notations. Let 
 $A_{k}: \mathcal{C}^\infty_c(X,L^k)\To\mathcal{C}^\infty(X,L^k)$ be a continuous operator with distribution 
 kernel $A_k(x,y)\in\mathcal{D}'(X\times X,L^k\boxtimes(L^k)^*)$. 
Let $s_1$, $s_2$ be local trivializing CR rigid sections of $L$ defined on $\mathbb R$-invariant open sets $D_1\subset X$, $D_2\subset X$, respectively, $|s_j(x)|^2_{h^L}=e^{-2\Phi_j(x)}$, $\Phi_j\in\mathcal{C}^\infty(D_j)$,$j=1,2$.  The localization of $A_{k}$ with respect to $s_1$ and $s_2$ are given by
\begin{equation*}
\begin{split}
&A_{k, s_1,s_2}: \mathcal{C}^\infty_c(D_1)\rightarrow\mathcal{C}^\infty(D_2),\\
&A_{k, s_1,s_2}(u):=e^{-k\Phi_2}s^{-k}_2A_{k}(s^{k}_1e^{k\Phi_1}u), \forall u\in\mathcal{C}^\infty_c(D_1).
\end{split}
\end{equation*}
When $s=s_1=s_2$, $D_1=D_2$, we write $A_{k,s}:=A_{k,s_1,s_2}$.
We write $A_k=O(k^{-\infty})$ on $X$ or $A_k(x,y)=O(k^{-\infty})$ on $X\times X$ if for 
any local trivializing CR rigid sections $s_1$, $s_2$ of $L$ defined on $\mathbb R$-invariant open sets $D_1\subset X$, $D_2\subset X$, respectively, we have 
\[A_{k, s_1,s_2}=O(k^{-\infty})\ \ \mbox{on $D_1\times D_2$}.\]

Let $s$ be a local trivializing CR rigid section of $L$ defined on a $\mathbb R$-invariant open set $D$, $|s(x)|^2_{h^L}=e^{-2\Phi(x)}$, $x\in D$. 
As before, let $P_{k,\tau^2,s}: \mathcal{C}^\infty_c(D)\To\mathcal{C}^\infty(D)$ be the localization of $P_{k,\tau^2}$ with respect to $s$. Let $P_{k,\tau^2,s}(x,y)\in\mathcal{C}^\infty(D\times D)$
be the Schwartz kernel of $P_{k,\tau^2,s}$ with respect to $dV_X$.
We have the following~\cite[Theorem 1.1]{hhl}

\begin{thm}\label{t-gue240717yyd}
With the notations and assumptions above,
consider a point $p\in X$ and a coordinates neighborhood
$(D,x=(x_1,\ldots,x_{2n+1}))$ centered at $p=0$ with $T=-\frac{\pr}{\pr x_{2n+1}}$.
Let $s$ be a local rigid CR trivializing section of $L$ on $D$ and set $\abs{s}^2_{h}=e^{-2\Phi}$. With the notations used above, 
\begin{equation}\label{sk2}
P_{k,\tau^2,s}(x,y)=\int_{\mathbb R} e^{ik\varphi(x,y,t)}a(x,y,t,k)dt+O(k^{-\infty})\:\:
\text{on $D\times D$},
\end{equation}
where $\varphi\in\mathcal{C}^\infty( D\times D\times I)$ is a phase function such that
for some constant $c>0$ we have
\begin{equation}\label{e-gue150807b}
\begin{split}
&d_x\varphi(x,y,t)|_{x=y}=-2{\rm Im\,}\ddbar_b\Phi(x)+t\omega_0(x),\\
&d_y\varphi(x,y,t)|_{x=y}=2{\rm Im\,}\ddbar_b\Phi(x)-t\omega_0(x),\\
&{\rm Im\,}\varphi(x,y,t)\geq c|z-w|^2,\\
&(x,y,t)\in D\times D\times I, x=(z, x_{2n+1}), y=(w, y_{2n+1}),\\
&\mbox{${\rm Im\,}\varphi(x,y,t)+\abs{\frac{\pr\varphi}{\pr t}(x,y,t)}^2\geq c\abs{x-y}^2$,
$(x,y,t)\in D\times D\times I$},\\
&\mbox{$\varphi(x,y,t)=0$ and $\frac{\pr\varphi}{\pr t}(x,y,t)=0$ if and only if $x=y$},
\end{split}
\end{equation}
and $a(x,y,t,k)\in S^{n+1}_{{\rm loc\,}}(1;D\times D\times I)
\cap\mathcal{C}^\infty_c(D\times D\times I)$ is a symbol with expansion
\begin{equation}\label{e-gue150807bI}
\begin{split}
&a(x,y,t,k)\sim\sum^\infty_{j=0}a_j(x,y,t)k^{n+1-j}\text{ in }S^{n+1}_{{\rm loc\,}}
(1;D\times D\times I), \\
\end{split}\end{equation}
and for
$x\in D_0$ and $t\in I$ we have
\begin{equation}\label{e-gue150807a}
a_0(x,x,t)=(2\pi)^{-n-1}\abs{\det\bigr(R^L_x-2t\mathcal{L}_x\bigr)}\tau^2(t),
\end{equation}
where $\omega_0\in\mathcal{C}^\infty(X,T^*X)$ is the global real $1$-form of unit length orthogonal
to $T^{*1,0}X\oplus T^{*0,1}X$, see \eqref{e-gue240718yyd}, $\abs{\det{\bigr(R^L_x-2t\mathcal{L}_x\bigr)}}=\abs{\lambda_1(x,t)\cdots\lambda_{n}(x,t)}$,
where $\lambda_j(x,t)$, $j=1,\ldots,n$, are the eigenvalues of the Hermitian quadratic
form $R^L_x-2t\mathcal{L}_x$ with respect to $\langle\,\cdot\,|\,\cdot\,\rangle$,
$R^L_x$ and $\mathcal{L}_x$ denote the curvature two form of $L$ and the Levi form
of $X$ respectively (see Definition~\ref{d-gue150808g} and \eqref{e-gue240718yydI}).
\end{thm}

Now, we shall introduce local coordinates and give a local expression for the phase function. In Section $4.4$ of \cite{hsiao} the author determined the tangential Hessian of the phase function. We recall \cite[Theorem 3.13]{hml}, which is an improvement of the result in \cite{hsiao} for the case of CR manifolds with transversal CR $\mathbb R$-action. 

\begin{thm} \label{thm:hml} Fix $(p,p,t_0) \in D\times D \times I$, and let $\overline{W}_{1,t_0},\,\dots,\overline{W}_{n,t_0}$ be an orthonormal rigid frame of $T^{1,0}_xX$ varying smoothly with $x$ in a neighborhood of $p$, for which the Hermitian quadratic form $R^L_x - 2t_0\mathcal{L}_x$ is diagonal at $p$, that is,
	\[(R^L_p-2\,t_0\mathcal{L}_p)(\overline{W}_{j,t_0}(p),\,W_{k,t_0}(p))=\lambda_j(t_0)\,\delta_{j,k}\,, \]
for $j,\,k=1,\dots,\,n$. Let $s$ be a local rigid CR frame of $L$ defined on $D$, $|s|^2_{h^L}=e^{-2\Phi}$, $\Phi\in\mathcal{C}^\infty(D)$. Let $x=(x_1,\ldots,x_{2n+1})$ be local coordinates defined in some small neighborhood of $p$ with $x(p)=0$ such that 
	\begin{equation}\label{e-gue240718yydII}
 \begin{split}
 &\omega_0(p)=\mathrm{d}x_{2n+1}\,, \qquad T=-\frac{\partial}{\partial x_{2n+1}},\\
	&\overline{W}_{j,\,t_0}(p)=\frac{\partial}{\partial z_j}+i\sum_{t=1}^{n}\tau_{j,\,t}\overline{z}_t\frac{\partial}{\partial x_{2n+1}} +O(\lvert z\rvert^2)\quad \text{for } j=1,\dots,n\,,\\
	&\left\langle \frac{\partial}{\partial x_j}(p)\bigg\vert \frac{\partial}{\partial x_\ell}(p) \right\rangle = 2\,\delta_{j,\ell}\quad \text{for } j,\ell=1,\dots,2n\,,\\
 &z_j=x_{j}+ix_{d+j},\ \j=1,\ldots,d,\\
 &z_j=x_{2j-1}+ix_{2j},\ \ j=d+1,\ldots,n,
 \end{split}\end{equation}
	and
	\begin{equation}\label{e-gue240718yyda}
	    \Phi(x)=\frac{1}{2}\sum_{l,\,t=1}^{n}\mu_{t,l}z_t\overline{z}_l+\sum_{l,\,t=1}^{n}(a_{l,t}z_lz_t+\overline{a}_{l,t}\overline{z}_l\overline{z}_t)+O(\lvert z\rvert^3)\,, \end{equation}
	where $\tau_{t,l},\,\mu_{t,l},\,a_{t,l}\in \mathbb{C}$, $\mu_{t,l}=\overline{\mu}_{l,t}$, $l,t=1,\dots,n$. Then, there exists a neighborhood of $(p,\,p)$ such that 
	\begin{align} \label{eq:phase}
		\varphi(x,y,t_0)=&t_0(x_{2n+1}-y_{2n+1})-\frac{i}{2}\sum_{j,l=1}^{n}(a_{l,j}+a_{j,l})(z_jz_l-w_jw_l) \\
		&+\frac{i}{2}\sum_{j,l=1}^{n}(\overline{a_{l,j}}+\overline{a_{j,l}})(\overline{z}_j\overline{z}_l-\overline{w}_j\overline{w}_l) +\frac{i\,t_0}{2}\sum_{j,l=1}^{n}(\overline{\tau_{l,j}}-\tau_{j,l})({z}_j\overline{z}_l-{w}_j\overline{w}_l) \notag \\
		&-\frac{i}{2}\sum_{j=1}^{n}\lambda_j(t_0)({z}_j\overline{w}_j-\overline{z}_j{w}_j) \notag\\
		&+\frac{i}{2}\sum_{j=1}^{n}\lambda_j(t_0)\lvert {z}_j-{w}_j\rvert^2+O(\lvert (z,w)\rvert^3). \notag
	\end{align}
\end{thm}


\section{CR manifolds with group action}\label{s-gue240718yyd}

From now on, we assume that $X$ admits a $d$-dimensional compact Lie group action $G$ with Lie algebra $\mathfrak{g}$. As before, let us denote by $\underline{\mathfrak{g}}$ the space of vector field on $X$ induced by the Lie algebra $\mathfrak{g}$ of $G$. Furthermore for every $\xi \in \mathfrak{g}$, we write $\xi_X$ to denote the infinitesimal vector field on $X$ induced by $\xi$. Recall that we work with Assumption~\ref{ass:2}. 

As Definition~\ref{Def:RigidCVB}, we can define $G$-equivariant rigid CR line bundle.

\begin{defn}\label{Def:RigidCVBz}
	A rigid CR line bundle $L\To X$ is called $G$-equivariant  rigid CR if the action $G$ can be lifted to $L$, for every $\xi_X\in\underline{\mathfrak{g}}$, $\xi_X$ can be lifted to $L$ and for every point \(p\in X\) there exists an open neighborhood \(U\) around \(p\) and a CR frame $s$ of \(L|_U\) with \(T(s)=0\) and $\xi_Xs=0$, for every $\xi_X\in\underline{\mathfrak{g}}$.
\end{defn}

From now on, we assume that $L$ is a $G$-equivariant rigid CR line bundle. We can repeat the proof of~\cite[Theorem 3.12]{hhl} with minor changes and deduce 

\begin{thm}\label{t-gue240719yyd}
With the assumptions and notations above, we can find local CR rigid $G$-invariant trivializing section $s_j$ defined on an open subset $D_j$ of $X$, $j=1,\ldots,N$, $N\in\mathbb N$, such that $X=\cup^N_{j=1}D_j$ and $D_j=\cup_{(g,\eta)\in G\times\mathbb R}\{g\cdot\eta\cdot x;\, x\in D_j\}$, for every $j=1,\ldots,N$. 
\end{thm}

From now on, for any local CR rigid trivializing section $s$ defined on an open set $D$ of $X$, by Theorem~\ref{t-gue240719yyd}, without loss of generality, we will always assume that $D$ and $s$ are $G\times\mathbb R$-invariant. 

From now on, we assume that the Hermitian metrices $\langle\,\cdot\,|\,\cdot\,\rangle$ and $h^L$ are $G\times\mathbb R$-invariant. We now introduce moment map associated to $R^L$ and $\omega_0$.

\begin{defn}\label{d-gue240719yyd}
Let $s$ be a CR rigid $G$-invariant trivializing section of $L$ defined on an open subset $D\subset X$, $|s|^2_{h^L}=e^{-2\Phi}$, $\Phi\in\mathcal{C}^\infty(D)$. Let 
\[\gamma_{\Phi}:=4{\rm Im\,}\ddbar_b\Phi=(-2i)(\ddbar_b\Phi-\pr_b\Phi).\]
The moment map on $D$ associated to the local weight $\Phi$ is the map $\gamma_{\Phi}: D\To\mathfrak{g}^*$ such that for all $x\in D$ and $\xi\in\mathfrak{g}$, we have 
\begin{equation}\label{e-gue240719ycd}
\langle\,\gamma_\Phi(x)\,,\,\xi\,\rangle=\gamma_{\Phi}(\xi_X(x)),
\end{equation}
where $\xi_X$ is the vector field on $X$ induced by $\xi$. 
\end{defn}

\begin{lem}\label{l-gue240719ycd}
Definition~\ref{d-gue240719yyd} is global defined. More precisely, the moment map $\gamma_{\Phi}$ given by \eqref{e-gue240719ycd} is independent of the choices of $\Phi$.
\end{lem}

\begin{proof}
 Let $s_1$, $s_2$ be CR rigid $G$-invariant trivializing sections of $L$ defined on an open subset $D\subset X$, 
$|s_j|^2_{h^L}=e^{2\Phi_j}$, $\Phi_j\in\mathcal{C}^\infty(D)$, $j=1,2$. We have 
$s_1=fs_2$ on $D$, $f$ is a rigid CR $G$-invariant function on $D$, $f(x)\neq0$, for every $x\in D$. Thus,
\[|s_1|^2_{h^L}=e^{-2\Phi_1}=|f|^2e^{-2\Phi_2}=e^{-2\Phi_2+\log|f|^2}\]
on $D$. Hence, 
\[\Phi_1=\Phi_2-\frac{1}{2}\log|f|^2.\]
Thus, 
\[\ddbar_b\Phi_1=\ddbar_b\Phi_2-\frac{1}{2}\frac{\ddbar_b\ol{f}}{\ol f}\]
on $D$ and hence 
\begin{equation}\label{e-gue240719ycdh}
\begin{split}
&\gamma_{\Phi_1}=(-2i)(\ddbar_b\Phi_1-\partial_b\Phi_1)\\
&=(-2i)(\ddbar_b\Phi_2-\partial_b\Phi_2)+i(\frac{\ddbar_b\ol f}{\ol f}-\frac{\partial_bf}{f})\\
&=\gamma_{\Phi_2}+i(\frac{\ddbar_b\ol f}{\ol f}-\frac{\partial_bf}{f}).
\end{split}
\end{equation}
Since $f$ is $G$-invariant, we have 
\begin{equation}\label{e-gue240719ycdi}
\langle\,
\frac{\ddbar_b\ol f}{\ol f}-\frac{\partial_bf}{f}\,,\,\xi_X\,\rangle=0\ \ \mbox{on $X$}, 
\end{equation}
for every $\xi\in\mathfrak{g}$. The lemma follows. 
\end{proof}

 From now on, we write 
 \[\gamma: X\To\mathfrak{g}^*\]
 to denote the moment map given by \eqref{e-gue240719ycd}.

 \begin{defn}\label{d-gue240921yyd}
The moment map associated to the form $\omega_0$ is the map $\mu:X \to \mathfrak{g}^*$ such that, for all $x \in X$ and $\xi \in \mathfrak{g}$, we have 
\begin{equation}\label{e-gue240921yydI}
\langle \mu(x), \xi \rangle = \omega_0(\xi_X(x)), 
\end{equation}
$\xi\in\mathfrak{g}$, $\xi_X$ : the vector field on $X$ induced by $\xi$.
\end{defn}

For every $t\in  I$, let $\hat\mu_t: X\To\mathfrak{g}^*$ be as in \eqref{e-gue240921yyd}. 
Recall that in this work, we work with Assumption~\ref{a-gue240720yyd}.



\subsection{CR reduction with respect to curvature data}
\label{sec:conred}

Let $Y:=\hat\mu^{-1}_t(0)$ and denote by $HY:={\rm Ker\,}\omega_0\cap TY$. Let $X_G:=Y/G$ 
and let $\pi: Y\To X_G$ be the natural projection. Let $HX_G:=\pi_*HY$, 
$\pi_*$ is the push-forward map of $\pi$. We are now going to prove that $X_G$ is a CR manifold. Fix $t_0\in I$. Recall that zero is a regular value of $\hat\mu_{t_0}$. Recall that $I$ is the open interval as in 
\eqref{e-gue240715ycd}. For every $x\in Y$, let 
\[b_x(\cdot,\cdot):=((-2i)R^L_x-2t_0d\omega_0(x))(\cdot,J\cdot)\]
be the bilinear form on $H_xX$. Let 
\[\underline{\mathfrak{g}}^{\perp_b}:=\{v\in HX;\, b(\xi_X,v)=0,\ \ \mbox{for all $\xi_X\in\underline{\mathfrak{g}}$}\}.\]
Since $R^L_x-2t_0\mathcal{L}_x$ is positive, \begin{equation}\label{e-gue240720yd}
\underline{\mathfrak{g}}\cap\underline{\mathfrak{g}}^{\perp_b}=\{0\}.
\end{equation}
Hence, $b$ is a non-degenerate bilinear form and thus 
\begin{equation}\label{e-gue240720yydI}
\underline{\mathfrak{g}}\oplus\underline{\mathfrak{g}}^{\perp_b}=H_xX,
\end{equation}
for every $x\in Y$. Let $H^HY:=\underline{\mathfrak{g}}^{\perp_b}|_Y\cap HY$. 
From \eqref{e-gue240720yydI}, we have 
\begin{equation}\label{e-gue240720yydII}
HY=\underline{g}|_Y\oplus H^HY.
\end{equation}

\begin{lem}\label{l-gue240720yyd}
We have 
\begin{equation}\label{e-gue240720ycd}
\underline{\mathfrak{g}}^{\perp_b}|_Y=JHY,
\end{equation}
and 
\begin{equation}\label{e-gue240720ycdI}
HX|_Y=J\underline{\mathfrak{g}}|_Y\oplus HY
=J\underline{\mathfrak{g}}|_Y\oplus\underline{\mathfrak{g}}|_Y\oplus H^HY.
\end{equation}
\end{lem}

\begin{proof}
Fix $p\in Y$ and let $s$ be a local CR rigid $G$-invariant trivializing section of $L$ 
defined on a $G$-invariant open subset $D$ of $p$ in $X$, $|s|^2_{h^L}=e^{-2\Phi}$. 
For $V\in H_pX$ and $\xi\in\mathfrak{g}$, we have 
\begin{equation}\label{e-gue240721yyd}
b_p(\xi_X,JV)=d\gamma_{\Phi}(\xi_X,JV).
\end{equation} 
From \eqref{e-gue240721yyd}, we see that $V\in\underline{\mathfrak{g}}^{\perp_b}_p$ 
if and only if $d\gamma_{\Phi}(\xi_X,JV)=0$, for all $\xi_X\in\underline{\mathfrak{g}}_p$. Since $0$ is a regular value of $\hat\mu_{t_0}$, $d\gamma_{\Phi}(\xi_X,JV)=0$, for all $\xi_X\in\underline{\mathfrak{g}}_p$ if and only 
$JV\in H_pY$. We get \eqref{e-gue240720ycd}.

Now, for $V\in H_pY$ and $\xi\in\mathfrak{g}$, we have 
\begin{equation}\label{e-gue240721yydI}
b_p(J\xi_X,V)=d\gamma_{\Phi}(V,\xi_X)=0.
\end{equation}
From \eqref{e-gue240721yydI}, 
\[{\rm dim\,}H_pX={\rm dim\,}H_pY+{\rm dim\,}J\underline{\mathfrak{g}}_p\]
and the fact that $b_p$ is non-degenerate, we obtain \eqref{e-gue240720ycdI}. 
\end{proof}

From \eqref{e-gue240720ycd}, we have 

\begin{equation}\label{e-gue240721yydII}
H^HY=JHY\cap HY.
\end{equation}
From \eqref{e-gue240720yydII}, we can identify $HX_G$ with $H^HY$ and from \eqref{e-gue240721yydII}, we can define 
complex structure map $J_G$ on $HX_G$: For $V \in HX_G$, we denote by $V^H$ its lift in $H^HY$, and we define $J_G$ on $X_G$ by
\begin{equation}\label{E:jg}
(J_GV)^H = J(V^H).
\end{equation}
Hence, we have $J_G: HX_G \to HX_G$ such that $J_G^2 = -\operatorname{id}$, where $\operatorname{id}$ denotes the identity map $\operatorname{id}  \, : \, HX_G \to HX_G.$ By complex linear extension of $J_G$ to $\mathbb{C}TX_G$, we can define the $i$-eigenspace of $J_G$ is given by $T^{1,0}X_G \, = \, \left\{ V \in \mathbb{C}HX_G \,; \, J_GV \, =  \,  \sqrt{-1}V  \right\}.$

\begin{prop}\label{p-gue240808yyd}
	The subbundle $T^{1,0}X_G$ is a CR structure of $X_G$.
\end{prop}

\begin{proof}
Let $u, v \in\mathcal{C}^\infty(X_G, T^{1,0}X_G)$, then we can find $U, V \in\mathcal{C}^\infty(X_G, HX_G)$ such that 
\[
u= U - \sqrt{-1}J_GU, \qquad v=V-\sqrt{-1}J_GV.
\]
By \eqref{E:jg}, we have
\[
u^H=U^H-\sqrt{-1}JU^H, \quad v = V^H - \sqrt{-1}JV^H \in T^{1,0}X \cap \mathbb{C}HY.
\]
Since $T^{1,0}X$ is a CR structure and it is clearly that $[u^H, v^H] \in \mathbb{C}HY,$ we have $[u^H, v^H] \in T^{1,0}X \cap \mathbb{C}HY.$ Hence, there is a $W \in\mathcal{C}^\infty(X, HX)$ such that
\[
[u^H, v^H] = W-\sqrt{-1}JW.
\]  
In particular, $W, JW \in HY$. Thus, $W \in HY \cap JHY = H^HY$. Let $X^H \in H^HY$ be a lift of $X \in TX_G$ such that $X^H=W$. Then we have
\[
[u, v] = \pi_*[u^H, v^H] = \pi_*(X^H -\sqrt{-1}JX^H) = X - \sqrt{-1}J_GX \in T^{1,0}X_G,
\]
i.e. we have $[\mathcal{C}^\infty(X_G, T^{1,0}X_G), \mathcal{C}^\infty(X_G, T^{1,0}X_G)] \subset\mathcal{C}^\infty(X_G, T^{1,0}X_G).$ Therefore, $T^{1,0}X_G$ is a CR structure of $X_G$.
\end{proof}

\section{Proofs of Theorem~\ref{thm:Gszego} and Theorem~\ref{thm:2}}

Before delving into the proof we shall introduce local coordinates compatible with the  actions of $G$ on $X$ and the $\mathbb R$-action.

\subsection{Local coordinates}
\label{sec:loccoor}

In this chapter we specialize the local coordinates introduced in Section \ref{sec:crgeom} taking into account the action of the group $G$ and $\mathbb{R}$. Let $e_0\in G$ be the identity element and let $v=(v_1,\dots,v_d)$ be the local coordinates of $G$ defined in a neighborhood $V$ of $e$ with $v(e_0)=(0,\dots,0)$. From now on, we will identify the element $g\in V$ with $v(g)$.    

\begin{prop} \label{prop:coordinates} 
Fix $t_0\in I$. Let $p\in Y$ and let $s$ be a local CR rigid $G$-invariant trivializing section of $L$ defined on a $G$-invariant open set $D$ of $p$ in $X$, $|s|^2_{h^L}=e^{-2\Phi}$. Then there exist local coordinates $x=(x_1,\,\dots,\,x_{2n+1})$ on $X$ defined in a neighborhood $U=U_1\times U_2\subset D$ of $p$ with $p\equiv 0$, $U_1\subset \mathbb{R}^d$ (resp.  $U_2\subset \mathbb{R}^{2n+1-d}$) is an open neighborhood of $0\in \mathbb{R}^d$ (resp.  $0\in \mathbb{R}^{2n+1-d}$) and a smooth function $\Gamma=(\Gamma_1,\dots,\Gamma_d)\in \mathcal{C}^{\infty}(U_2,\,U_1)$ with $\Gamma(0)=0\in \mathbb{R}^d$, $\Gamma$ is independent of $x_{2n+1}$, such that
	\begin{equation}\label{e-gue161202}
 \begin{split}
& \mbox{$T=-\frac{\pr}{\pr x_{2n+1}}$ on $U$},\\
		&(v_1,\dots,\,v_d)\circ (\Gamma(x_{d+1},\dots,x_{2n}), x_{d+1},\dots,x_{2n+1}) \\
		&\quad =(v_1+\Gamma_1(x_{d+1},\dots,x_{2n}),\dots,v_d+\Gamma_d(x_{d+1},\dots,x_{2n}),x_{d+1},\dots,x_{2n+1}) 
	\end{split}
 \end{equation}
for each $(v_1,\dots,v_d)\in V$ and for each $(x_{d+1},\dots,\,x_{2n+1})\in U_2$, 
	\begin{equation}\label{e-gue161206}
 \begin{split}
	&\underline{\mathfrak{g}}=\mathrm{span}\left\{\frac{\partial}{\partial x_1},\cdots,\,\frac{\partial}{\partial x_d}\right\}\,,\\
		&Y =\{x_{d+1}=\cdots = x_{2d} =0 \}\,, \\
		& J\left(\frac{\partial}{\partial x_j}\right)=\frac{\partial}{\partial x_{d+j}} \quad \text{on }Y \text{ for }j=1,\dots, d\,,
  \end{split}
  \end{equation}
  \begin{equation}\label{e-gue161202I}
  \begin{split}
		& T^{1,0}X = \mathrm{span}\{ Z_{1,t_0},\dots,\,Z_{n,t_0}\}\,, \\
		& Z_{j,t_0}(p)=\frac{1}{2}\left(\frac{\partial}{\partial x_j}-i\frac{\partial}{\partial x_{d+j}}\right)(p)\,,\quad j=1,\dots, d\,, \\
		& Z_{j,t_0}(p)=\frac{1}{2}\left(\frac{\partial}{\partial x_{2j-1}}-i\frac{\partial}{\partial x_{2j}}\right)(p)\,,\quad j=d+1,\dots, d\,, \\
		& \left(R^{L}_p-2t_0\mathcal{L}_p\right)\left(Z_{j,t_0}(p),\,\overline{Z}_{\ell,t_0}(p)\right)=\lambda_j(t_0)\delta_{j,\ell}\,,\qquad  j,\ell=1,2,\dots,n\,, \\
		& \langle Z_{j,t_0}(p)\vert Z_{\ell,t_0}(p)\rangle=\delta_{j,\ell}\,,\qquad  j,\ell=1,2,\dots,n\,,
	\end{split}
 \end{equation}
 where $\{Z_{1,t_0},\ldots,Z_{n,t_0}\}$ is an orthonormal basis of $T^{1,0}X$ on $U$ depending smoothly in $x\in U$.

 Moreover, let $\Td x_j$, $j=1,\ldots,2n+1$, be the coordinates as in Proposition~\ref{thm:hml}. we have 
 \begin{equation}\label{e-gue240826yydII}
 \begin{split}
&\Td x_j=x_j+O(\abs{\mathring{x}^2}),\ \ j=1,\ldots,2n,\\
&\Td x_{2n+1}=x_{2n+1}+\sum^d_{j,\ell=1}\frac{i}{2}(\tau_{j,\ell}-\ol{\tau_{j,\ell}})x_jx_\ell+
\sum^d_{j,\ell=1}-(\tau_{j,\ell}+\ol{\tau_{j,\ell}})x_{d+j}x_\ell\\
&\quad+\sum^n_{j=d+1}\sum^d_{\ell=1}i(\tau_{j,\ell}-\ol{\tau_{j,\ell}})x_{2j-1}x_\ell+\sum^n_{j=d+1}\sum^d_{\ell=1}-(\tau_{j,\ell}+\ol{\tau_{j,\ell}})x_{2j}x_\ell+
O(\abs{\mathring{x}^3}),
 \end{split}
 \end{equation}
 where $\mathring{x}=(x_1,\ldots,x_{2n})$.

\end{prop}

\begin{proof}
From the standard proof of Frobenius Theorem, it is not difficult to see that there exist local coordinates $v=(v_1,\ldots,v_d)$ of $G$ defined in a neighborhood $V$ of $e_0$ with $v(e_0)=(0,\ldots,0)$ and local coordinates $x=(x_1,\ldots,x_{2n+1})$ of $X$ defined in a neighborhood $U\subset D$ of $p$ with $x(p)=0$ such that
\begin{equation}\label{e-gue161203}
\renewcommand{\arraystretch}{1.2}
\begin{array}{ll}
&(v_1,\ldots,v_d)\circ (0,\ldots,0,x_{d+1},\ldots,x_{2n+1})\\
&=(v_1,\ldots,v_d,x_{d+1},\ldots,x_{2n+1}),\ \ \forall (v_1,\ldots,v_d)\in V,\ \ \forall (0,\ldots,0,x_{d+1},\ldots,x_{2n+1})\in U,
\end{array}
\end{equation}
and
\begin{equation}\label{e-gue161103a}
\underline{\mathfrak{g}}={\rm span\,}\set{\frac{\pr}{\pr x_1},\ldots,\frac{\pr}{\pr x_d}},
\end{equation} 
\begin{equation}\label{e-gue240809yyd}
T=-\frac{\pr}{\pr x_{2n+1}}\ \ \mbox{on $D$}.
\end{equation}
Consider the linear map 
\[
\renewcommand{\arraystretch}{1.2}
\begin{array}{rl}
R:\underline{\mathfrak{g}}_p&\To\underline{\mathfrak{g}}_p,\\
u&\To Ru,\ \ \langle\,Ru\,|\,v\,\rangle=\langle\,(-2i)R^L_p-2t_0d\omega_0(p))\,,\,Ju\wedge v\,\rangle.
\end{array}
\]
Since $R$ is self-adjoint, by using linear transformation in $(x_1,\ldots,x_d)$, we can take $(x_1,\ldots,x_d)$ such that, for $j, \ell = 1, 2, \ldots, d$,
\begin{equation}\label{e-gue161102}
\langle\,R\frac{\pr}{\pr x_j}(p)\,|\,\frac{\pr}{\pr x_\ell}(p)\,\rangle=4\lambda_j(t_0)\delta_{j,\ell}, \quad \langle\,\frac{\pr}{\pr x_j}(p)\,|\,\frac{\pr}{\pr x_\ell}(p)\,\rangle=2\delta_{j,\ell}.
\end{equation}
By taking linear transformation in $(v_1,\ldots,v_d)$, \eqref{e-gue161203} still hold. 

Let $\hat\mu_{t_0}(\frac{\pr}{\pr x_j})=a_j(x)\in\mathcal{C}^\infty(U)$, $j=1,2,\ldots,d$. Since $a_j(x)$ is $G\times\mathbb R$-invariant, we have $\frac{\pr a_j(x)}{\pr x_s}=0$, $\frac{\pr a_j}{\pr x_{2n+1}}=0$, $j,s=1,2,\ldots,d$.  By the definition of the moment map, we have 
\[
Y\bigcap U=\set{x\in U;\, a_1(x)=\cdots=a_d(x)=0}.
\] 

Since $0$ is a regular value of the moment map $\hat\mu_{t_0}$, the matrix $\left(\frac{\pr a_j}{\pr x_s}(p)\right)_{1\leq j\leq d,d+1\leq s\leq 2n+1}$ is of rank $d$. We may assume that the matrix $\left(\frac{\pr a_j}{\pr x_s}(p)\right)_{1\leq j\leq d,d+1\leq s\leq 2d}$ is non-singular. Thus, $(x_1,\ldots,x_d,a_1,\ldots,a_d,x_{2d+1},\ldots,x_{2n+1})$ are also local coordinates of $X$. Hence, we can take $v=(v_1,\ldots,v_d)$ and $x=(x_1,\ldots,x_{2n+1})$ such that \eqref{e-gue161203}, \eqref{e-gue161103a}, \eqref{e-gue240809yyd} and \eqref{e-gue161102} hold and 
\begin{equation}\label{e-gue161103I}
Y\bigcap U=\set{x=(x_1,\ldots,x_{2n+1})\in U;\, x_{d+1}=\cdots=x_{2d}=0}.
\end{equation} 

On $Y\bigcap U$, let
\[J(\frac{\pr}{\pr x_j})=b_{j,1}(x)\frac{\pr}{\pr x_1}+\cdots+b_{j,2n+1}(x)\frac{\pr}{\pr x_{2n+1}},\ \ j=1,2,\ldots,d.\]
Since we only work on $Y$, $b_{j,\ell}(x)$ is independent of $x_{d+1},\ldots,x_{2d}$, for all $j=1,\ldots,d$, $\ell=1,\ldots,2n+1$. Moreover, it is easy to see that $b_{j,\ell}(x)$ is also independent of $x_{1},\ldots,x_{d}$, $x_{2n+1}$, for all $j=1,\ldots,d$, $k=1,\ldots,2n+1$. Let $\Td x''=(x_{2d+1},\ldots,x_{2n})$. Hence, $b_{j,\ell}(x)=b_{j,\ell}(\Td x'')$, $j=1,\ldots,d$, $\ell=1,\ldots,2n+1$.
We claim that the matrix $\left(b_{j,\ell}(\Td x'')\right)_{1\leq j\leq d,d+1\leq k\leq 2d}$ is non-singular near $p$. If not, it is easy to see that there is a non-zero vector $u\in J\underline{\mathfrak{g}}\bigcap HY$. Let $u=Jv$, $v\in\underline{\mathfrak{g}}$. Then, $v\in\underline{\mathfrak{g}}\bigcap JHY=\underline{\mathfrak{g}}\bigcap\underline{\mathfrak{g}}^{\perp_b}$. Since $\underline{\mathfrak{g}} \cap \underline{\mathfrak{g}}^{\perp_b} = \left\{ 0 \right\}$ on $Y$, we deduce that $v=0$ and we get a contradiction. The claim follows. From the claim, we can use linear transformation in $(x_{d+1},\ldots,x_{2d})$ (the linear transform depends smoothly on $\Td x''$) such that on $Y$, 
\[
J(\frac{\pr}{\pr x_j})=b_{j,1}(\Td x'')\frac{\pr}{\pr x_1}+\cdots+b_{j,d}(\Td x'')\frac{\pr}{\pr x_{d}}+\frac{\pr}{\pr x_{d+j}}+b_{j,2d+1}(\Td x'')\frac{\pr}{\pr x_{2d+1}}+\cdots+b_{j,2n+1}(\Td x'')\frac{\pr}{\pr x_{2n+1}},
\]
where $j=1,2,\ldots,d$. Consider the coordinates change:
\[
\renewcommand{\arraystretch}{1.2}
\begin{array}{ll}
&x=(x_1,\ldots,x_{2n+1})\To u=(u_1,\ldots,u_{2n+1}),\\
&(x_1,\ldots,x_{2n+1}) \To(x_1-\sum^d_{j=1}b_{j,1}(\Td x'')x_{d+j},\ldots,x_d-\sum^d_{j=1}b_{j,d}(\Td x'')x_{d+j},x_{d+1},\ldots,x_{2d},\\
&\quad\quad x_{2d+1}-\sum^d_{j=1}b_{j,2d+1}(\Td x'')x_{d+j},\ldots,x_{2n+1}-\sum^d_{j=1}b_{j,2n+1}(\Td x'')x_{d+j}).
\end{array}
\]
Then, 
\[
\renewcommand{\arraystretch}{1.2}
\begin{array}{ll}
&\frac{\pr}{\pr x_j}\rightarrow\frac{\pr}{\pr u_j},\ \ j=1,\ldots,d,2d+1,\ldots,2n+1,\\
&\frac{\pr}{\pr x_{d+j}}\rightarrow -b_{j,1}\frac{\pr}{\pr u_1}-\cdots-b_{j,d}\frac{\pr}{\pr u_d}+\frac{\pr}{\pr u_{d+j}}-b_{j,2d+1}\frac{\pr}{\pr u_{2d+1}}-\cdots-b_{j,2n+1}\frac{\pr}{\pr u_{2n+1}},\ \ j=1,\ldots,d.
\end{array}
\]
Hence, on $Y\bigcap U$, $J(\frac{\pr}{\pr x_j})\rightarrow\frac{\pr}{\pr u_{d+j}}$, $j=1,\ldots,d$. Thus, we can take $v=(v_1,\ldots,v_d)$ and $x=(x_1,\ldots,x_{2n+1})$ such that \eqref{e-gue161202}, \eqref{e-gue161103a}, \eqref{e-gue161102}, \eqref{e-gue161103I} hold and on $Y\bigcap U$, 
\[
J(\frac{\pr}{\pr x_j})=\frac{\pr}{\pr x_{d+j}},\ \ j=1,2,\ldots,d.
\]

Let $Z_j=\frac{1}{2}(\frac{\pr }{\pr x_j}-i\frac{\pr}{\pr x_{d+j}})(p)\in T^{1,0}_pX$, $j=1,\ldots,d$. From \eqref{e-gue161102}, we can check that
\[
(R^L_p-2t_0\mathcal{L}_p)(Z_j,\ol Z_k)=\lambda_j(t_0)\delta_{j,k}, \quad \langle\,Z_j\,|\,Z_k\,\rangle=\delta_{j,k}, \quad j, k = 1, \ldots, d.
\]
Since $\underline{\mathfrak{g}}_p$ is orthogonal to $H_pY\bigcap JH_pY$ and $H_pY\bigcap JH_pY\subset\underline{\mathfrak{g}}^{\perp_b}_p$, we can find an orthonormal frame $\set{Z_1,\ldots,Z_d,V_1,\ldots,V_{n-d}}$ for $T^{1,0}_pX$ such that $R^L_p-2t_0\mathcal{L}_p$ is diagonalized with respect to $Z_1,\ldots,Z_d,V_1,\ldots,V_{n-d}$, where $V_1\in\Complex H_pY\bigcap J\Complex H_pY ,\ldots,V_{n-d}\in\Complex H_pY\bigcap J\Complex H_pY$. Write 
\[{\rm Re\,}V_j=\sum^{2n}_{k=1}\alpha_{j,k}\frac{\pr}{\pr x_k},\ \ {\rm Im\,}V_j=\sum^{2n}_{k=1}\beta_{j,k}\frac{\pr}{\pr x_k},\ \ j=1,\ldots,n-d.\]
We claim that $\alpha_{j,k}=\beta_{j,k}=0$, for all $k=d+1,\ldots,2d$, $j=1,\ldots,n-d$. Fix $j=1,\ldots,n-d$. Since ${\rm Re\,}V_j\in\underline{\mathfrak{g}}^{\perp_b}_p$ and ${\rm span\,}\set{\frac{\pr}{\pr x_{d+1}},\ldots,\frac{\pr}{\pr x_{2d}}}\in\underline{\mathfrak{g}}_p^{\perp_b}$, we conclude that 
\begin{equation}\label{e-gue161218}
\sum^d_{k=1}\alpha_{j,k}\frac{\pr}{\pr x_k}+\sum^{2n}_{k=2d+1}\alpha_{j,k}\frac{\pr}{\pr x_k}\in\underline{\mathfrak{g}}^{\perp_b}_p\bigcap H_pY.
\end{equation} 
From \eqref{e-gue240720ycd} and \eqref{e-gue161218}, we deduce that 
\[\sum^d_{k=1}\alpha_{j,k}\frac{\pr}{\pr x_k}+\sum^{2n}_{k=2d+1}\alpha_{j,k}\frac{\pr}{\pr x_k}\in JH_pY\bigcap H_pY=\underline{\mathfrak{g}}^{\perp_b}_p\bigcap H_pY\]
and hence 
\begin{equation}\label{e-gue161218I}
J\Bigr(\sum^d_{k=1}\alpha_{j,k}\frac{\pr}{\pr x_k}+\sum^{2n}_{k=2d+1}\alpha_{j,k}\frac{\pr}{\pr x_k}\Bigr)\in\underline{\mathfrak{g}}^{\perp_b}_p\bigcap H_pY.
\end{equation}
From \eqref{e-gue161218I} and notice that $J({\rm Re\,}V_j)\in\underline{\mathfrak{g}}^{\perp_b}_p$, we deduce that 
\[J(\sum^{2d}_{k=d+1}\alpha_{j,k}\frac{\pr}{\pr x_k})\in\underline{\mathfrak{g}}_p\bigcap\underline{\mathfrak{g}}^{\perp_b}_p=\set{0}.\]
Thus, $\alpha_{j,k}=0$, for all $k=d+1,\ldots,2d$, $j=1,\ldots,n-d$.
Similarly, we can repeat the procedure above and deduce that $\beta_{j,k}=0$, for all $k=d+1,\ldots,2d$, $j=1,\ldots,n-d$. 

Since ${\rm span\,}\set{{\rm Re\,}V_j, {\rm Im\,}V_j;\, j=1,\ldots,n-d}$ is transversal to $\underline{\mathfrak{g}}_p\oplus J\underline{\mathfrak{g}}_p$, we can take linear transformation in $(x_{2d+1},\ldots,x_{2n})$ so that 
\[
\renewcommand{\arraystretch}{1.2}
\begin{array}{ll}
&{\rm Re\,}V_j=\alpha_{j,1}\frac{\pr}{\pr x_1}+\cdots+\alpha_{j,d}\frac{\pr}{\pr x_{d}}+\frac{\pr}{\pr x_{2j-1+2d}},\ \ j=1,2,\ldots,n-d,\\
&{\rm Im\,}V_j=\beta_{j,1}\frac{\pr}{\pr x_1}+\cdots+\beta_{j,d}\frac{\pr}{\pr x_{d}}+\frac{\pr}{\pr x_{2j+2d}},\ \ j=1,2,\ldots,n-d.
\end{array}
\]
Consider the coordinates change:
\[
\renewcommand{\arraystretch}{1.2}
\begin{array}{rcl}
x=(x_1,\ldots,x_{2n+1})&\To & u=(u_1,\ldots,u_{2n+1}),\\
(x_1,\ldots,x_{2n+1})& \To & (x_1-\sum^d_{j=1}\alpha_{j,1}x_{2j-1+2d}-\sum^{d}_{j=1}\beta_{j,1}x_{2j+2d},\ldots,x_d \\ 
& & -\sum^d_{j=1}\alpha_{j,d}x_{2j-1+2d}  -  \sum^{d}_{j=1}\beta_{j,d}x_{2j+2d}, x_{d+1},\ldots, x_{2n+1})
\end{array}
\]
Then, 
\[
\renewcommand{\arraystretch}{1.2}
\begin{array}{rcl}
\frac{\pr}{\pr x_j}& \rightarrow & \frac{\pr}{\pr u_j},\ \ j=1,\ldots,2d,\\
\frac{\pr}{\pr x_{2j-1+2d}}& \rightarrow & -\alpha_{j,1}\frac{\pr}{\pr u_1}-\cdots-\alpha_{j,d}\frac{\pr}{\pr u_d}+\frac{\pr}{\pr u_{2j-1+2d}},\ \ j=1,\ldots,n-d,\\
\frac{\pr}{\pr x_{2j+2d}}& \rightarrow &  -\beta_{j,1}\frac{\pr}{\pr u_1}-\cdots-\beta_{j,d}\frac{\pr}{\pr u_d}+\frac{\pr}{\pr u_{2j+2d}},\ \ j=1,\ldots,n-d.
\end{array}
\]
 Thus, we can take $v=(v_1,\ldots,v_d)$ and $x=(x_1,\ldots,x_{2n+1})$ such that \eqref{e-gue161202}, \eqref{e-gue161206} and \eqref{e-gue161202I} hold.  

Let $\Td x=(\Td x_1,\ldots,\Td x_{2n+1})$ be the coordinates as in Theorem~\ref{thm:hml}. It is easy to see that 
\begin{equation}\label{e-gue161219b}
\renewcommand{\arraystretch}{1.2}
\begin{array}{ll}
&\Td x_j=x_j+h_j(\mathring{x}),\ \ h_j(\mathring{x})=O(\abs{x}^2),\ \ j=1,2,\ldots,2n,\\
&\Td x_{2n+1}=x_{2n+1}+h_{2n+1}(\mathring{x}),\ \ h_{2n+1}(\mathring{x})=O(\abs{x}^2),
\end{array}
\end{equation}
where $\mathring{x}=(x_1,\ldots,x_{2n})$. 
We may change $x_{2n+1}$ be $x_{2n+1}+h_{2n+1}(0,\ldots,0,x_{d+1},\ldots,x_{2n})$ and we have 
\begin{equation}\label{e-gue161209}
\frac{\pr^2\Td x_{2n+1}}{\pr x_j\pr x_k}(p)=0,\ \ j, k=\set{d+1,\ldots,2n}.
\end{equation}
Note that when we change $x_{2n+1}$ to $x_{2n+1}+h_{2n+1}(0,\ldots,0,x_{d+1},\ldots,x_{2n})$, $\frac{\pr}{\pr x_j}$ will change to $\frac{\pr}{\pr x_j}+\alpha_j(x)\frac{\pr}{\pr x_{2n+1}}$, $j=d+1,\ldots,2n$, where $\alpha_j(x)$ is a smooth function on $Y\bigcap U$, independent of $x_1,\ldots,x_{d}$, $x_{2n+1}$ and $\alpha_j(0)=0$, $j=d+1,\ldots,2n$. Hence, on $Y\bigcap U$, we have $J(\frac{\pr}{\pr x_j})=\frac{\pr}{\pr x_{d+j}}+a_j(x)\frac{\pr}{\pr x_{2n+1}}$, $j=1,2,\ldots,d$,  where $a_j(x)$ is a smooth function on $\mu^{-1}(0)\bigcap U$, independent of $x_1,\ldots,x_{2d}$, $x_{2n+1}$ and $a_j(0)=0$, $j=1,\ldots,d$. 

From \eqref{e-gue240718yydII} and \eqref{e-gue161219b}, it is straightforward to see that 
\begin{equation}\label{e-gue161209I}
\begin{split}
\omega_0(\Td x)&=d\Td x_{2n+1}-i\sum^n_{j,t=1}\tau_{j,t}\ol{\Td z_t}d\Td z_j+i\sum^n_{j,t=1}\ol\tau_{j,t}\Td z_td\ol{\Td z_j}+O(\abs{\mathring{x}}^2)\\
&=dx_{2n+1}-i\sum^n_{j,t=1}\tau_{j,t}\ol z_tdz_j+i\sum^n_{j,t=1}\ol\tau_{j,t}z_td\ol{z_j}+\sum^n_{j=1}(\frac{\pr\Td x_{2n+1}}{\pr z_j}dz_j+\frac{\pr\Td x_{2n+1}}{\pr \ol z_j}d\ol z_j)+O(\abs{\mathring{x}}^2).
\end{split}
\end{equation}
Note that $\omega_0$ is $G$-invariant. From this observation and \eqref{e-gue161209I}, we deduce that 
\begin{equation}\label{e-gue240826yyd}
\begin{split}
&\frac{\pr^2\Td x_{2n+1}}{\pr z_j\pr x_\ell}(p)=i\tau_{j,\ell},\ \ j\in\set{1,\ldots,n}, \ell\in\set{1,\ldots,d},\\
&\frac{\pr^2\Td x_{2n+1}}{\pr\ol z_j\pr x_\ell}(p)=-i\ol{\tau_{j,\ell}},\ \ j\in\set{1,\ldots,n}, \ell\in\set{1,\ldots,d}.
\end{split}
\end{equation}
From \eqref{e-gue240826yyd}, we can computer that
\begin{equation}\label{e-gue240828yyd}
\begin{split}
&\frac{\pr^2\Td x_{2n+1}}{\pr x_j\pr x_\ell}(p)=i\tau_{j,\ell}-i\ol{\tau_{j,\ell}},\ \ j, \ell\in\set{1,\ldots,d},\\
&\frac{\pr^2\Td x_{2n+1}}{\pr x_{d+j}\pr x_\ell}(p)=-(\tau_{j,\ell}+\ol{\tau_{j,\ell}}),\ \ j, \ell\in\set{1,\ldots,d},\\
&\frac{\pr^2\Td x_{2n+1}}{\pr x_{2j-1}\pr x_\ell}(p)=i\tau_{j,\ell}-i\ol{\tau_{j,\ell}},\ \ j\in\set{d+1,\ldots,n},\ \ \ell\in\set{1,\ldots,d},\\
&\frac{\pr^2\Td x_{2n+1}}{\pr x_{2j}\pr x_\ell}(p)=-(\tau_{j,\ell}+\ol{\tau_{j,\ell}}),\ \ j\in\set{d+1,\ldots,n},\ \ \ell\in\set{1,\ldots,d}.
\end{split}
\end{equation}
From \eqref{e-gue161209} and \eqref{e-gue240828yyd}, we get \eqref{e-gue240826yydII}. 

\end{proof}

Let us now write the phase function \eqref{eq:phase} in local coordinates defined above.

\begin{prop}\label{p-gue240812yyd} 
Let $p\in\mu^{-1}(0)$ and fix $t_0\in I$. 
Let $s$ be a local rigid CR frame of $L$ defined on $D$, $|s|^2_{h^L}=e^{-2\Phi}$, $\Phi\in\mathcal{C}^\infty(D)$. Let $x=(x_1,\ldots,x_{2n+1})$ be local coordinates as in Proposition \ref{prop:coordinates}.  Then, there exists a neighborhood of $(p,\,p)$ such that 
\begin{equation} \label{eq:phasez}
\begin{split}
\varphi(x,y,t_0)=&t_0(x_{2n+1}-y_{2n+1})-\frac{i}{2}\sum_{j,l=1}^{n}(a_{l,j}+a_{j,l})(z_jz_l-w_jw_l)\\
&+\frac{i}{2}\sum_{j,l=1}^{n}(\overline{a}_{l,j}+\overline{a}_{j,l})(\overline{z}_j\overline{z}_l-\overline{w}_j\overline{w}_l) +\frac{i\,t_0}{2}\sum_{j,l=1}^{n}(\overline{\tau}_{l,j}-\tau_{j,l})({z}_j\overline{z}_l-{w}_j\overline{w}_l) \\
		&-\frac{i}{2}\sum_{j=1}^{n}\lambda_j(t_0)({z}_j\overline{w}_j-\overline{z}_j{w}_j)+\frac{i}{2}\sum_{j=1}^{n}\lambda_j(t_0)\lvert {z}_j-{w}_j\rvert^2\\
  &+i\frac{t_0}{2}\sum^d_{j,l=1}(\tau_{j,l}-\ol{\tau_{j,l}})(x_jx_l-y_jy_l)+t_0\sum^d_{j,l=1}(\tau_{j,l}+\ol{\tau_{j,l}})(-x_{d+j}x_l+y_{d+j}y_l)\\
  &+t_0\sum^n_{j=d+1}\sum^d_{l=1}i(\tau_{j,l}-\ol{\tau_{j,l}})(x_{2j-1}x_l-y_{2j-1}y_l)\\
&+t_0\sum^n_{j=d+1}\sum^d_{l=1}-(\tau_{j,l}+\ol{\tau_{j,l}})(x_{2j}x_l-y_{2j}y_l)+O(\lvert (z,w)\rvert^3),
	\end{split}
 \end{equation}
 where $\tau_{j,l}, a_{j,l}\in\mathbb C$, $j,l=1,\ldots,n$, $\lambda_j(t_0)$, $j=1,\ldots,n$, are as in Theorem~\ref{thm:hml}. 

 Moreover, we have 
 \begin{equation}\label{e-gue240830yyd}
\mu_{j,\ell}+t_0(\ol{\tau_{\ell,j}}+\tau_{j,\ell})=\delta_{j,\ell}\lambda_j(t_0),\ \ j,\ell=1,\ldots,n,
 \end{equation}
 and for all $j=1,\ldots,d$, $\ell=1,\ldots,n$
 \begin{equation}\label{e-gue240830yyda}
\begin{split}
\frac{1}{2}\mu_{j,\ell}+\ol{a_{j,\ell}}+\ol{a_{\ell,j}}=0,\ \ j=1,\ldots,d,\ \ \ell=1,\ldots,n,\\
\frac{1}{2}\mu_{\ell,j}+a_{j,\ell}+a_{\ell,j}=0,\ \ j=1,\ldots,d,\ \ \ell=1,\ldots,n
\end{split}
 \end{equation}
 where $\mu_{j,\ell}$, $a_{j,ell}$, are as in \eqref{e-gue240718yyda}

\end{prop}

\begin{proof}
From \eqref{eq:phase} and \eqref{e-gue240826yydII},  we get the \eqref{eq:phasez}.

From \eqref{e-gue240718yydII}, it is straightforward to check that 
\begin{equation}\label{e-gue240830yydI}
(R^L_p-2t_0\mathcal{L}_p)(\ol{W}_{j,t_0}(p),W_{\ell,t}(p))=\mu_{j,\ell}+t_0(\ol{\tau_{\ell,j}}+\tau_{j,\ell})=\lambda_j(t_0)\delta_{j,\ell},
\end{equation}
for all $j, \ell=1,\ldots,n$. From \eqref{e-gue240830yydI}, we get \eqref{e-gue240830yyd}. 



Since $\Phi$ is $G$-invariant, for every $j=1,\ldots,d$,  we have 
\begin{equation}\label{e-gue240901yyd}
\sum^n_{\ell=1}\frac{1}{2}\mu_{j,\ell}\ol z_\ell+\sum^n_{\ell=1}\frac{1}{2}\mu_{\ell,j}z_\ell+\sum^n_{\ell=1}a_{j,\ell}z_{\ell}+\sum^n_{\ell=1}a_{\ell,j}z_\ell+
\sum^n_{\ell=1}\ol{a_{j,\ell}}\,\ol z_\ell+\sum^n_{\ell=1}\ol{a_{\ell,j}}\,\ol z_\ell=0.
\end{equation}
From \eqref{e-gue240901yyd}, we get \eqref{e-gue240830yyda}. 
\end{proof}

\subsection{$G$-invariant Szeg\H{o} kernels asymptoics near $Y$}\label{s-gue240812yyde}

Let 
\[\mathcal{H}^0_{b,I}(X,L^k)^G:=\{u\in\mathcal{H}^0_{b,I}(X,L^k);\, g^*u=u,\ \ \mbox{for all $g\in G$}\}.\]
Let 
\[\Pi^G_{k,I}: L^2(X,L^k)\To\mathcal{H}^0_{b,I}(X,L^k)^G\]
be the orthogonal projection with respect to $(\,\cdot\,|\,\cdot\,)_k$. We consider the $G$-invariant weighted Fourier-Szeg\H{o} operator 
\begin{equation}\label{e-gue240812ycdeb}
P^G_{k,\tau^2}:=F_{k,\tau}\circ\Pi^G_{k,I}\circ F_{k,\tau}: L^2(X,L^k)\To\mathcal{H}^0_{b,I}(X,L^k),
\end{equation}
where $F_{k,\tau}$ is as in \eqref{e-gue150807I}. Let $P^G_{k,\tau^2}$ be as in \eqref{e-gue150807II}. The following was proved in~\cite[Theorem 5.5]{HLM22}. 

\begin{thm}\label{t-gue240813yyd}
Let $\chi, \hat\chi\in\mathcal{C}^\infty(X)$ with ${\rm supp\,}\chi\cap{\rm supp\,}\hat\chi=\emptyset$. Then, 
\begin{equation}\label{e-gue240813yyd}
\chi P_{k,\tau^2}\hat\chi=O(k^{-\infty})\ \ \mbox{on $X\times X$}.
\end{equation}
\end{thm}

Fix $p\in Y$ and let $s$ be a local CR rigid trivializing section of $L$ defined on an open set $D\subset X$. Let $d\mu(g)$ be the Haar measure on $G$ with $\int_Gd\mu(g)=1$. Let $V$ be an open neighborhood of $e_0\in G$ as in Proposition~\ref{prop:coordinates}. We have 
\[P^G_{k,\tau^2}(x,y)=\int_G\chi(g)P_{k,\tau^2}(x,g\cdot y)d\mu(g)+\int_G(1-\chi(g))P_{k,\tau^2}(x,g\cdot y)d\mu(g),\]
where $\chi\in\mathcal{C}^\infty_c(V)$, $\chi=1$ near $e_0$. Since $G$ is freely on $Y$, if $U$ and $V$ are small, there is a constant $c>0$ such that 
\begin{equation}\label{e-gue161231cr}
d(x,g\cdot y)\geq c,\ \ \forall x, y\in U, g\in{\rm Supp\,}(1-\chi),
\end{equation}
where $U\subset D$ is an open set of $p\in Y$ as in Proposition~\ref{prop:coordinates}. From now on, we take $U$ and $V$ small enough so that \eqref{e-gue161231cr} holds. In view of Theorem~\ref{t-gue240813yyd}, we see that $P_{k,\tau^2}(x,y)$ is $k$-negligible  away from diagonal. From this observation and \eqref{e-gue161231cr}, we conclude that $\int_G(1-\chi(g))P_{k,\tau^2}(x,g\cdot y)d\mu(g)=O(k^{-\infty})$ on $U\times U$ and hence 
\begin{equation}\label{e-gue170102}
\mbox{$P^G_{k,\tau^2}(x,y)=\int_G\chi(g)P_{k,\tau^2}(x,g\cdot y)d\mu(g)+O(k^{-\infty})$ on $U\times U$}.
\end{equation}
From Theorem~\ref{t-gue240717yyd} and \eqref{e-gue170102}, we get 
\[P^G_{k,\tau^2,s}(x,\,y)=\int_G\int_{\mathbb{R}}e^{ik\,\varphi(x,g\cdot y,t)}a(x,g\cdot y,t,k)\,\chi(g)\,\mathrm{d}td\mu(g)+ O(k^{-\infty}).\]

Recall that $\varphi$ has the form
\[\varphi(x,y,t)= (x_{2n+1}-y_{2n+1})t+\varphi_0(\mathring{x},\mathring{y},t), \]
where $\mathring{x}=(x_1,\ldots,x_{2n})$, $\varphi(\mathring{x},\mathring{y},t)\in\mathcal{C}^\infty(D\times D\times I)$.
Now, we use the coordinates as in Proposition~\ref{prop:coordinates}. Put $x'=(x_1,\ldots,x_d)$, $x''=(x_{d+1},\ldots,x_{2n+1})$, $\mathring{x}''=(x_{d+1},\ldots,x_{2n})$, $x=(x',\,x'')=(x',\,\hat{x}'',\tilde{x}'')$ where $\hat{x}''=(x_{d+1},\dots,x_{2d})$, and $\tilde{x}''=(x_{2d+1},\dots,x_{2n+1})$. Since $P^G_{k,\tau^2}(x,y)$ is $G$-invariant we have
\[P^G_{k,\tau^2,s}(x,\,y)=P^G_{k,\tau^2,s}((0,x''),\,(\Gamma(\mathring{y}''),y'')) \]
where $\Gamma$ is as in Proposition~\ref{prop:coordinates}.


Now, write $\mathring{x}''=(x_{d+1},\dots,x_{2n})$. Assume that on $V$, we have
\[d\mu(g)=m(v)\,\mathrm{d}v=m(v_1,\dots,\,v_d)\,\mathrm{d}v_1\cdots\mathrm{d}v_d\]
on $V$ and $m$ is a real-valued smooth function on $G$.
From Proposition~\ref{prop:coordinates}, we have
\begin{align*}P_{k,\tau^2,s}^G((0,x''),(\Gamma(\mathring{y}''),y''))=&\int_V\int_{\mathbb{R}} e^{ik\,\varphi((0,x''),(\Gamma(\mathring{y}'')+v,y''),t)}a((0,x''),(\Gamma(\mathring{y}'')+v,y''),t)\,\chi(v)m(v)\,\mathrm{d}t\, \mathrm{d}v \\
	&+ O(k^{-\infty}) 
\end{align*}
on $D\times D$, where $V$ is a small open neighborhood of the identity $e_0$ of $G$. Let $v=(v_1,\dots,\,v_d)$ be the local coordinates of $G$
defined in a neighborhood $V$ of $e$ with $v(e) =(0,\dots,\,0)$. From now on, we will identify the element $g\in V$ with $v(g)$.  By local coordinates of Proposition~\ref{prop:coordinates} and \eqref{eq:phasez}, it is easy to check that 
\begin{equation}\label{e-gue240814ycd}
\det\left(\left(\frac{\partial^2\varphi}{\partial v_\ell\partial v_j}(p,p,t_0)\right)_{j,\,k=1}^{d}\right)= i^d\,\lvert \lambda_1(t_0)\rvert \cdots \lvert\lambda_d(t_0)\rvert \neq 0. \end{equation}
We aim to apply the stationary phase formula of Melin and Sj\"ostrand, see \cite{ms}, but we first need to introduce some notations.

Let $W$ be an open set of $\mathbb R^N$, $N\in\mathbb N$. From now on, we write $W^\mathbb C$ to denote an open set in $\mathbb C^N$ with $W^\mathbb C\bigcap\mathbb R^N=W$ and for $f\in\mathcal{C}^\infty(W)$, from now on, we write $\Td f\in \mathcal{C}^\infty(W^\mathbb C)$ to denote an almost analytic extension of $f$.
For every $t\in I$, let $h(\mathring{x}'',\mathring{y}'',t)\in\mathcal{C}^{\infty}(U\times U,\,\mathbb{C}^d)$ be the solution of the system
	\begin{equation} \label{eq: varphihat}
		\frac{\partial \widetilde{\varphi}_0}{\partial{\tilde{y}}_j}((0,\mathring{x}''),(h(\mathring{x}'',\mathring{y}'',t)+\Gamma(\mathring{y}''),\,\mathring{y}''),\,t)=0 
	\end{equation}
for $j=1,2,\dots,d$. Let us set
\begin{equation} \label{eq:A}
A(x'',\,y'',\,t):=(x_{2n+1}-y_{2n+1})t+\widetilde{\varphi}_0((0,\mathring{x}''),(h(\mathring{x}'',\mathring{y}'',t)+\Gamma(\mathring{y}''),\,\mathring{y}''),\,t) \,,
\end{equation}
it is known that $\mathrm{Im}A \geq 0$, see \cite[p. 147]{ms}. Furthermore, we note that
\begin{align*} 
	&\frac{\partial \varphi_0}{\partial v_j}(x'=v+\Gamma(y'')=0,\,\hat{x}''=\hat{y}''=0,\,\widetilde{x}''=\widetilde{y}'',\,t)\\
 &= \left\langle2\,\mathrm{Im}\overline{\partial}_b\Phi(x)-t\,\omega_0(x),\,\frac{\partial}{\partial x_j}\right\rangle=0,
\end{align*}
for every $j=0,\dots,\,d$ where $x=(0,(0,\,\widetilde{x}''))$. Hence, for every $t\in I$ the critical points are $\hat{x}''=\hat{y}''=0,\,\widetilde{x}''=\widetilde{y}'',\,x'=v+\Gamma(y'')=0$ and we find that
\begin{equation}\label{e-gue240909ycd}
\begin{split}
&\mathrm{d}_xA(x''=y'',\,\hat{x}''=0,\,t)=-2\,\mathrm{Im}\overline{\partial}_b\Phi(x)+t\,\omega_0(x),\\
&\mathrm{d}_yA(x''=y'',\,\hat{x}''=0,\,t)=2\,\mathrm{Im}\overline{\partial}_b\Phi(x)-t\,\omega_0(x),\\
&A(x''=y'',\,\hat{x}''=0,\,t)=0.
\end{split}\end{equation}

Thus, we can now use the stationary phase formula of Melin and Sj\"ostrand, we can carry out the $v$ integral and get
\begin{equation}\label{e-gue240903yyd}P_{k,\tau^2,s}^G(x,y) \equiv  \int_{\mathbb{R}} e^{ikA(x'',y'',t)}b(x'',y'',t,k)dt + O(k^{-\infty})\ \ \mbox{on $U\times U$},\end{equation}
$p\in U\subset D$, $U$ is an open set as in Proposition~\ref{prop:coordinates}, 
\begin{equation}\label{e-gue150807bIz}
\begin{split}
&b(x'',y'',t,k)\sim\sum^\infty_{j=0}b_j(x'',y'',t)k^{n+1-\frac{d}{2}-j}\text{ in }S^{n+1-\frac{d}{2}}_{{\rm loc\,}}
(1;U\times U\times I), \\
&b_j(x'',y'',t)\in\mathcal{C}^\infty(U\times U\times I),\ \ j=0,1,\ldots,\\
&b(x'',y'',t,k)=b_j(x'',y'',t)=0\ \ \mbox{if $t\notin I$, $j=0,1,\ldots$},
\end{split}\end{equation}
and 
\begin{equation}\label{e-gue240814yyd}
b_0(p,p,t_0)=m(0)(2\pi)^{-n-1+\frac{d}{2}}\abs{\lambda_1(t_0)}^{\frac{1}{2}}\cdots\abs{\lambda_d(t_0)}^{\frac{1}{2}}\abs{\lambda_{d+1}(t_0)}\cdots\abs{\lambda_n(t_0)}\tau^2(t_0).
\end{equation} 

We now study the property of the phase $A(x'',y'',t)$. We need the following which is known (see Section 2 in~\cite{ms})

\begin{thm}\label{t-gue170105}
There exist a constant $c>0$ and an open set $\Omega\subset\mathbb R^d$, $0\in\Omega$, such that  
\begin{equation}\label{e-gue160105r}
{\rm Im\,}A(x'',y'',t)\geq c\inf_{v\in\Omega}\set{{\rm Im\,}\varphi_0((0,\mathring{x}''),(v+\Gamma(\mathring{y}''),\mathring{y}''),t)+\abs{d_v\varphi_0((0,\mathring{x}''),(v+\Gamma(\mathring{y}''),\mathring{y}''),t)}^2},
\end{equation}
for all $((0,x''),(0,y''),t)\in U\times U\times I$.
\end{thm}

We can now prove 

\begin{thm}\label{t-gue170105I}
If $U$ is small enough, then there is a constant $c>0$ such that 
\begin{equation}\label{e-gue170106m}
{\rm Im\,}A(x'',y'',t)\geq c\Bigr(\abs{\hat x''}^2+\abs{\hat y''}^2+\abs{\mathring{x}''-\mathring{y}''}^2\Bigr),\ \ \forall ((0,x''),(0,y''))\in U\times U.
\end{equation}
\end{thm}

\begin{proof}
From \eqref{e-gue150807b}, we see that there is a constant $c_1>0$ such that 
\begin{equation}\label{e-gue170106}
{\rm Im\,}\varphi_0((0,\mathring{x}''),(v+\Gamma(\mathring{y}''),\mathring{y}''),t)\geq c_1(\abs{v+\Gamma(\mathring{y}'')}^2+\abs{\mathring{x}''-\mathring{y}''}^2),\ \ \forall v\in\Omega,
\end{equation}
where $\Omega$ is any open set of $0\in\mathbb R^d$. From \eqref{e-gue160105r} and \eqref{e-gue170106}, we conclude that there is a constant $c_2>0$ such that 
\begin{equation}\label{e-gue170106I}
{\rm Im\,}A(x'',y'',t)\geq c_2(\abs{\mathring{x}''-\mathring{y}''}^2+\abs{d_{y'}\varphi_0((0,\mathring{x}''),(0,\mathring{x}''))}^2).
\end{equation}
From \eqref{e-gue240814ycd}, we see that the matrix 
\[\abs{{\rm det\,}\left(\frac{\pr^2\varphi_0}{\pr x_j\pr x_\ell}(p,p,t)+\frac{\pr^2\varphi_0}{\pr y_j\pr y_\ell}(p,p,t)\right)_{1\leq \ell\leq d, d+1\leq j\leq 2d}}\geq C,\]
for all $t\in I$, where $C>0$ is independent of $t$. From this observation and notice that $d_{y'}\varphi_0((0,\mathring{x}''),(0,\mathring{x}''),t)|_{\hat x''}=0$, we deduce that if $U$ is small enough then there is a constant 
$c_3>0$ such that 
\begin{equation}\label{e-gue170106II}
\abs{d_{y'}\varphi_0((0,\mathring{x}''),(0,\mathring{x}''),t)}\geq c_3\abs{\hat x''},
\end{equation}
for all $t\in I$. 
From \eqref{e-gue170106I} and \eqref{e-gue170106II}, the theorem follows. 
\end{proof}

Now, we determine the Hessian of $A(x'',y'',t)$ at $(p,p)$. Let \[\hat{h}(\mathring{x}'',\mathring{y}'',t):=h(\mathring{x}'',\mathring{y}'',t)+\Gamma(\mathring{y}'')=(\hat h_1(\mathring{x}'',\mathring{y}'',t),\ldots,\hat h_d(\mathring{x}'',\mathring{y}'',t)).\]  By \eqref{eq: varphihat}, we get
\begin{equation} \label{eq:hath} 
\begin{split}
&\frac{\partial^2\varphi_0}{\partial y_{s}\partial y_1}(p,p,t_0)+\sum_{j=1}^d\frac{\partial^2 \varphi_0}{\partial y_1\,\partial y_j}(p,p,t_0)\,\frac{\partial \hat{h}_j}{\partial y_{s}}(p,p,t_0)=0,\\
&\frac{\partial^2\varphi_0}{\partial x_{s}\partial y_1}(p,p,t_0)+\sum_{j=1}^d\frac{\partial^2 \varphi_0}{\partial y_1\,\partial y_j}(p,p,t_0)\,\frac{\partial \hat{h}_j}{\partial x_{s}}(p,p,t_0)=0,
\end{split}
\end{equation}
$s=d+1,\ldots,2n$.
Now, we would like to use \eqref{eq:hath} to get $\hat{h}$ up to second order. From \eqref{eq:phasez} and by explicit computation, we get 
\begin{equation}\label{e-gue240822yyd}
\frac{\partial^2 \varphi_0}{\partial y_1 \partial x_{s}}(p,p,t_0)=\lambda_1(t_0)\delta_{s,d+1},\end{equation}
$s=d+1,\ldots,2n$,
and 
\begin{equation}\label{e-gue240820yyd}
\begin{split}
&\frac{\partial^2 \varphi_0}{\partial y_1 \partial y_{j}}(p,p,t_0)=i(a_{1,j}+a_{j,1})-i(\ol{a_{1,j}}+\ol{a_{j,1}})+\delta_{j,1}i\lambda_1(t_0),\ \ j=1,\ldots,d,\\
&\frac{\partial^2 \varphi_0}{\partial y_1 \partial y_{d+s}}(p,p,t_0)=-(a_{1,s}+a_{s,1})-(\ol{a_{1,s}}+\ol{a_{s,1}})+
\frac{t_0}{2}(\tau_{s,1}+\ol{\tau_{s,1}})+
\frac{t_0}{2}(\tau_{1,s}+\ol{\tau_{1,s}}),\ \ s=1,\ldots,d,\\
&\frac{\partial^2 \varphi_0}{\partial y_1 \partial y_{2\ell-1}}(p,p,t_0)=i(a_{1,\ell}+a_{\ell,1})-i(\ol{a_{1,\ell}}+\ol{a_{\ell,1}})+
\frac{i}{2}t_0\ol{\tau_{\ell,1}}-
\frac{i}{2}t_0\tau_{\ell,1}+\frac{i}{2}t_0\tau_{1,\ell}-\frac{i}{2}t_0\ol{\tau_{1,\ell}},\ \ \ell=d+1,\ldots,n,\\
&\frac{\partial^2 \varphi_0}{\partial y_1 \partial y_{2\ell}}(p,p,t_0)=-(a_{1,\ell}+a_{\ell,1})-(\ol{a_{1,\ell}}+\ol{a_{\ell,1}})+
\frac{1}{2}t_0\ol{\tau_{\ell,1}}+
\frac{1}{2}t_0\tau_{1,\ell}+\frac{1}{2}t_0\ol{\tau_{1,\ell}}+\frac{1}{2}t_0\tau_{\ell,1},\ \ \ell=d+1,\ldots,n.
\end{split}
\end{equation}
From \eqref{e-gue240830yyd} and \eqref{e-gue240830yyda}, we have 

\begin{equation}\label{e-gue240901ycd}
\begin{split}
&a_{1,j}+a_{j,1}=-\frac{1}{2}\mu_{1,j}=-\frac{1}{2}\mu_{j,1}=-\frac{1}{2}\ol{\mu_{1,j}},\ \ 
 j=1,\ldots,d,\\
&a_{1,\ell}+a_{\ell,1}=-\frac{1}{2}\mu_{\ell,1}=\frac{t_0}{2}(\ol{\tau_{1,\ell}}+\tau_{\ell,1}),\ \ \ell=d+1,\ldots,n.
\end{split}
\end{equation}
From  \eqref{e-gue240830yyd}, \eqref{e-gue240830yyda}, \eqref{e-gue240820yyd} and \eqref{e-gue240901ycd}, it is straigtforward to see that 

\begin{equation}\label{e-gue240902yyd}
\begin{split}
&\frac{\partial^2 \varphi_0}{\partial y_1 \partial y_{j}}(p,p,t_0)=\delta_{j,1}i\lambda_1(t_0),\ \ j=1,\ldots,d,\\
&\frac{\partial^2 \varphi_0}{\partial y_1 \partial y_{d+s}}(p,p,t_0)=\lambda_1(t_0)\delta_{1,s},\ \ s=1,\ldots,d,\\
&\frac{\partial^2 \varphi_0}{\partial y_1 \partial y_{2\ell-1}}(p,p,t_0)=0,\ \ \ell=d+1,\ldots,n,\\
&\frac{\partial^2 \varphi_0}{\partial y_1 \partial y_{2\ell}}(p,p,t_0)=0,\ \ \ell=d+1,\ldots,n.
\end{split}
\end{equation}

From \eqref{eq:hath}, \eqref{e-gue240822yyd} and \eqref{e-gue240902yyd}, we get 
\begin{equation}\label{e-gue240822yydI}
\frac{\pr\hat h_1}{\pr y_s}(p,p,t_0)=\frac{\pr\hat h_1}{\pr x_s}(p,p,t_0)=i\delta_{s,d+1},\ \ s=d+1,\ldots,2n.
\end{equation}

In a similar way, we repeat the procedure above and get
\begin{equation}\label{e-gue240822yydIII}\frac{\partial\hat h_j}{\partial x_s}(p,p,t_0)=
\frac{\partial\hat h_j}{\partial y_s}(p,p,t_0) = i\delta_{s,d+j}, \end{equation}
$j=1,\ldots,d$, $s=1,\ldots,2n$.

By \eqref{eq:phasez}, \eqref{eq:A} and \eqref{e-gue240822yydIII}, it is straigtforward to check that 

\begin{thm}\label{t-gue240903yyd}
With the notations above, let $x=(x_1,\ldots,x_{2n+1})$ be the local coordinates as in Proposition~\ref{prop:coordinates}. Then for $A(x'',y'',t)\in\mathcal{C}^\infty(U\times U\times I)$ in \eqref{eq:A}, we have
\begin{equation}\label{eq:phaseA}
\begin{split}
	A(x'',y'',t_0) = & \,t_0(x_{2n-1}-y_{2n-1})+\frac{i}{2}\sum_{j=1}^{d}\lambda_j(t_0)(x_{d+j}^2+y_{d+j}^2)\\
 &+\frac{i}{2}\sum^d_{j,\ell=1}(a_{\ell,j}+a_{j,\ell})(x_{d+j}x_{d+\ell}-y_{d+j}y_{d+\ell})
 -\frac{i}{2}\sum^d_{j,\ell=1}(\ol{a_{\ell,j}}+\ol{a_{j,\ell}})(x_{d+j}x_{d+\ell}-y_{d+j}y_{d+\ell})\\
 &+\frac{i}{2}t_0\sum^d_{j,\ell=1}(\ol{\tau_{\ell,j}}-\tau_{j,\ell})(x_{d+j}x_{d+\ell}-y_{d+j}y_{d+\ell})\\
 &-\frac{i}{2}\sum^d_{j=1}\sum^n_{\ell=d+1}(a_{\ell,j}+a_{j,\ell})(ix_{d+j}z_\ell-iy_{d+j}w_\ell)-\frac{i}{2}\sum^n_{j=d+1}\sum^d_{\ell=1}(a_{\ell,j}+a_{j,\ell})(iz_jx_{d+j}-iw_jy_{d+\ell})\\
 &-\frac{i}{2}\sum^n_{j,\ell=d+1}(a_{\ell,j}+a_{j,\ell})(z_jz_\ell-w_jw_\ell)\\
 &+\frac{i}{2}\sum^d_{j=1}\sum^n_{\ell=d+1}(\ol{a_{\ell,j}}+\ol{a_{j,\ell}})(-ix_{d+j}\ol z_\ell+iy_{d+j}\ol w_\ell)+\frac{i}{2}\sum^n_{j=d+1}\sum^d_{\ell=1}(\ol{a_{\ell,j}}+\ol{a_{j,\ell}})(-i\ol z_jx_{d+j}+i\ol w_jy_{d+\ell})\\
 &+\frac{i}{2}\sum^n_{j,\ell=d+1}(\ol a_{\ell,j}+\ol a_{j,\ell})(\ol z_j\ol z_\ell-\ol w_j\ol w_\ell)\\
 &+\frac{it_0}{2}\sum^d_{j=1}\sum^n_{\ell=d+1}(\ol{\tau_{\ell,j}}-\tau_{j,\ell})(ix_{d+j}\ol z_\ell-iy_{d+j}\ol w_\ell)\\
 &+\frac{it_0}{2}\sum^n_{j=d+1}\sum^d_{\ell=1}(\ol{\tau_{\ell,j}}-\tau_{j,\ell})(-ix_{d+\ell}z_j+iy_{d+\ell}w_j)\\
 &+\frac{it_0}{2}\sum^n_{j,\ell=d+1}(\ol{\tau_{\ell,j}}-\tau_{j,\ell})(z_j\ol z_\ell-w_j\ol w_\ell)\\
 &-\frac{i}{2}\sum_{j=d+1}^{n-1}\lambda_j(t_0)({z}_j\overline{w}_j-\overline{z}_j{w}_j)\\
	&+\frac{i}{2}\sum_{j=d+1}^{n}\lambda_j(t_0)\lvert {z}_j-{w}_j\rvert^2+O(\lvert (\mathring{x}'',\mathring{y}'')\rvert^3),
\end{split}
\end{equation}
where $x''=(x_{d+1},\ldots,x_{2n+1})$, $\mathring{x}''=(x_{d+1},\ldots,x_{2n})$
\end{thm}

From \eqref{e-gue240903yyd} and Theorem~\ref{t-gue240903yyd}, we get \eqref{e-gue240903yydI}. 

Now, we prove \eqref{e-gue240903yydII}. We need 

\begin{lem}\label{l-gue170110}
Let $p\notin\mu^{-1}(0)$. Then, there are open sets $U$ of $p$ and $V$ of $e_0\in G$ such that for any $\chi\in \mathcal{C}^\infty_c(V)$, we have 
\begin{equation}\label{e-gue170110}
\int_GP_{k,\tau^2}(x,g\cdot y)\chi(g)d\mu(g)=O(k^{-\infty})\ \ \mbox{on $X\times U$}.
\end{equation}
\end{lem}

\begin{proof}
 Take local coordinates $v=(v_1,\ldots,v_d)$ of $G$ defined in a neighborhood $V$ of $e_0$ with $v(e_0)=(0,\ldots,0)$, local coordinates $x=(x_1,\ldots,x_{2n+1})$ of $X$ defined in a neighborhood $U=U_1\times U_2$ of $p$ with $0\leftrightarrow p$, where $U_1\subset\Real^d$ is an open set of $0\in\mathbb R^d$,  $U_2\subset\mathbb R^{2n+1-d}$ is an open set of $0\in\mathbb R^{2n+1-d}$, such that  
\[
\renewcommand{\arraystretch}{1.2}
\begin{array}{ll}
&(v_1,\ldots,v_d)\circ (\gamma(x_{d+1},\ldots,x_{2n}),x_{d+1},\ldots,x_{2n+1})\\
&=(v_1+\gamma_1(x_{d+1},\ldots,x_{2n}),\ldots,v_d+\gamma_d(x_{d+1},\ldots,x_{2n}),x_{d+1},\ldots,x_{2n+1}),\\
&\forall (v_1,\ldots,v_d)\in V,\ \ \forall (x_{d+1},\ldots,x_{2n+1})\in U_2,
\end{array}
\]
and
\[
\underline{\mathfrak{g}}={\rm span\,}\set{\frac{\pr}{\pr x_1},\ldots,\frac{\pr}{\pr x_d}},
\]
where  $\gamma=(\gamma_1,\ldots,\gamma_d)\in \mathcal{C}^\infty(U_2,U_1)$ with $\gamma(0)=0\in\mathbb R^d$. 

Let $s$ be a local CR rigid trivializing section of $L$ defined on an open set $U$ of $p$. We first assume that $x, y\in U$. 
From Theorem~\ref{t-gue240717yyd}, we have 
\[
\int_GP_{k,\tau^2,s}(x,g\cdot y)\chi(g)d\mu(g)\equiv\int e^{i(\varphi(x,(v+\gamma(y''),y''),t)k}a(x,(v+\gamma(y''),y''),t,k)\chi(v)m(v)dvdt,
\]
where $y''=(y_{d+1},\ldots,y_{2n+1})$, $m(v)dv=d\mu|_V$. Since $p\notin\mu^{-1}(0)$ and notice that $d_y\varphi(x,x,t)=2{\rm Im\,}\ddbar_b\Phi-t\omega_0(x)$, we deduce that if $V$ and $U$ are small then
$d_v(\varphi(x,(v+\gamma(y''),y''),t))\neq0$, for every $v\in V$, $(x,y)\in U\times U$. Hence, by using integration by parts with respect to $v$, we get
\begin{equation}\label{e-gue170110r}
 \int_GP_{k,\tau^2,s}(x,g\cdot y)\chi(g)d\mu(g)=O(k^{-\infty})\ \ \mbox{on $U$}.
 \end{equation}
From  \eqref{e-gue170110r}, we get 
\begin{equation}\label{e-gue170110z}
\int_GP_{k,\tau^2}(x,g\cdot y)\chi(g)d\mu(g)=O(k^{-\infty})\ \ \mbox{on $U\times U$}.
\end{equation}

Now fix $x_0\in X$, $x_0\notin U$. From Theorem~\ref{t-gue240813yyd}, we can check that 
\begin{equation}\label{e-gue170110zz}
\int_GP_{k,\tau^2}(x,g\cdot y)\chi(g)d\mu(g)=O(k^{-\infty})\ \ \mbox{on $W\times U$},
\end{equation}
where $W$ is a small open neighborhood of $x_0$. From \eqref{e-gue170110z} 
and \eqref{e-gue170110zz}, the lemma follows. 
\end{proof}

\begin{lem}\label{l-gue170111}
Let $p\notin\mu^{-1}(0)$ and let $h\in G$. We can find open sets $U$ of $p$ and $V$ of $h$ such that for every $\chi\in\mathcal{C}^\infty_c(V)$, we have 
$\int_GP_{k,\tau^2}(x,g\cdot y)\chi(g)d\mu(g)=O(k^{-\infty})\ \ \mbox{on $X\times U$}.$
\end{lem}

\begin{proof}
Let $U$ and $V$ be open sets as in Lemma~\ref{l-gue170110}. Let $\hat V=Vh$. Then, $\hat V$ is an open set of $h$. Let $\hat\chi\in\mathcal{C}^\infty_c(\hat V)$. We have 
\begin{equation}\label{e-gue170111}
\int_GP_{k,\tau^2}(x,g\cdot y)\hat\chi(g)d\mu(g)=\int_GP_{k,\tau^2}(x,g\cdot h\cdot y)\hat\chi(g\cdot h)d\mu(g)=\int_GP_{k,\tau^2}(x,g\cdot h\cdot y)\chi(g)d\mu(g),
\end{equation}
where $\chi(g):=\hat\chi(g\cdot h)\in\mathcal{C}^\infty_c(V)$. From \eqref{e-gue170111} and Lemma~\ref{l-gue170110}, we deduce that 
\[\int_GP_{k,\tau^2}(x,g\cdot y)\hat\chi(g)d\mu(g)=O(k^{-\infty})\ \ \mbox{on $X\times U$}.\]
The lemma follows. 
\end{proof} 

\begin{proof}[Proof of \eqref{e-gue240903yydII}]
Fix $p\notin\mu^{-1}(0)$. 
Let $h\in G$. By Lemma~\ref{l-gue170111}, we can find open sets $U_h$ of $p$ and $V_h$ of $h$ such that for every $\chi\in\mathcal{C}^\infty_c(V_h)$, we have 
\begin{equation}\label{e-gue170111c}
\int_GP_{k,\tau^2}(x,g\cdot y)\chi(g)d\mu(g)=O(k^{-\infty})\ \ \mbox{on $X\times U_h$}.
\end{equation}
Since $G$ is compact, we can find open sets $U_{h_j}$ and $V_{h_j}$, $j=1,\ldots,N$, such that $G=\bigcup^N_{j=1}V_{h_j}$. Let $U=\bigcap^N_{j=1}U_{h_j}$ and let $\chi_j\in\mathcal{C}^\infty_c(V_{h_j})$, $j=1,\ldots,N$, with $\sum^N_{j=1}\chi_j=1$ on $G$. From \eqref{e-gue170111c}, we have 
\[
P^G_{k,\tau^2}(x,y)=\int_GP_{k,\tau^2}(x,g\cdot y)d\mu(g)=\sum^N_{j=1}\int_GP_{k,\tau^2}(x,g\cdot y)\chi_j(g)d\mu(g)=O(k^{-\infty})\ \ \mbox{on $X\times U$}.
\]
\end{proof}

\begin{proof}[Proof of Theorem~\ref{thm:2}]
We now determine the leading term $b_0(p,p,t_0)$. In view of \eqref{e-gue240814yyd}, we only need to calculate $m(0)$.
Put $Y_p=\set{g\cdot p;\, g\in G}$. $Y_p$ is a $d$-dimensional submanifold of $X$. The $G$-invariant Hermitian metric $\langle\,\cdot\,|\,\cdot\,\rangle$ induces a volume form $dV_{Y_p}$ on $Y_p$. Put 
\[
V_{{\rm eff\,}}(p):=\int_{Y_p}dV_{Y_p}.
\]
For $f(g)\in\mathcal{C}^\infty(G)$, let $\hat f(g\cdot p):=f(g)$, $\forall g\in G$. Then, $\hat f\in \mathcal{C}^\infty(Y_p)$. Let $d\hat\mu$ be the measure on $G$ given by $\int_Gfd\hat\mu:=\int_{Y_p}\hat fdv_{Y_p}$, for all $f\in\mathcal{C}^\infty(G)$. It is not difficult to see that $d\hat\mu$ is a Haar measure and 
\begin{equation}\label{e-gue170108}
\int_Gd\hat\mu=V_{{\rm eff\,}}(p). 
\end{equation} 
In view of \eqref{e-gue161202I}, we see that
$\set{\frac{1}{\sqrt{2}}\frac{\pr}{\pr x_1},\ldots,\frac{1}{\sqrt{2}}\frac{\pr}{\pr x_d}}$ is an orthonormal basis for $\underline{\mathfrak{g}}_p$. From this observation and \eqref{e-gue170108}, we deduce that 
\begin{equation}\label{e-gue241025yyd}
m(0)=2^{\frac{d}{2}}\frac{1}{V_{{\rm eff\,}}(p)}.
\end{equation} 
From \eqref{e-gue240814yyd} and \eqref{e-gue241025yyd}, we get 
Theorem \ref{thm:2}. 
\end{proof}

\section{Proof of Theorem~\ref{thm:quantandred}}

\subsection{Preparation}\label{s-gue170226}

Fix $p\in Y$ and let $x=(x_1,\ldots,x_{2n+1})$ be the local coordinates as in Proposition~\ref{prop:coordinates} defined in an open set $U$ of $p$. We may assume that $U=\Omega_1\times\Omega_2\times\Omega_3\times\Omega_4$, where $\Omega_1\subset\mathbb R^d$, $\Omega_2\subset\mathbb R^d$ are open sets of $0\in\mathbb R^d$, $\Omega_3\subset\mathbb R^{2n-2d}$ is an open set of $0\in\mathbb R^{2n-2d}$ and $\Omega_4$ is an open set of $0\in\mathbb R$. From now on, we identify $\Omega_2$ with 
\[\set{(0,\ldots,0,x_{d+1},\ldots,x_{2d},0,\ldots,0)\in U;\, (x_{d+1},\ldots,x_{2d})\in\Omega_2},\] 
$\Omega_3$ with $\set{(0,\ldots,0,x_{2d+1},\ldots,x_{2n},0)\in U;\, (x_{d+1},\ldots,x_{2n})\in\Omega_3}$,  $\Omega_2\times\Omega_3$ with 
\[\set{(0,\ldots,0,x_{d+1},\ldots,x_{2n},0)\in U;\, (x_{d+1},\ldots,x_{2n})\in\Omega_2\times\Omega_3}.\] 
For $x=(x_1,\ldots,x_{2n+1})$, we write $x''=(x_{d+1},\ldots,x_{2n+1})$, $\mathring{x}''=(x_{d+1},\ldots,x_{2n})$,
$\hat x''=(x_{d+1},\ldots,x_{2d})$, 
\[\Td x''=(x_{2d+1},\ldots,x_{2n+1}),\ \ \Td{\mathring{x}}''=(x_{2d+1},\ldots,x_{2n}).\] 
From now on, we identify $x''$ with $(0,\ldots,0,x_{d+1},\ldots,x_{2n+1})\in U$, $\mathring{x}''=(x_{d+1},\ldots,x_{2n})$ with 
$(0,\ldots,0,x_{d+1},\ldots,x_{2n},0)\in U$, $\hat x''$ with 
$(0,\ldots,0,x_{d+1},\ldots,x_{2d},0,\ldots,0)\in U,$ \\
$\Td x''$ with $(0,\ldots,0,x_{2d+1},\ldots,x_{2n+1})\in U$, $\Td{\mathring{x}}''$ with $(0,\ldots,0,x_{2d+1},\ldots,x_{2n},0)$. Since $G$ acts freely on $Y$, we take $\Omega_2$ and $\Omega_3$ small enough so that if $x, x_1\in\Omega_2\times\Omega_3$ and $x\neq x_1$, then 
\begin{equation}\label{e-gue170227c}
g\cdot x\neq g_1\cdot x_1,\ \ \forall g, g_1\in G.
\end{equation} 

Let $A(x,y,t)\in\mathcal{C}^\infty(U\times U\times I)$ be as in Theorem~\ref{t-gue240903yyd}.
From $\ddbar_bP^G_{k,\tau^2}=0$, we can check that 
\begin{equation}\label{e-gue170225II}
\mbox{$\ddbar_bA(x,y,t)$ vanishes to infinite order at ${\rm diag\,}\Bigr((Y\bigcap U)\times(Y\bigcap U)\Bigr)$}.
\end{equation}
From \eqref{e-gue170225II} and notice that $\frac{\pr}{\pr x_j}+i\frac{\pr}{\pr x_{d+j}}\in T^{0,1}_xX$, $j=1,\ldots,d$, where $x\in Y$  and 
$\frac{\pr}{\pr x_j}A(x,y,t)=\frac{\pr}{\pr y_j}A(x,y,t)=0$, $j=1,\ldots,d$, we conclude that 
\[
\renewcommand{\arraystretch}{1.2}
\begin{array}{ll}
&\mbox{$\frac{\pr}{\pr x_{d+j}}A(x,y,t)|_{x_{d+1}=\cdots=x_{2d}=0}$ and $\frac{\pr}{\pr y_{d+j}}A(x,y,t)|_{y_{d+1}=\cdots=y_{2d}=0}$ vanish to infinite order at}\\
&{\rm diag\,}\Bigr(Y\bigcap U)\times Y\bigcap U)\Bigr).
\end{array}
\]
Let $G_j(x,y,t):=\frac{\pr}{\pr y_{d+j}}A(x,y,t)|_{y_{d+1}=\cdots=y_{2d}=0}$, $H_j(x,y,t):=\frac{\pr}{\pr x_{d+j}}A(x,y,t)|_{x_{d+1}=\cdots=x_{2d}=0}$, $j=1,\ldots,d$. Put
\[
A_1(x,y,t):=A(x,y,t)-\sum^d_{j=1}y_{d+j}G_j(x,y,t),\quad A_2(x,y,t):=A(x,y,t)-\sum^d_{j=1}x_{d+j}H_j(x,y,t).
\]
Then, for $j=1, 2, \ldots, d$,
\begin{equation}\label{e-gue170226Ip}
\frac{\pr}{\pr y_{d+j}}A_1(x,y,t)|_{y_{d+1}=\cdots=y_{2d}=0}=0 \quad \text{and} \quad \frac{\pr}{\pr x_{d+j}}A_2(x,y,t)|_{x_{d+1}=\cdots=x_{2d}=0}=0,
\end{equation}
and, for $j=1,2$,
\begin{equation}\label{e-gue170226II}
\mbox{$A(x,y,t)-A_j(x,y,t)$ vanishes to infinite order at ${\rm diag\,}\Bigr((Y\bigcap U)\times(Y\bigcap U)\Bigr)$}. 
\end{equation} 

We also write $u=(u_1,\ldots,u_{2n+1})$ to denote the local coordinates of $U$. For any smooth function $f\in\mathcal{C}^\infty(U)$, we write $\Td f\in\mathcal{C}^\infty(U^{\mathbb C})$ to denote an almost analytic extension of $f$, where $U^\mathbb C$ is an open set in $\mathbb C^{2n+1}$ with $U^\mathbb C\cap\mathbb R^{2n+1}=U$. We consider the following two systems 
\begin{equation}\label{e-gue170226III}
\begin{split}
&\frac{\pr\Td A_1}{\pr\Td u_{2d+j}}(\Td x,\Td{\Td u''},\Td t)+\frac{\pr\Td A_2}{\pr\Td x_{2d+j}}(\Td{\Td u''},\Td y,\Td s)=0,\ \ j=1,2,\ldots,2n-2d,\\
&\frac{\pr\Td A_2}{\pr\Td s}(\Td{\Td u''},\Td y,\Td s)=0,
\end{split}
\end{equation}
and 
\begin{equation}\label{e-gue170226IIIa}
\begin{split}
&\frac{\pr\Td A_1}{\pr\Td u_{d+j}}(\Td x,\Td{u''},\Td t)+\frac{\pr\Td A_2}{\pr\Td x_{d+j}}(\Td{u''},\Td y,\Td s)=0,\ \ j=1,2,\ldots,2n-d,\\
&\frac{\pr\Td A_2}{\pr\Td s}(\Td{u''},\Td y,\Td s)=0.
\end{split}
\end{equation}
where $\Td{\Td u''}=(0,\ldots,0,\Td u_{2d+1},\ldots,\Td u_{2n+1})$, $\Td{u''}=(0,\ldots,0,\Td u_{d+1},\ldots,\Td u_{2n+1})$. From \eqref{e-gue170226Ip}, we can take $\Td A_1$ and $\Td A_2$ so that for every $j=1,2,\ldots,d$, 
\begin{equation}\label{e-gue170226I}
\frac{\pr\Td A_1}{\pr\Td u_{d+j}}(\Td x,\Td{u''},t)=0 \quad \text{and} \quad
\frac{\pr\Td A_2}{\pr\Td x_{d+j}}(\Td{u''},\Td y,t)=0,\ \ \mbox{if $\Td u_{d+1}=\cdots=\Td u_{2d}=0$},
\end{equation}
and, for $j=1, 2$,
\begin{equation}\label{e-gue170227}
\Td A_j(\Td x, \Td y,t)=\Td x_{2n+1}-\Td y_{2n+1}+\Td{\hat A_j}(\Td{\mathring{x}''},\Td{\mathring{y}''},t),\ \ \Td{\hat A_j}\in\mathcal{C}^\infty(U^\mathbb C\times U^\mathbb C),
\end{equation}
where $\Td{\mathring{x}''}=(0,\ldots,0,\Td x_{d+1},\ldots,\Td x_{2n},0)$, $\Td{\mathring{y}''}=(0,\ldots,0,\Td y_{d+1},\ldots,\Td y_{2n},0)$. 

From \eqref{e-gue240909ycd}, 
it is not difficult to see that 
\[\begin{split}
&\frac{\pr\Td A_1}{\pr\Td u_{d+j}}(\Td x'',\Td x'',t)+\frac{\pr\Td A_2}{\pr\Td x_{d+j}}(\Td x'',\Td x'',t)=0,\ \ j=1,2,\ldots,2n-d,\\
&\frac{\pr\Td A_2}{\pr\Td s}(\Td x'',\Td x'',t)=0.
\end{split}\]
Hence, at $x'=0, \hat x''=0$, $(\Td{\Td{u}''},\Td s)=(\Td x'',t)$ and $(\Td{u''},\Td s)=(\Td x'',t)$
are real critical points of \eqref{e-gue170226III} and \eqref{e-gue170226IIIa} respectively. 
Let 
\[\begin{split}
&F(\Td x,\Td y,\Td{\Td {u}''},\Td s,\Td t):=
(\Td A_1(\Td x,\Td{\Td{u}''},\Td t)+\Td A_2(\Td{\Td{u}''},\Td y,\Td s), \frac{\pr\Td A_2}{\pr\Td s}(\Td{\Td{u}''},\Td y,\Td s)),\\
&\hat F(\Td x,\Td y,\Td{u}'',\Td s,\Td t):=
(\Td A_1(\Td x,\Td{u''},\Td t)+\Td A_2(\Td{u''},\Td y,\Td s), \frac{\pr\Td A_2}{\pr\Td s}(\Td{u''},\Td y,\Td s)).
\end{split}\]
Let ${\rm Hess\,}_{(\Td s,\Td{\Td{u}''})}F(\Td x,\Td y,\Td{\Td{u}''},\Td s,\Td t)$ denote the complex Hessian of $F$ with respect to $(\Td s,\Td{\Td{u}''})$ at $(\Td x,\Td y,\Td{\Td{u}''},\Td s,\Td t)$ and let ${\rm Hess\,}_{(\Td s,\Td{u}'')}\hat F(\Td x,\Td y,\Td{u}'',\Td s,\Td t)$ denote the complex Hessian of $\hat F$ with respect to $(\Td s,\Td{u}'')$ at $(\Td x,\Td y,\Td{u}'',\Td s,\Td t)$. 
We can check that  the matrices 
\[{\rm Hess\,}_{(\Td s,\Td{\Td{u}''})}F(\Td x,\Td y,\Td{\Td{u}''},\Td s,\Td t)|_{\Td x=\Td y=\Td x'', \Td{\Td{u}''}=\Td x'', \Td s=t},\ \ {\rm Hess\,}_{(\Td s,\Td{u}'')}\hat F(\Td x,\Td y,\Td{u}'',\Td s,\Td t)|_{\Td x=\Td y=\Td x'', \Td{u}''=\Td x'', \Td s=t}\]
are non-singular, for every $t\in I$. Moreover, fix $t_0\in I$. From Theorem~\ref{t-gue240903yyd}, 
it is straightforward  to see that 
\begin{equation}\label{e-gue240912yyd}
\begin{split}
&\det{\rm Hess\,}_{(\Td s,\Td{\Td{u}''})}F(\Td x,\Td y,\Td{\Td{u}''},\Td s,\Td t)|_{\Td x=\Td y=p, \Td{\Td{u}''}=p, \Td s=t_0}=(-1)(2i\abs{\lambda_{d+1}(t_0)}\cdots 2i\abs{\lambda_n(t_0)})^2,\\
&\det{\rm Hess\,}_{(\Td s,\Td{u}'')}\hat F(\Td x,\Td y,\Td{u}'',\Td s,\Td t)|_{\Td x=\Td y=p, \Td{u}''=p, \Td s=t_0}\\
&=(-1)(2i\abs{\lambda_1(t_0)}\cdots 2i\abs{\lambda_d(t_0)})(2i\abs{\lambda_{d+1}(t_0)}\cdots 2i\abs{\lambda_n(t_0)})^2.
\end{split}
\end{equation}
Hence, near $(p,p)$, we can solve \eqref{e-gue170226III} and \eqref{e-gue170226IIIa} and the solutions are unique. Let 
\[\begin{split}
&(\Td{\Td u''},\Td s)=(\alpha(\Td x,\Td y,\Td t),\gamma(\Td x,\Td y,\Td t)),\\
&\alpha(\Td x,\Td y,\Td t)=(\alpha_{2d+1}(\Td x,\Td y,\Td s),\ldots,\alpha_{2n}(\Td x,\Td y,\Td t)\in\mathcal{C}^\infty(U^\mathbb C\times U^\mathbb C\times I^\mathbb C,\mathbb C^{2n-2d}),\\ 
&\gamma(\Td x,\Td y,\Td t)\in\mathcal{C}^\infty(U^\mathbb C\times U^\mathbb C\times I^\mathbb C,\mathbb C),\end{split}\]
and 
\[\begin{split}
    &(\Td{u''},\Td s)=(\beta(\Td x,\Td y,\Td t),\delta(\Td x,\Td y,\Td t)),\\
    &\beta(\Td x,\Td y,\Td t)=(\beta_{d+1}(\Td x,\Td y,\Td s),\ldots,\beta_{2n}(\Td x,\Td y,\Td t)\in\mathcal{C}^\infty(U^\mathbb C\times U^\mathbb C\times I^\mathbb C,\mathbb C^{2n-d}),\\
    &\delta(\Td x,\Td y,\Td t)\in\mathcal{C}^\infty(U^\mathbb C\times U^\mathbb C\times I^\mathbb C,\mathbb C)\end{split}\]
be the solutions of \eqref{e-gue170226III} and \eqref{e-gue170226IIIa}, respectively. From \eqref{e-gue170226I}, it is easy to see that 
\begin{equation}\label{e-gue170226b}
\begin{split}
&\beta(x,y,t)=(\beta_{d+1}(x,y,t),\ldots,\beta_{2n}(x,y,t))=(0,\ldots,0,\alpha_{2d+1}(x,y,t),\ldots,\alpha_{2n}(x,y,t)),\\
&\delta(x,y,t)=\gamma(x,y,t).
\end{split}
\end{equation}
From \eqref{e-gue170226b}, we see that the value of $\Td A_1(x,\Td{\Td u''},\Td t)+\Td A_2(\Td{\Td u''},y,\Td s)$ at critical points $\Td{\Td u''}=\alpha(x,y,t)$, $\Td s=\gamma(x,y,t)$ is equal to the value of 
$\Td A_1(x,\Td{u''},t)+\Td A_2(\Td{u''},y,t)$ at critical points $\Td{u''}=\beta(x,y,t)$, $\Td s=\delta(x,y,t)$. Put 
\begin{equation}\label{e-gue170226bI}
\begin{split}
A_3(x,y,t):&=\Td A_1(x,\alpha(x,y,t),t)+\Td A_2(\alpha(x,y,t),y,\gamma(x,y,t))\\
&=\Td A_1(x,\beta(x,y,t),t)+\Td A_2(\beta(x,y,t),y,\delta(x,y,t)).
\end{split}
\end{equation}
$A_3(x,y,t)$ is a complex phase function. From \eqref{e-gue170227}, we have
\[
A_3(x,y,t)=x_{2n+1}-y_{2n+1}+\hat A_3(\mathring{x}'',\mathring{y}''),\ \ \hat A_3(\mathring{x}'',\mathring{y}'')\in\mathcal{C}^\infty(U\times U).
\]

\begin{defn}\label{d-gue240909yyd}
Let $\Phi_1, \Phi_2\in\mathcal{C}^\infty(U\times U\times I)$. 
Assume that $\Phi_1$ and $\Phi_2$ satisfy \eqref{e-gue240909ycd} and \eqref{e-gue170106m}. We say that $\Phi_1$ and $\Phi_2$ are equivalent on $U$ if for any $b_1(x,y,t,k)\in S^{n-\frac{d}{2}}_{{\rm loc\,},{\rm cl\,}}(U\times U\times I)$, ${\rm supp\,}_tb_1(x,y,t,k)\subset I$, we can find 
$b_2(x,y,t,k)\in S^{n-\frac{d}{2}}_{{\rm loc\,},{\rm cl\,}}(U\times U\times I)$, ${\rm supp\,}_tb_2(x,y,t,k)\subset I$, such that 
\[\int e^{ik\Phi_1(x,y,t}b_1(x,y,t,k)dt=\int e^{ik\Phi_2(x,y,t)}b_2(x,y,t,k)dt+O(k^{-\infty})\ \ \mbox{on $U\times U$}\]
and vise versa.
\end{defn} 

\begin{thm}\label{t-gue170226cw}
$A_1$ and $A_3$ are equivalent on $U$ in the sense of Definition~\ref{d-gue240909yyd}. 
\end{thm}

\begin{proof}
Let $s$ be a local rigid CR trivializing section of $L$ defiend on $U$. 
We consider the localized kernel of $P_{k,\tau^2}\circ P_{k,\tau^2}$ on $U$. Let $V\Subset U$ be an open set of $p$. Let $\chi(x'')\in\mathcal{C}^\infty_c(\Omega_2\times\Omega_3\times\Omega_4)$. From \eqref{e-gue170227c}, we can extend $\chi(x'')$ to 
$W:=\set{g\cdot x;\, g\in G, x\in\Omega_2\times\Omega_3\times\Omega_4}$
by $\chi(g\cdot x''):=\chi(x'')$, for every $g\in G$. Assume that  $\chi=1$ on some neighborhood of $V$. Let $\chi_1\in\mathcal{C}^\infty_c(U)$ with $\chi_1=1$ on some neighborhood 
of $V$ and ${\rm Supp\,}\chi_1\subset\set{x\in X;\, \chi(x)=1}$. We have 
\begin{equation}\label{e-gue170227b}
\chi_1P^G_{k,\tau^2}\circ P^G_{k,\tau^2}=\chi_1P^G_{k,\tau^2}\chi\circ P^G_{k,\tau^2}+\chi_1P^G_{k,\tau^2}(1-\chi)\circ P^G_{k,\tau^2}.
\end{equation}
Let's first consider $\chi_1P^G_{k,\tau^2}(1-\chi)\circ P^G_{k,\tau^2}$. We have 
\begin{equation}\label{e-gue170227bI}
(\chi_1P^G_{k,\tau^2}(1-\chi))(x,u)
=\chi_1(x)\int_GP_{k,\tau^2}(x,g\cdot u)(1-\chi(u))d\mu(g).
\end{equation} 
If $u\notin\set{x\in X;\, \chi(x)=1}$. Since ${\rm Supp\,}\chi_1\subset\set{x\in X;\, \chi(x)=1}$ and $\chi(x)=\chi(g\cdot x)$, for every $g\in G$, for every $x\in X$, we conclude that $g\cdot u\notin{\rm Supp\,}\chi_1$, for every $g\in G$. From this observation and notice that $P^G_{k,\tau^2}$ is $O(k^{-\infty})$ away from diagonal, we deduce that $\chi_1P^G_{k,\tau^2}(1-\chi)=O(k^{-\infty})$ and hence 
\begin{equation}\label{e-gue170227bII}
\chi_1P^G_{k,\tau^2}(1-\chi)\circ P^G_{k,\tau^2}=O(k^{-\infty}). 
\end{equation}
From \eqref{e-gue170227b} and \eqref{e-gue170227bII}, we get 
\begin{equation}\label{e-gue170227bIII}
\chi_1P^G_{k,\tau^2}\circ P^G_{k,\tau^2}=\chi_1P^G_{k,\tau^2}\chi\circ P^G_{k,\tau^2}+O(k^{-\infty}). 
\end{equation}
We can check that on $U$, 
\begin{equation}\label{e-gue170227y}
\renewcommand{\arraystretch}{1.2}
\begin{array}{ll}
&(\chi_1P^G_{k,\tau^2}\chi\circ P^G_{k,\tau^2})_s(x,y)\\
&=\int e^{ikA_1(x,u'',t)+ikA_2(u'',y,s)}\chi_1(x)g(x,\mathring{u}'',t,k)\chi(u'')g(u'',\mathring{y}'',s,k)dv(u'')ds+O(k^{-\infty}),
\end{array}
\end{equation}
where $(\chi_1P^G_{k,\tau^2}\chi\circ P^G_{k,\tau^2})_s(x,y)$ denotes the distribution kernel of the localization of $\chi_1P^G_{k,\tau^2}\chi\circ P^G_{k,\tau^2}$ with respect to $s$ (see the discussion after \eqref{sk}) and 
where $d\mu(g)dv(u'')=dV_X(x)$ on $U$. We use complex stationary phase formula of Melin-Sj\"ostrand to carry out the integral \eqref{e-gue170227y} and get 
\begin{equation}\label{e-gue170227yI}
\renewcommand{\arraystretch}{1.2}
\begin{array}{ll}
&(\chi_1P^G_{k,\tau^2}\chi\circ P^G_{k,\tau^2})(x,y)=\int e^{ikA_3(x,y,t)}a(x,y,t,k)dt+O(k^{-\infty})\ \ \mbox{on $U$}, \\
&a(x,y,t,k)\in S^{n+1-\frac{d}{2}}_{{\rm loc\,}}(1; U\times U\times I),\\
&\mbox{$a(x,y,t,k)\sim\sum^\infty_{j=0}k^{n+1-\frac{d}{2}-j}a_j(x,y,t)$ in $S^{n+1-\frac{d}{2}}_{{\rm loc\,}}(1; U\times U\times I)$},\\
&a_j(x,y,t)\in\mathcal{C}^\infty(U\times U\times I),\ \ j=0,1,2,\ldots,\\
&{\rm supp\,}_ta(x,y,t,k)\subset I,\ \ {\rm supp\,}_ta_j(x,y,t)\subset I,\ \ j=0,1,\ldots,\\
&a_0(x,x,t)\neq0,\ \ \mbox{for every $x\in Y\cap U$, $t\in\set{t\in I;\, \tau(t)\neq0}$}.
\end{array}
\end{equation}
From \eqref{e-gue170227bIII}, \eqref{e-gue170227yI} and notice that $(\chi_1P^G_{k,\tau^2}\circ P^G_{k,\tau^2})(x,y)=(\chi_1P^G_{k,\tau^2})(x,y)$, we deduce that 
\begin{equation}\label{e-gue170227yII}
\int e^{ikA_3(x,y,t)}a(x,y,t,k)dt=\int e^{ikA(x,y,t)}\chi_1(x)\hat b(x,y,t,k)dt+O(k^{-\infty})\ \ \mbox{on $U$},
\end{equation}
where $\hat b(x,y,t,k)\in S^{n+1-\frac{d}{2}}_{{\rm loc\,}}(1; U\times U\times I)$, $\hat b(x,y,t,k)\sim\sum^\infty_{j=0}k^{n+1-\frac{d}{2}-j}\hat b_j(x,y,t)$ in $S^{n+1-\frac{d}{2}}_{{\rm loc\,}}(1; U\times U\times I)$,
$\hat b_j(x,y,t)\in\mathcal{C}^\infty(U\times U\times I)$, $j=0,1,2,\ldots$, 
${\rm supp\,}_t\hat b(x,y,t,k)\subset I$, ${\rm supp\,}_t\hat b_j(x,y,t)\subset I$, $j=0,1,\ldots$, $\hat b_0(x,x,t)\neq0$, for every $x\in Y\cap U$, $t\in\set{t\in I;\, \tau(t)\neq0}$. 
Now, let $\alpha(x,y,t,k)\in S^{n+1-\frac{d}{2}}_{{\rm loc\,},{\rm cl\,}}(1; U\times U\times I)$. Without loss of generality, we assume that ${\rm supp\,}_t\alpha(x,y,t,k)\subset\set{t\in I;\, \tau(t)=1}$ and ${\rm supp\,}_{(x,y)}\alpha(x,y,t,k)\subset\set{(x,y)\in I;\, \chi_1(x)=\chi_1(y)=1}$. Let 
\[F_k(x,y):=\int e^{ikA(x,y,t)}\alpha(x,y,t,k)dt.\]
From complex stationary phase formula, we can find a classical pseudodifferential operator $E$ of order zero on $U$ such that 
\begin{equation}\label{e-gue240913yyd}
    E\circ\int e^{ikA(x,y,t)}\chi_1(x)\hat b(x,y,t,k)dt=F_k(x,y)+O(k^{-\infty}). 
    \end{equation}
From \eqref{e-gue170227yII}, \eqref{e-gue240913yyd} and by complex stationary phase formula, we get 
\begin{equation}\label{e-gue240913yydI}
\begin{split}
    &E\circ\int e^{ikA_3(x,y,t)}a(x,y,t,k)dt\\
   &=F_k(x,y)+O(k^{-\infty})\\
   &=\int e^{ikA_3(x,y,t)}\beta(x,y,t,k)dt+O(k^{-\infty}),
\end{split}    
    \end{equation}
    where $\beta(x,y,t,k)\in S^{n+1-\frac{d}{2}}_{{\rm loc\,},{\rm cl\,}}(1; U\times U\times I)$. From \eqref{e-gue240913yydI}, the theorem follows. 
\end{proof}

The following two theorems follow from \eqref{e-gue170226II}, \eqref{e-gue170226bI}, Theorem~\ref{t-gue170226cw}, complex stationary phase formula of Melin-Sj\"ostrand~\cite{ms} and some straightforward computation. We omit the details. 

\begin{thm}\label{t-gue170301w}
With the notations used above, let 
\[
\begin{split}
&F_k(x,y)=\int e^{ik A(x,y,t)}a(x,y,t,k)dt,\quad  G_k(x,y)=\int e^{ikA(x,y,t)}b(x,y,t,k)dt,\\
&a(x,y,t,k)\in S^{m}_{{\rm loc\,},{\rm cl\,}}(1; U\times U\times I),\ \quad b(x,y,t,k)\in S^{\ell}_{{\rm loc\,},{\rm cl\,}}(1; U\times U\times I),\\
&{\rm supp\,}_ta(x,y,t,k)\subset I,\ \ {\rm supp\,}_tb(x,y,t,k)\subset I.
\end{split}
\]
Let $\chi(x'')\in\mathcal{C}^\infty_c(\Omega_2\times\Omega_3\times\Omega_4)$. Then, we have 
\[\begin{split}
&\int F_k(x,u)\chi(u'')G_k(u,y)dV_X(u'')=\int e^{ikA(x,y,t)}c(x,y,t,k)dt+O(k^{-\infty}),\\
&c(x,y,t,k)\in S^{m+\ell-(n+1-\frac{d}{2})}_{{\rm loc\,},{\rm cl\,}}(1; U\times U\times I),\\
&c_0(x,x,t)=(2\pi)^{n-\frac{d}{2}+1}\abs{{\rm det\,}(R^L_x-2t\mathcal{L}_x)}^{-1}\abs{\det R_x(t)}^{\frac{1}{2}}a_0(x,x,t)b_0(x,x,t)\chi(x''),\ \ \forall x\in Y\bigcap U,
\end{split}
\]
where $\abs{\det R_x(t)}$ is in the discussion after \eqref{e-gue240913ycd} and 
$c_0$, $a_0$, $b_0$ denote the leading terms of $c$, $a$, $b$ respectively. 

Moreover, if there are $N_1, N_2\in\mathbb N$, such that $\abs{a_0(x,y,t)}\leq C\abs{(x,y)-(x_0,x_0)}^{N_1}$,  $\abs{b_0(x,y,t)}\leq C\abs{(x,y)-(x_0,x_0)}^{N_2}$, 
for all $x_0\in Y\bigcap U$, $t\in I$, where $C>0$ is a constant, then, 
\[
\abs{c_0(x,y,t)}\leq\hat C\abs{(x,y)-(x_0,x_0)}^{N_1+N_2}, 
\]
for all $x_0\in Y\bigcap U$, $t\in I$, where $\hat C>0$ is a constant.
\end{thm}

\begin{thm}\label{t-gue170301wI}
With the notations used above, let 
\[\begin{split}
&\mathcal{F}_k(x,\Td y'')=\int e^{ikA(x,\Td y'',t)}\alpha(x,\Td y'',t,k),\quad \mathcal{G}_k(\Td x'',y)=\int e^{ikA(\Td x'',y,t)}\beta(\Td x'',y,t,k)dt,\\
&\alpha(x,\Td y'',t,k)\in S^{m}_{{\rm loc\,},{\rm cl\,}}(1; U\times(\Omega_3\times\Omega_4)\times I),\quad \beta(\Td x'',y,t,k)\in S^{\ell}_{{\rm loc\,},{\rm cl\,}}(1; (\Omega_3\times\Omega_4)\times U\times I).
\end{split}
\]
Let $\chi_1(\Td x'')\in\mathcal{C}^\infty_c(\Omega_3\times\Omega_4)$. Then, we have 
\[
\renewcommand{\arraystretch}{1.2}
\begin{array}{ll}
&\int\mathcal{F}_k(x,\Td u'')\chi_1(\Td u'')\mathcal{G}_k(\Td u'',y)dV_X(\Td u)=e^{ikA(x,y,t)}\gamma(x,y,t,k)+O(k^{-\infty}),\\
&\gamma(x,y,t,k)\in S^{m+\ell-(n-d+1)}_{{\rm loc\,},{\rm cl\,}}(1; U\times U\times I),\\
&\gamma_0(x,x,t)=(2\pi)^{n-d+1}\abs{{\rm det\,}(R^L_x-2t\mathcal{L}_{x})}^{-1}\abs{\det R_x(t)}\alpha_0(x,\Td x'',t)\beta_0(\Td x'',x,t)\chi_1(\Td x''),\ \ \forall x\in Y\bigcap U,
\end{array}
\]
where $\abs{\det R_x(t)}$ is in the discussion after \eqref{e-gue240913ycd} and 
$\gamma_0$, $\alpha_0$, $\beta_0$ denote the leading term of $\gamma$, $\alpha$, $\beta$ respectively. 

Moreover, if there are $N_1, N_2\in\mathbb N$, such that $\abs{\alpha_0(x,\Td y'',t)}\leq C\abs{(x,\Td y'')-(x_0,x_0)}^{N_1}$,  $\abs{\beta_0(x,\Td y'',t)}\leq C\abs{(x,\Td y'')-(x_0,x_0)}^{N_2}$, 
for all $x_0\in Y\bigcap U$, $t\in I$, where $C>0$ is a constant, then, 
\[
\abs{\gamma_0(x,y,t)}\leq\hat C\abs{(x,y)-(x_0,x_0)}^{N_1+N_2}, 
\]
for all $x_0\in Y\bigcap U$, $t\in I$, where $\hat C>0$ is a constant.
\end{thm} 

\subsection{Theorem~\ref{thm:quantandred}}\label{s-gue170303} 

Let $P_{k,X_G,\tau^2}$ be the weighted Fourier-Szeg\H{o} operator on $X_G$ as in \eqref{e-gue150807II}. Fix $p\in Y$. Let $s$ be a local CR rigid $G$-invariant trivializing section of $L$ defined on a $G$-invariant open set $D$ of $p$. Let $x=(x_1,\ldots,x_{2n+1})$ be the local coordinates as in Proposition~\ref{prop:coordinates} defined in an open set $U$ of $p$, $U\subset D$. We will use the same notations as in Section~\ref{s-gue170226}. We will identify $\Omega_3\times\Omega_4$ with an open set in $X_G$ and we will identify $p$ as a point in $X_G$. Let $\varphi_{X_G}(\Td x'',\Td y'',t)$ be the phase as in Theorem~\ref{t-gue240717yyd}. We write $A(\Td x'',\Td y'',t):=A(x,y,t)|_{(\Omega_3\times\Omega_4)\times(\Omega_3\times\Omega_4)}$. It is not difficult to see that $\ddbar_{b,X_G}A(\Td x'',\Td y'',t)$ vanishes to infinite order at $\Td x''=\Td y''$ and $A(\Td x'',\Td y'',t)$ satisfies \eqref{eq:phase}. From this observation, we can repeat the process in~\cite[Section 3.7]{Hsiao08} with minor change and deduce that 

\begin{equation}\label{e-gue240930yyd}
\mbox{$\varphi_{X_G}(\Td x'',\Td y'',t)$ and $A(\Td x'',\Td y'',t)$ are equivariant on $\Omega_3\times\Omega_4$.}
\end{equation} 
Let $s$ be a local trivializing $G$-invariant 
CR rigid sections of $L|_{X_G}$ defined on $W:=\Omega_3\times\Omega_4$.
Thus, 
\begin{equation}\label{e-gue240930ycd}
\renewcommand{\arraystretch}{1.2}
\begin{array}{ll}
&P_{X_G,k,s}(\Td x'',\Td y'')=e^{ik A(\Td x'',\Td y'',t)}b(\Td x'',\Td y'',t,k)dt+O(k^{-\infty})\ \ \mbox{on $W$},\\
&\beta(\Td x'',\Td y'',t,k)\in S^{n+1-d}_{{\rm loc\,}}(1; W\times W\times I),\\
&\mbox{$\beta(\Td x'',\Td y'',t,k)\sim\sum^\infty_{j=0}k^{n+1-d-j}b_j(\Td x'',\Td y'',t)$ in $S^{n+1-d}_{{\rm loc\,}}(1; W\times W\times I)$},\\
&\beta_j(\Td x'',\Td y'',t)\in\mathcal{C}^\infty(W\times W\times I),\ \ j=0,1,2,\ldots,\\
&{\rm supp\,}_t\beta\subset I,\ \ {\rm supp\,}_t\beta_j\subset I,\ \ j=0,1,2,\ldots,\\
&\beta_0(\Td x'',\Td x'')=(2\pi)^{-(n-d)-1}\abs{\det(R^{L_{X_G}}_{\Td x''}-2t\mathcal{L}_{X_G,\Td x''})}\tau^2(t),\ \ \forall \Td x''\in W.
\end{array}
\end{equation}

Let 
\[
f(x)=\pi^{\frac{d}{4}}\sqrt{V_{{\rm eff\,}}(x)}\abs{\det\,R_x(t)}^{-\frac{1}{4}}\in\mathcal{C}^\infty(Y)^G.
\]
Recall that $R_x$ is given by \eqref{e-gue240913ycd}. We will identify $f$ with a smooth function on $X_G$, then $f\in\mathcal{C}^\infty(X_G)$. Let
\[\begin{split}
\sigma_k:\mathcal{C}^\infty(X,L^k)&\To H^0_b(X_G,L^k_G),\\
u&\To k^{-\frac{d}{4}}P_{k,X_G,\tau^2}\circ f\circ\gamma_G\circ P^G_{k,\tau^2}u,
\end{split}
\]
where $\gamma_G: \mathcal{ C}^\infty(X,L^k)^G\To\mathcal{C}^\infty(X_G)$ is the natural restriction. Let $R^{L_{X_G}}$ be the curvature of $L_{X_G}:=L|_{X_G}$ induced by $h^L$ and let $\mathcal{L}_{X_G}$ be the Levi form on $X_G$ induced by $\omega_{0,X_G}:=\omega_0|_{X_G}$. We can now prove 

\begin{thm}\label{t-gue170304ry}
With the notations used above, if $y\notin Y$, then for any open set $D$ of $y$ with $\ol D\bigcap Y=\emptyset$, we have
\begin{equation}\label{e-gue170304ryI}
\sigma_k=O(k^{-\infty})\ \ \mbox{on $X_G\times D$.}
\end{equation} 

Let $p\in Y$. Let $s$ be a local trivializing $G$-invariant 
CR rigid section of $L$ defined on an open set $D$ of $p$ in $X$. Let $x=(x_1,\ldots,x_{2n+1})$ be the local coordinates as in Proposition~\ref{prop:coordinates} defined 
in an open set $U$ of $p$, $U\subset D$. Let $\sigma_{k,s}$ be the localization of $\sigma_k$ with respect to $s$. Then, 
\begin{equation}\label{e-gue170304ryIII}
\renewcommand{\arraystretch}{1.2}
\begin{array}{ll}
&\sigma_{k,s}(\Td x'',y)=\int e^{ikA(\Td x'',y'',t)}\alpha(\Td x'',y'',k,t)+O(k^{-\infty})\ \ \mbox{on $W\times U$},\\
&\alpha(\Td x'',y'',t,k)\in S^{n-\frac{3}{4}d+1}_{{\rm loc\,}}(1; W\times U\times I),\\
&\mbox{$\alpha(\Td x'',y'',t,k)\sim\sum^\infty_{j=0}k^{n+1-\frac{3}{4}d-j}\alpha_j(\Td x'',y'',t)$ in $S^{n+1-\frac{3}{4}d}_{{\rm loc\,}}(1; W\times U\times I)$},\\
&\alpha_j(\Td x'',y'',t)\in\mathcal{C}^\infty(W\times U\times I),\ \ j=0,1,2,\ldots,\\
&{\rm supp\,}_t\alpha(\Td x'',y'',t,k)\subset I,\ \ {\rm supp\,}_t\alpha_j(\Td x'',y'',t,k)\subset I, \ \ j=0,1,\ldots,
\end{array}
\end{equation}
\begin{equation}\label{e-gue170304rc}
\alpha_0(\Td x'',\Td x'',t)=2^{-n-1+d}\pi^{\frac{3d}{4}-n-1}\frac{1}{\sqrt{V_{{\rm eff\,}}(\Td x'')}}
\abs{\det\,(R^{L_{X_G}}_{\Td x''}-2t\mathcal{L}_{\Td x''})}\abs{\det\,R_{\Td x''}}^{\frac{1}{4}}\tau^4(t),\ \ \forall \Td x''\in W,
\end{equation}
where $W=\Omega_3\times\Omega_4$, $\Omega_3$ and $\Omega_4$ are open sets as in the beginning of Section~\ref{s-gue170226}. 
\end{thm}

\begin{proof}
Note that $P^G_{k,\tau^2}=O(k^{-\infty})$ away $Y$. From this observation, we get \eqref{e-gue170304ryI}. 

Fix $u=(u_1,\ldots,u_{2n+1})\in Y\bigcap U$. From \eqref{e-gue170304ryI}, we only need to show that \eqref{e-gue170304ryIII} and \eqref{e-gue170304rc} hold near $u$ and  we may assume that $u=(0,\ldots,0,u_{2d+1},\ldots,u_{2n},u_{2n+1})=\Td u''$. Let $V$ be a small neighborhood of $u$. 
 Let $\chi(\Td u'')\in\mathcal{C}^\infty_c(\Omega_3\times\Omega_4)$. From \eqref{e-gue170227c}, we can extend $\chi(\Td x'')$ to 
\[Q=\set{g\cdot x;\, g\in G, x\in \Omega_3\times\Omega_4}\]
by $\chi(g\cdot\Td x''):=\chi(\Td x'')$, for every $g\in G$. Assume that $\chi=1$ on some neighborhood of $V$. Let $V_G:=V/G$ and let $\pi: V\To V_G$ be the natural projection. 
Let $\chi_1\in\mathcal{C}^\infty(X_G)$ with $\chi_1=1$ on some neighborhood 
of $V_G$ and ${\rm Supp\,}\chi_1\subset\set{\pi(x)\in Y_G;\, x\in Y, \chi(x)=1}$. We have 
\begin{equation}\label{e-gue170227bq}
\renewcommand{\arraystretch}{1.3}
\begin{array}{cl}
\chi_1\sigma_k&=k^{-\frac{d}{4}}\chi_1P_{k,X_G,\tau^2}\circ f\circ\gamma_G\circ P^G_{k,\tau^2}\\
&=k^{-\frac{d}{4}}\chi_1P_{k,X_G,\tau^2}\circ f\circ\gamma_G\circ\chi P^G_{k,\tau^2}\\
&\quad+k^{-\frac{d}{4}}\chi_1P_{k,X_G,\tau^2}\circ f\circ\gamma_G\circ(1-\chi)P^G_{k,\tau^2}.
\end{array}
\end{equation}
If $u\in Y$ but $u\notin\set{x\in X;\, \chi(x)=1}$. Since ${\rm Supp\,}\chi_1\subset\set{\pi(x)\in X;\, x\in Y, \chi(x)=1}$ and $\chi(x)=\chi(g\cdot x)$, for every $g\in G$, for every $x\in X$, we conclude that $\pi(u)\notin{\rm Supp\,}\chi_1$. From this observation, we get 
\begin{equation}\label{e-gue170227bIIq}
k^{-\frac{d}{4}}\chi_1P_{k,X_G,\tau^2}\circ f\circ\gamma_G\circ(1-\chi)P^G_{k,\tau^2}=O(k^{-\infty})\ \ \mbox{on $X_G\times X$}. 
\end{equation}
From \eqref{e-gue170227bq} and \eqref{e-gue170227bIIq}, we get 
\[
\chi_1\sigma_k=k^{-\frac{d}{4}}\chi_1P_{k,X_G,\tau}\circ f\circ\gamma_G\circ\chi P^G_{k,\tau}+O(k^{-\infty})\ \ \mbox{on $X_G\times X$}. 
\]
From Theorem~\ref{thm:Gszego} and \eqref{e-gue240930ycd}, we can check that on $U$, 
\begin{equation}\label{e-gue170227yq}
\chi_1\sigma_{k,s}(\Td x'',y)=\int e^{ikA(\Td x'',\Td v'',t)+ikA(\Td v'',y,s)}\chi_1(\Td x)\beta(\Td x'',\Td v'',k,t)\hat b(\Td v'',y,k,s)dV_{X_G}(\Td v'')ds+O(k^{-\infty}),
\end{equation}
where $\hat b(\Td v'',y,s,k)=\Bigr(f\circ\gamma_G\circ\chi(\Td v'')\circ g\Bigr)(\Td v'',y,s,k)$, $g$ is the symbol as in Theorem~\ref{thm:Gszego}. From \eqref{e-gue170227yq} and Theorem~\ref{t-gue170301wI}, we see that \eqref{e-gue170304ryIII} and \eqref{e-gue170304rc} hold near $u$. The theorem follows.    
\end{proof} 

Let  
\[
F_k:=\sigma_k^*\sigma_k: \mathcal{C}^\infty(X,L^k)\To \mathcal{H}^0_b(X,L^k)^G,\quad \hat F_k:=\sigma_k\sigma^*_k: \mathcal{C}^\infty(X_G)\To\mathcal{H}^{0}_{b}(X_G,L^k_{X_G}).
\]
From Theorem~\ref{t-gue170301w}, Theorem~\ref{t-gue170301wI}, we can repeat the proof of Theorem~\ref{t-gue170304ry} with minor change and deduce the following two theorems

\begin{thm}\label{t-gue170305a}
With the notations used above, if $y\notin Y$, then for any open set $D$ of $y$ with $\ol D\bigcap Y=\emptyset$, we have
$
F_k=O(k^{-\infty})\ \ \mbox{on $X\times D$.}
$
Let $p\in Y$. Let $s$ be a local trivializing $G$-invariant 
CR rigid sections of $L$ defined on an open sets $D$ of $p$ in $X$. Let $x=(x_1,\ldots,x_{2n+1})$ be the local coordinates as in Proposition~\ref{prop:coordinates} defined 
in an open set $U$ of $p$, $U\subset D$. Let $F_{k,s}$ be the localization of $\sigma_k$ with respect to $s$. Then, 
\[
\renewcommand{\arraystretch}{1.2}
\begin{array}{ll}
&F_k(x,y)=e^{ik A(x'',y'',t)}a(x'',y'',t,k)+O(k^{-\infty})\ \ \mbox{on $U\times U$},\\
&a(x'',y'',t,k)\in S^{n+1-\frac{d}{2}}_{{\rm loc\,}}(1; U\times U\times I),\\
&\mbox{$a(x'',y'',t,k)\sim\sum^\infty_{j=0}k^{n+1-\frac{d}{2}-j}a_j(\Td x'',y'',t)$ in $S^{n+1-\frac{d}{2}}_{{\rm loc\,}}(1; U\times U\times I)$},\\
&a_j(x'',y'',t)\in\mathcal{C}^\infty(U\times U\times I),\ \ j=0,1,2,\ldots,\\
&{\rm supp\,}_ta\subset I,\ \ {\rm supp\,}a_j\subset I,\ \ j=0,1,2,\ldots,
\end{array}
\]
and
\begin{equation}\label{e-gue170305bII}
a_0(\Td x'',\Td x'',t)=2^{-n-1+d}\,\frac{1}{V_{\mathrm{eff}}(\Td x'')}\,\lvert \det R_{\Td x''}(t)\rvert^{-1/2} \pi^{-n-1+d/2}\,\lvert \det(R^L_x-2t\mathcal{L}_{\Td x''})\rvert\tau^8(t),
\end{equation}
for all $\Td x''\in U$.
\end{thm} 

\begin{thm}\label{t-gue170305aI}
Let $p\in Y$. Let $s$ be a local trivializing $G$-invariant 
CR rigid section of $L$ defined on an open set $D$ of $p$ in $X$ and let $x=(x_1,\ldots,x_{2n+1})$ be the local coordinates as in Proposition~\ref{prop:coordinates} defined 
in an open set $U$ of $p$, $U\subset D$. Then, 
\[
\renewcommand{\arraystretch}{1.2}
\begin{array}{ll}
&\hat F_k(\Td x'',\Td y'')=\int e^{ikA(\Td x'',\Td y'',t)}\hat a(\Td x'',\Td y'',t,k)dt+O(k^{-\infty})\ \ \mbox{on $W\times W$},\\
&\hat a(\Td x'',\Td y'',t,k)\in S^{n+1-d}_{{\rm loc\,}}(1; W\times W\times I),\\
&\mbox{$\hat a(\Td x'',\Td y'',t,k)\sim\sum^\infty_{j=0}k^{n+1-d-j}\hat a_j(\Td x'',\Td y'',t)$ in $S^{n-d}_{{\rm loc\,}}(1; W\times W\times I)$},\\
&\hat a_j(\Td x'',\Td y'',t)\in\mathcal{C}^\infty(W\times W\times I),\ \ j=0,1,2,\ldots,\\
&{\rm supp\,}_t\hat a\subset I,\ \ {\rm supp\,}_t\hat a_j\subset I,\ \ j=0,1,2,\ldots,\\
&\hat a_0(\Td x'',\Td x'',t)=2^{-n+\frac{3}{2}d-1}\pi^{d-n-1}\lvert \det(R^{L_{X_G}}_{\Td x''}-2t\mathcal{L}_{X_G,\Td x''})\rvert
\tau^8(t),\ \ \forall \Td x''\in W,
\end{array}
\]
where $W=\Omega_3\times\Omega_4$, $\Omega_3$ and $\Omega_4$ are open sets as in the beginning of Section~\ref{s-gue170226}. 
\end{thm}

Let 
\begin{equation}\label{e-gue241015yyd}
R_k:=F_k-P^G_{k,\tau^8}:\mathcal{C}^\infty(X,L^k)\To\mathcal{H}^0_{b}(X,L^k)^G.\end{equation}
 Our next goal is to show that for $k$ large, $I+R_k:\mathcal{C}^\infty(X,L^k)\To\mathcal{C}^\infty(X,L^k)$ is injective. Let $\norm{\cdot}_k$ be the $L^2$ norm induced by $(\,\cdot\,|\,\cdot\,)_k$. 

 From Theorem~\ref{t-gue170305a}, we see that if $y\notin Y$, then for any open set $D$ of $y$ with $\ol D\bigcap Y=\emptyset$, we have
\begin{equation}\label{e-gue170305f}
R_k=O(k^{-\infty})\ \ \mbox{on $X\times D$.}
\end{equation}

Let $p\in Y$. Let $s$ be a local trivializing $G$-invariant 
CR rigid section of $L$ defined on an open set $D$ of $p$ in $X$ and let $x=(x_1,\ldots,x_{2n+1})$ be the local coordinates as in Proposition~\ref{prop:coordinates} defined 
in an open set $U$ of $p$, $U\subset D$. Then, 
\begin{equation}\label{e-gue170305fII}
\renewcommand{\arraystretch}{1.2}
\begin{array}{ll}
&R_k(x,y)=e^{ikA(x'',y'',t)}r(x'',y'',t,k)+O(k^{-\infty})\ \ \mbox{on $U\times U$},\\
&r(x'',y'',t,k)\in S^{n+1-\frac{d}{2}}_{{\rm loc\,}}(1; U\times U\times I),\\
&\mbox{$r(x'',y'',t,k)\sim\sum^\infty_{j=0}k^{n+1-\frac{d}{2}-j}r_j(x'',y'')$ in $S^{n-\frac{d}{2}}_{{\rm loc\,}}(1; U\times U\times I)$},\\
&r_j(x'',y'',t)\in\mathcal{C}^\infty(U\times U\times I),\ \ j=0,1,2,\ldots,\\
&{\rm supp\,}_tr\subset I,\ \ {\rm supp\,}_tr_j\subset I,\ \ j=0,1,\ldots.
\end{array}
\end{equation}
Moreover, from \eqref{e-gue241007yyd} and \eqref{e-gue170305bII}, it is easy to see 
\begin{equation}\label{e-gue170305fIII}
\abs{r_0(x,y)}\leq C\abs{(x,y)-(x_0,x_0)},
\end{equation}
for all $x_0\in Y\bigcap U$, where $C>0$ is a constant. We use $\norm{\cdot}$ to denote the standard $L^2$ norm on $X$ induced by the given volume form $dV_X$. We need 

\begin{lem}\label{l-gue170306s}
Let $p\in Y$. Let $x=(x_1,\ldots,x_{2n+1})$ be the local coordinates as in Proposition~\ref{prop:coordinates} defined 
in an open set $U$ of $p$, $U\subset D$.
Let
\[
\renewcommand{\arraystretch}{1.2}
\begin{array}{ll}
&H_k(x,y)=\int e^{ikA(x'',y'',t)}h(x,y,t,k)dt\ \ \mbox{on $U\times U$},\\
&h(x,y,t,k)\in S^{n-\frac{d}{2}}_{{\rm loc\,}}(1; U\times U\times I),\\
&\mbox{$h(x,y,t,k)\sim\sum^\infty_{j=0}k^{n-\frac{d}{2}-j}h_j( x,y,t)$ in $S^{n-\frac{d}{2}}_{{\rm loc\,}}(1; U\times U\times I)$},\\
&h(x,y,t,k)\in\mathcal{C}^\infty_c(U\times U\times I),\\
&h_j(x,y,t)\in\mathcal{C}^\infty_c(U\times U\times I),\ \ j=0,1,2,\ldots.
\end{array}
\]
Then, 
\begin{equation}\label{e-gue170306s}
\norm{H_ku}\leq \delta_k\norm{u},\ \ \forall u\in\mathcal{C}^\infty(X),\ \ \forall k\in\mathbb N,
\end{equation}
where $\delta_k$ is a sequence with $\lim_{k\To\infty}\delta_k=0$. 
\end{lem}

\begin{proof}
Fix $N\in\mathbb N$. It is not difficult to see that 
\begin{equation}\label{e-gue170306Ix}
\norm{H_ku}\leq\norm{(H^*_kH_k)^{2^N}u}^{\frac{1}{2^{N+1}}}\norm{u}^{1-\frac{1}{2^{N+1}}},\ \ \forall u\in\mathcal{C}^\infty(X),
\end{equation}
where $H^*_k$ denotes the adjoint of $H_k$ with respect to the given volume form $dV_X$. From Theorem~\ref{t-gue170301w}, we can repeat the proof of Theorem~\ref{t-gue170304ry} with minor change and deduce that 
\[
\renewcommand{\arraystretch}{1.2}
\begin{array}{ll}
&(H^*_kH_k)^{2^N}(x,y)=e^{ikA(x'',y'',t)}p (x,y,t,k)+O(k^{-\infty})\ \ \mbox{on $U\times U$},\\
&p(x,y,t,k)\in S^{n+1-2^{N+1}-\frac{d}{2}}_{{\rm loc\,}}(1; U\times U\times I),\\
&p(x,y,t,k)\in\mathcal{C}^\infty_0(U\times U\times I).
\end{array}
\]
Hence, 
\begin{equation}\label{e-gue170306xIII}
\abs{(H^*_kH_k)^{2^N}(x,y)}\leq \hat Ck^{n+1-2^{N+1}-\frac{d}{2}},\ \ \forall (x,y)\in U\times U,
\end{equation}
where $\hat C>0$ is a constant independent of $k$. Take $N$ large enough so that $n+1-2^{N+1}-\frac{d}{2}<0$. From \eqref{e-gue170306Ix} and \eqref{e-gue170306xIII}, 
we get \eqref{e-gue170306s}. 
\end{proof}

We also need 

\begin{lem}\label{l-gue170306}
Let $p\in Y$. Let $x=(x_1,\ldots,x_{2n+1})$ be the local coordinates as in Proposition~\ref{prop:coordinates} defined 
in an open set $U$ of $p$, $U\subset D$.
Let 
\[
\renewcommand{\arraystretch}{1.2}
\begin{array}{ll}
&B_k(x,y)=\int e^{ikA(x'',y'',t)}g(x,y,t,k)dt\ \ \mbox{on $U\times U$},\\
&g(x,y,t,k)\in S^{n+1-\frac{d}{2}}_{{\rm loc\,}}(1; U\times U\times I),\\
&\mbox{$g(x,y,t,k)\sim\sum^\infty_{j=0}k^{n+1-\frac{d}{2}-j}g_j(x,y,t)$ in $S^{n+1-\frac{d}{2}}_{{\rm loc\,}}(1; U\times U\times I)$},\\
&g_j(x,y,t)\in\mathcal{C}^\infty_c(U\times U\times I),\ \  j=0,1,2,\ldots,\\
&g(x,y,t,k)\in\mathcal{C}^\infty_c(U\times U\times I). 
\end{array}
\]
Suppose that 
\[
\abs{g_0(x,y)}\leq C\abs{(x,y)-(x_0,x_0)},
\]
for all $x_0\in Y\bigcap U$, where $C>0$ is a constant. Then, 
\begin{equation}\label{e-gue170306I}
\norm{B_ku}\leq \varepsilon_k\norm{u},\ \ \forall u\in\mathcal{C}^\infty(X),\ \ \forall k\in\mathbb N,
\end{equation}
where $\varepsilon_k$ is a sequence with $\lim_{k\To\infty}\varepsilon_k=0$. 
\end{lem}

\begin{proof}
Fix $N\in\mathbb N$. It is not difficult to see that 
\begin{equation}\label{e-gue170306Id}
\norm{B_ku}\leq\norm{(B^*_kB_k)^{2^N}u}^{\frac{1}{2^{N+1}}}\norm{u}^{1-\frac{1}{2^{N+1}}},\ \ \forall u\in\mathcal{C}^\infty(X),
\end{equation}
where $B^*_k$ denotes the adjoint of $B_k$ with respect to the volume form $dV_X$. From Theorem~\ref{t-gue170301w}, we can repeat the proof of Theorem~\ref{t-gue170304ry} with minor change and deduce that 
\[
\renewcommand{\arraystretch}{1.2}
\begin{array}{ll}
&(B^*_kB_k)^{2^N}(x,y)=\int e^{ikA(x'',y'',t)}\hat g(x,y,t,k)dt+O(k^{-\infty})\ \ \mbox{on $U\times U$},\\
&\hat g(x,y,t,k)\in S^{n+1-\frac{d}{2}}_{{\rm loc\,}}(1; U\times U\times I),\\
&\mbox{$\hat g(x,y,t,k)\sim\sum^\infty_{j=0}k^{n+1-\frac{d}{2}-j}\hat g_j(x,y,t)$ in $S^{n+1-\frac{d}{2}}_{{\rm loc\,}}(1; U\times U\times I)$},\\
&\hat g_j(x,y,t)\in\mathcal{C}^\infty_c(U\times U\times I),\ \ j=0,1,2,\ldots, \\
&\hat g(x,y,t,k)\in\mathcal{C}^\infty_c(U\times U\times I),
\end{array}
\]
and 
\begin{equation}\label{e-gue170306III}
\abs{\hat g_0(x,y,t)}\leq C\abs{(x,y)-(x_0,x_0)}^{2^{N+1}},
\end{equation}
for all $x_0\in Y\bigcap U$, where $C>0$ is a constant. 

Let 
\[\begin{split}
&(B^*_kB_k)^{2^N}_0(x,y)=\int e^{ik\Psi(x'',y'',t)}\hat g_0(x,y,t,k)dt,\\&(B^*_kB_k)^{2^N}_1(x,y)=\int e^{ik\Psi(x'',y'',t)}h(x,y,t,k)dt,
\end{split}\]
where $h(x,y,t,k)=\hat g(x,y,t,k)-\hat g_0(x,y,t,k)$. It is clear that 
$h(x,y,t,k)\in S^{n-\frac{d}{2}}_{{\rm loc\,}}(1; U\times U\times I).$
From Lemma~\ref{l-gue170306s}, we see that 
\begin{equation}\label{e-gue170306e}
\norm{(B^*_kB_k)^{2^N}_1u}\leq \delta_k\norm{u},\ \ \forall u\in\mathcal{C}^\infty(X),\ \ \forall k\in\mathbb N,
\end{equation}
where $\delta_k$ is a sequence with $\lim_{k\To\infty}\delta_k=0$. 

Since $\pr_t\pr_{x_{2n+1}}A|_{(x,y)\in Y\times U}\neq0$, by Malgrange preparation theorem, we have 
\[\pr_tA(x,y,t)=\alpha(x,y,t)(x_{2n+1}-\beta(x,y,t))\]
in $V\times V\times I$, where $V$ is a small open set of $p$, $\alpha, \beta\in\mathcal{C}^\infty(V\times V\times I)$.
We can consider Taylor expansion of $\Td g_0(x,y,t)$ at $x_{2n+1}=\beta(x,y,t)$ and by using integration by parts with respect to $t$, we may take $g_0$ so that 
\begin{equation}\label{e-gue241007ycd}
\mbox{$g_0$ is independent of $x_{2n+1}$}.
\end{equation}
From \eqref{e-gue170306III} and \eqref{e-gue241007ycd}, we see that 
\begin{equation}\label{e-gue170306r}
\abs{\hat g_0(x,y,t)}\leq C_1\Bigr(\abs{\hat x''}+\abs{\hat y''}+\abs{\Td{\mathring{x}}''-\Td{\mathring{y}}''}\Bigr)^{2^{N+1}},
\end{equation}
where $C_1>0$ is a constant. From \eqref{e-gue170106m}, we see that 
\begin{equation}\label{e-gue170306rI}
\abs{{\rm Im\,}A(x,y,t)}\geq c\Bigr(\abs{\hat x''}^2+\abs{\hat y''}^2+\abs{\Td{\mathring{x}}''-\Td{\mathring{y}}''}^2\Bigr),
\end{equation}
where $c>0$ is a constant. From \eqref{e-gue170306r} and \eqref{e-gue170306rI}, we conclude that 
\begin{equation}\label{e-gue170306rII}
\abs{(B^*_kB_k)^{2^N}_0(x,y)}\leq \hat Ck^{-2^N+n-\frac{d}{2}+1},\ \ \forall (x,y)\in U\times U,
\end{equation}
where $\hat C>0$ is a constant independent of $k$. From \eqref{e-gue170306rII}, we see that if $N$ large enough, then
\begin{equation}\label{e-gue170306ryIII}
\norm{(B^*_kB_k)^{2^N}_0u}\leq\hat\delta_k\norm{u},\ \ \forall u\in\mathcal{C}^\infty(X),\ \ \forall k\in\mathbb N,
\end{equation}
where $\hat\delta_k$ is a sequence with $\lim_{k\To\infty}\hat\delta_k=0$. 

From \eqref{e-gue170306Id}, \eqref{e-gue170306e} and \eqref{e-gue170306ryIII}, we get \eqref{e-gue170306I}. 
\end{proof}

From \eqref{e-gue170305f} and Lemma~\ref{l-gue170306}, we get 

\begin{thm}\label{t-gue241015yyd}
With the notations used above, we have
\begin{equation}\label{e-gue241015yydI}
\norm{R_ku}_k\leq \delta_k\norm{u}_k,\ \ \forall u\in\mathcal{C}^\infty(X,L^k),\ \ \forall k\in\mathbb N,
\end{equation}
where $R_k$ is as in \eqref{e-gue241015yyd} and $\delta_k$ is a sequence with $\lim_{k\To\infty}\delta_k=0$. 

In particular, for $k\gg1$, 
\begin{equation}\label{e-gue241015yydII}
\mbox{$I+R_k:
\mathcal{C}^\infty(X,L^k)\To\mathcal{C}^\infty(X,L^k)$ is injectiuve}.
\end{equation}
\end{thm}

\begin{proof}[Proof of Theorem~\ref{thm:quantandred}]
Fix $\lambda\in{\rm Spec\,}(-iT)$, $\tau\equiv1$ near $\lambda$. Let $u\in H^0_{b,\lambda}(X,L^k)^G$. 
If $\sigma_ku=0$. Then, $\sigma^*_k\sigma_ku=(P^G_{k,\tau^8}+R_k)u=(I+R_k)u=0$. From \eqref{e-gue241015yydII}, we get $u=0$ if $k\gg1$. Note that 
$\sigma_k$ maps the space $\mathcal{H}^0_{b,\lambda}(X,L^k)^G$ into $\mathcal{H}^0_{b,\lambda}(X_G,L^k_G)$. Thus, 
 $\sigma_k: \mathcal{H}^0_{b,\lambda}(X,L^k)^G\To\mathcal{H}^0_{b,\lambda}(X_G,L^k_G)$ is injective if $k\gg1$. 

 From Theorem~\ref{t-gue170305aI}, we have $\sigma_k\sigma^*_k=C_0(P_{k,X_G,\tau^8}+Q_k)$, where $Q_k$ is a semi-classical complex Fourier integral operator of the same type and order of $Q_k$ vanishes at the diagonal, where $C_0>0$ is a constant. We can repeat the proof of Theorem~\ref{t-gue241015yyd} with minor change and deduce that $I+Q_k: \mathcal{C}^\infty(X_G,L^k_G)\To\mathcal{C}^\infty(X_G,L^k_G)$ is injective, if $k\gg1$. Note that $({\rm Im\,}\sigma_k)^\perp\cap \mathcal{H}^0_{b,\lambda}(X_G,L^k)^G\subset {\rm Ker\,}\sigma^*_k\cap \mathcal{H}^0_{b,\lambda}(X_G,L^k)^G$. We conculde that 
$\sigma_k: \mathcal{H}^0_{b,\lambda}(X,L^k)^G\To\mathcal{H}^0_{b,\lambda}(X_G,L^k_G)$ is surjective if $k\gg1$. The theorem follows. 
\end{proof}

\bigskip
\textbf{Acknowledgements:} Andrea Galasso would like to express his thanks to the Mathematics Institute of Universität zu K\"oln for their hospitality throughout his scholarship: INdAM (Istituto Nazionale di Alta Matematica) foreign scholarship. Chin-Yu Hsiao was supported by Taiwan Ministry of Science and Technology projects 109-2923-M-001-010-MY4, 113-2115-M-002-011-MY3.

\end{document}